\documentclass[leqno,11pt]{article}%
\usepackage{amsfonts}
\usepackage{amsmath}
\usepackage{amssymb}
\usepackage{palatino}
\usepackage[colorlinks]{hyperref}
\usepackage{graphicx}%
\providecommand{\U}[1]{\protect\rule{.1in}{.1in}}
\hypersetup{colorlinks=true, urlcolor=red, linkcolor=blue, citecolor=blue}
\newtheorem{theorem}{Theorem}[section]
\newenvironment{acknowledgement}
{\par\medskip\noindent\textbf{Acknowledgement.}\ }
{\par\medskip}

\newtheorem{condition}[theorem]{Assumption}

\newtheorem{corollary}[theorem]{Corollary}

\newtheorem{definition}{Definition}[section]

\newtheorem{lemma}{Lemma}[section]

\newtheorem{problem}[theorem]{Problem}
\newtheorem{proposition}{Proposition}[section]
\newtheorem{remark}{Remark}[section]

\newenvironment{proof}[1][Proof]{\noindent\textbf{#1.} }{\ \rule{0.5em}{0.5em}}
\begin{document}

\title{Planar Obliquely Reflected BSVIs on General Filtered Spaces: Non-Symmetric
Rotation Fields and Associated Control Problems\thanks{E-mails:
grajdeanuandreea19@gmail.com (Andreea Negru\c{t}), aurel.rascanu@uaic.ro
(Aurel R\u{a}\c{s}canu), eduard.rotenstein@uaic.ro (Eduard Rotenstein)\newline%
$\sharp$~corresponding author}}
\author{Andreea Negru\c{t}$^{a}$, Aurel R\u{a}\c{s}canu$^{b,c}$, Eduard
Rotenstein$^{a,b,\sharp}$\medskip\\$^{a}${\small Simion Stoilow Institute of Mathematics of the Romanian Academy,
}\\{\small 21 Calea Grivi\c{t}ei, Bucharest, Rom\^{a}nia}\\$^{b}${\small Faculty of Mathematics, "Alexandru Ioan Cuza" University of
Ia\c{s}i, }\\{\small 9 Carol I Blvd., Ia\c{s}i,} {\small Rom\^{a}nia}\\$^{c}${\small Octav Mayer Institute of Mathematics of the Romanian Academy,
Ia\c{s}i branch, }\\{\small Bd. Carol I no. 8, Rom\^{a}nia}}
\date{}
\maketitle

\begin{abstract}
We prove existence and uniqueness of a c\`{a}dl\`{a}g solution to a planar
backward stochastic variational inequality on a general complete filtered
probability space, driven by a square integrable martingale, which may have
jumps. The multivalued term is the exterior normal cone operator of a bounded
uniformly convex planar domain, and the reflection direction is generated by a
time-dependent non-symmetric rotation field. The non-symmetry creates a
first-order tangential boundary term that destroys the standard monotonicity
and quadratic contraction estimates, used in the symmetric oblique-reflection
theory. We overcome this obstruction by constructing an explicit symmetric
two-point kernel with state dependent coefficients, whose boundary derivative
cancels the leading tangential contribution. A weighted martingale-exponential
estimate then controls the second order defect and the jump terms. Under
suitable geometric and quantitative compatibility constraints, we obtain the
unique strong c\`{a}dl\`{a}g solution. We also formulate associated control
problems for the rotation angle and prove existence of an optimal control in a
compact class of bounded-rate angle paths.

\end{abstract}


\textbf{MSC2020 Subject Classification: }60H10, 60H30, 49K45

\section{A motivating obstacle problem}

Backward stochastic differential equations (BSDEs, for short) and their
multivalued counterparts provide a flexible framework for studying nonlinear
stochastic dynamics, variational inequalities (BSVIs), obstacle problems,
partial differential equations, and stochastic control. In the classical
Brownian setting, the martingale component is usually represented through a
stochastic integral with respect to the Brownian motion. On a general filtered
probability space, however, such a representation need not be available. This
leads naturally to a formulation in which the martingale itself, rather than
an integrand $Z$, is regarded as one of the unknowns.

This point of view was developed for backward stochastic dynamics on general
complete filtered probability spaces, where neither It\^{o} integration with
respect to a prescribed Brownian motion nor a martingale representation
theorem is required. In this framework, the third variable of the generator is
replaced by a suitable functional transformation of the martingale component.
The resulting formulation is particularly well suited to c\`{a}dl\`{a}g
solutions and allows the driving noise to be an arbitrary square-integrable
martingale, possibly with jumps.

The purpose of the present paper is to study a class of planar backward
stochastic variational inequalities with oblique reflection. More precisely,
we consider equations of the form:%
\[
\left\{
\begin{array}
[c]{l}%
Y_{t}+%
{\displaystyle\int_{t}^{T}}
\Theta_{r}U_{r}dr=\eta+%
{\displaystyle\int_{t}^{T}}
F\left(  r,Y_{r},\mathcal{R}_{r}(M)\right)  dr-(M_{T}-M_{t}),\quad t\in\left[
0,T\right]  ,\text{ }\mathbb{P}\text{-a.s.,}\medskip\\
U_{r}\in\partial I_{K}(Y_{r}),\text{ \ }d\mathbb{P\otimes}dr\text{-a.e. on
}\Omega\times\left[  0,T\right]  ,
\end{array}
\right.
\]
where $K\subset\mathbb{R}^{2}$ is a bounded uniformly convex domain, $\partial
I_{K}$ denotes the subdifferential operator of its convexity indicator
function, $M$ is a square-integrable c\`{a}dl\`{a}g martingale, and $\Theta$
is a time-dependent rotation matrix. Thus, the process $Y$ is constrained to
remain in $K$, while the finite-variation correction acts along an oblique
direction determined by the rotation field. The c\`{a}dl\`{a}g framework is
essential in this setting. It permits the treatment of general martingale
noises with jumps and requires all estimates to be compatible simultaneously
with the continuous quadratic variation and with the discrete jump
contributions. The jump terms cannot be treated as a lower-order perturbation
since they enter directly into the It\^{o} formula, the quadratic variation,
the stability estimates, and the identification of the limiting variational
inequality. The main novelty, and at the same time the main difficulty, lies
in the fact that the oblique reflection field is not generated by a symmetric
positive definite matrix. Although the analysis is restricted to dimension
two, the restriction is structural, rather than merely technical. It is
possible to construct an explicit, symmetric, non-quadratic (from now on, it
should be understand that the Hessian isn't constant in the pair $\left(
x,y\right)  $, but the function is a quadratic form in $\left\vert
x-y\right\vert $) two-point test function that incorporates both the normal
and tangential components of the oblique direction. In the symmetric theory,
the reflection terms can be controlled by the monotonicity of the
subdifferential, as one can see in the seminal works on this topic provided by
Gassous, R\u{a}\c{s}canu, and Rotenstein \cite{Gassous/Rascanu/Rotenstein:12}
(forward case), or \cite{Gassous/Rascanu/Rotenstein:15} (backward case). The
symmetry of the perturbing term is crucial in their arguments. It provides a
genuine quadratic energy and prevents the appearance of a first-order
tangential contribution on the boundary of the domain. This mechanism breaks
down for a non-symmetric rotation, as we show below.

The two-point kernel used in this paper is motivated and inspired by the
classical boundary-adapted test functions, introduced in the study of problems
featuring some type of oblique reflection. In the viscosity-solution
framework, such functions are constructed from a symmetric positive definite
matrix field satisfying a geometric compatibility assumption (see Lions and
Sznitman \cite{Lions/Sznitman:84}, Barles \cite{Barles:99}, Barles and Da Lio
\cite{Barles/DaLio:06}). The probabilistic part of our argument is
particularly close in spirit to the weighted martingale-exponential method
developed by Chassagneux, Nadtochiy, and Richou \cite{Chassagneux/Richou:20},
\cite{Chassagneux/Nadtochiy/Richou:22}, for reflected BSDEs in non-convex
domains and for obliquely reflected BSDEs. Their approach is relevant here
because the second-order remainder generated by a boundary-adapted test
function is naturally integrated against a random quadratic-variation measure.
This is precisely the obstruction that arises in our c\`{a}dl\`{a}g setting.

Our contribution is an explicit planar realization, adapted to the rotation
field, symmetrized in the two variables and compatible with a c\`{a}dl\`{a}g
martingale noise driving the equation. The c\`{a}dl\`{a}g setup is motivated
by the need to make another improvement in the line of research, by going
beyond the symmetric setting and beyond Brownian filtrations, simultaneously.
The approach for the c\`{a}dl\`{a}g framework follows the one developed by
Liang, Lyons, and Qian \cite{Liang/Lyons/Qian:11} and used recently by
Bensoussan, Li, and Yam \cite{Bensoussan/Li/Yam:2018}. However,
\cite{Bensoussan/Li/Yam:2018} impose a quite strong assumption regarding the
boundedness of the multivalued operator, which was overcome by Negru\c{t},
R\u{a}\c{s}canu, and Rotenstein \cite{Negrut/Rascanu/Rotenstein:26}.

This planar construction has several decisive properties. First, the test
function we introduce is symmetric with respect to both components, which is
indispensable when the two arguments are two penalized solutions (obtained by
the classical Moreau-Yosida technique) and either one may reach the boundary.
Second, under a suitable angle restriction, this test function remains
uniformly equivalent to the square distance between its arguments. Third, its
derivative in the oblique reflection direction has the correct sign, up to a
quadratic error, which is controlled by the curvature of $bd\left(  K\right)
$. Finally, its Hessian matrix satisfies a lower bound, involving the
difference of the martingale increments, together with a defect proportional
to a suitable term, which we can control. This defect is the price we have to
pay for abandoning symmetry, and it remains present even though the test
function is explicitly constructed.

Additional substantial technical difficulties are brought by the
c\`{a}dl\`{a}g character of the driving martingale. Applying It\^{o}'s formula
to the test function $\Phi_{t}\left(  Y_{t}^{\varepsilon},Y_{t}^{\delta
}\right)  $ produces a second-order contribution, involving both the
continuous quadratic variations and the jumps of the two martingales. The
resulting error cannot, in general, be absorbed by a deterministic Gronwall
argument. Instead, it should naturally be integrated against the random
increasing measure $d\left[  M^{\varepsilon}\right]  _{t}+d\left[  M^{\delta
}\right]  _{t}$ generated by the quadratic variations of the approximating
martingales of the penalizing equations. In order to overcome this
obstruction, we use a weighted estimate whose exponential weight is
constructed from the quadratic variations of the approximating martingales,
together with the absolutely continuous terms generated by the Lipschitz
coefficients and the time dependence of the rotation field. This weighted
martingale-exponential argument cancels the random-measure contribution
pathwise before expectations are taken. The construction of the weight
requires uniform conditional estimates for the martingale quadratic
variations. These are obtained through conditional bounds for the generator,
the penalization terms, and the martingale functional.

Our main result establishes existence and uniqueness of a c\`{a}dl\`{a}g
solution to the obliquely reflected planar BSVI, under suitable assumptions on
the structural data: the generator, the martingale functional, the uniformly
convex domain, and the rotation angle. The assumptions on the rotation angle
have some geometric interpretations. One condition ensures the ellipticity of
the two-point kernel $\Phi_{t}$, while another expresses the compatibility
between the tangential displacement generated by the oblique field and the
curvature of the boundary, quantified through the uniform convexity constants
of $I_{K}$. A further smallness condition is needed to absorb the second-order
defect in the weighted c\`{a}dl\`{a}g estimate.

The paper is organized as follows. \textit{Section 2} introduces the general
filtered probability space framework, the c\`{a}dl\`{a}g process spaces, the
martingale functional $\mathcal{R}$, and the assumptions on the generator.
\textit{Section 3} recalls the Moreau-Yosida penalization procedure and
establishes the basic uniform estimates for the approximating equations.
\textit{Section 4} develops the planar geometry of the rotation field,
constructs the non-quadratic, two-point kernel, proves the cancellation
estimate for the tangential boundary term, and establishes the weighted
convergence of the penalized family. The existence and uniqueness of the
c\`{a}dl\`{a}g obliquely reflected solution are then obtained. In
\textit{Section 5} we formulate some associated control problems for the
rotation angle and prove existence of an optimal control in a compact class of
bounded-rate angle paths. The final section, \textit{Annex}, collects the
c\`{a}dl\`{a}g stochastic calculus and convex analytic tools used throughout
the proofs.

\section{Preliminaries\label{Working setup}}

Throughout the paper we work on a complete filtered probability space $\left(
\Omega,\mathcal{F},\mathbb{F}=\{\mathcal{F}_{t}\}_{t\geq0},\mathbb{P}\right)
$ satisfying the usual hypotheses, and on a fixed time interval $\left[
0,T\right]  $, $T>0$. We shall use the following spaces of stochastic processes:

\begin{enumerate}
\item[$\left(  a\right)  $] $\mathbb{L}_{m\times d}^{p},$ $p\geq0,$
$m,d\in\mathbb{N}^{\ast}$, is the (non-separable) complete metric space of
adapted c\`{a}gl\`{a}d processes $G:\Omega\times\left[  0,T\right]
\rightarrow\mathbb{R}^{m\times d}$, and $\mathbb{D}_{d}^{p}$ is the complete
metric space of adapted c\`{a}dl\`{a}g processes $X:\Omega\times\left[
0,T\right]  \rightarrow\mathbb{R}^{d}$. In both cases the metric is defined by%
\[
\rho\left(  X,Y\right)  =\left\{
\begin{array}
[c]{ll}%
\left(  \mathbb{E}\sup\limits_{t\in\left[  0,T\right]  }\left\vert X_{t}%
-Y_{t}\right\vert ^{p}\right)  ^{1\wedge\left(  1/p\right)  }\quad & \text{if
}p>0,\medskip\\
\mathbb{E}\left(  1\wedge\sup\limits_{t\in\left[  0,T\right]  }\left\vert
X_{t}-Y_{t}\right\vert \right)  \quad & \text{if }p=0.
\end{array}
\right.
\]
If $p\geq1$, then the spaces are Banach spaces with the norm $\left\Vert
X\right\Vert =$ $\rho\left(  X,0\right)  .$ In the case $\mathbb{D}_{d}^{2}$
we denote%
\[
\left\vert \left\vert \left\vert X\right\vert \right\vert \right\vert
_{T}=\left\vert \left\vert \left\vert X\right\vert \right\vert \right\vert
_{\left[  0,T\right]  }:=\mathbb{E}\sup_{t\in\left[  0,T\right]  }\left\vert
X_{t}\right\vert ^{2}<\infty.
\]

\item[$\left(  b\right)  $] $\mathcal{M}_{d}^{2}\subset\mathbb{D}_{d}^{2}$ is
the Hilbert space of stochastic processes $M:\Omega\times\left[  0,T\right]
\rightarrow\mathbb{R}^{d},$ $M_{0}=0,$ which are c\`{a}dl\`{a}g square
integrable martingales on $\left[  0,T\right]  $, endowed with the inner
product and the corresponding norm
\[
\left\langle M,N\right\rangle _{\mathcal{M}}=\mathbb{E}\left\langle
M_{T},N_{T}\right\rangle \quad\text{and}\quad\left\Vert M\right\Vert
_{\mathcal{M}}=\sqrt{\mathbb{E}\left\vert M_{T}\right\vert ^{2}}.
\]
Since the filtration $\mathbb{F}=\{\mathcal{F}_{t}\}_{t\geq0}$ satisfies the
usual hypotheses, every martingale admits a unique c\`{a}dl\`{a}g modification
(see Protter \cite[Chapter I, Theorem 9]{Protter:05}). In what follows, each
martingale is identified with its c\`{a}dl\`{a}g version. The space
$\mathcal{M}_{d}^{2}$ is a closed linear subspace of the Banach space
$\mathbb{D}_{d}^{2}$ since, by Doob's maximal $L^{2}$-inequality and the fact
that $M_{0}=0$, the norms$\ \left\Vert \cdot\right\Vert _{2}$ and $\left\vert
\left\vert \left\vert \cdot\right\vert \right\vert \right\vert _{\left[
0,T\right]  }$ are equivalent:%
\begin{equation}
\mathbb{E}\left\vert M_{t}\right\vert ^{2}\leq\mathbb{E}\sup_{r\in\left[
0,t\right]  }\left\vert M_{r}\right\vert ^{2}\leq4\mathbb{E}\left\vert
M_{t}\right\vert ^{2},\quad\forall t\in\left[  0,T\right]  .
\label{bdg adapted}%
\end{equation}
In particular, $\left\Vert M\right\Vert _{2}\leq\left\vert \left\vert
\left\vert M\right\vert \right\vert \right\vert _{\left[  0,T\right]  }%
\leq2\left\Vert M\right\Vert _{2}~.$

For $M\in\mathcal{M}_{d}^{2}$ we denote by $\left\langle M\right\rangle $ the
\textit{predictable quadratic variation} of $M$: the unique predictable,
c\`{a}dl\`{a}g, nondecreasing process with $\left\langle M\right\rangle
_{0}=0$ such that
\begin{equation}
\left\vert M_{t}\right\vert ^{2}-\left\langle M\right\rangle _{t}\quad\text{is
a local martingale;} \label{def-angle}%
\end{equation}
its existence and uniqueness are given by the Doob--Meyer decomposition of the
submartingale $\left\vert M\right\vert ^{2}$. For $M,N\in\mathcal{M}_{d}^{2}$
one sets $\left\langle M,N\right\rangle :=\left(  \left\langle
M+N\right\rangle -\left\langle M-N\right\rangle \right)  /4$.

The quadratic variation of $M$ is the c\`{a}dl\`{a}g, nondecreasing, adapted
process defined by%
\begin{equation}%
\begin{array}
[c]{rl}%
\lbrack M]_{t}:= & \text{(\textit{up)-}}\lim\limits_{n\rightarrow\infty
}\left(
{\displaystyle\sum\limits_{k=0}^{n-1}}
\left\vert M_{t\wedge t_{k+1}}-M_{t\wedge t_{k}}\right\vert ^{2}\right)  ,
\end{array}
\label{qv-def-m}%
\end{equation}
where $t_{k}=kT/n~$and the convergence is in probability, uniform with respect
to $t\in\left[  0,T\right]  $. Taking expectations in (\ref{def-angle}) and
(\ref{qv-def-m}), it follows that
\begin{equation}
\mathbb{E}\left\langle M\right\rangle _{t}=\mathbb{E}[M]_{t}=\mathbb{E}%
\left\vert M_{t}\right\vert ^{2},\qquad t\in\left[  0,T\right]  .
\label{angle-energy}%
\end{equation}

\item[$\left(  c\right)  $] $\Lambda_{d\times k}^{2}$ (and $\Lambda_{d}%
^{2}:=\Lambda_{d\times1}^{2}$) is the Hilbert space of $\mathcal{F}_{t}%
$-progressively measurable, $\mathbb{R}^{d\times k}$-valued stochastic
proce\-sses $X,Y:\Omega\times\left[  0,T\right]  \rightarrow\mathbb{R}%
^{d\times k}$ such that $\mathbb{E}\int_{0}^{T}\left\vert X_{r}\right\vert
^{2}dr<\infty$, equipped with the norm $\left\Vert \cdot\right\Vert
_{\Lambda_{d}^{2}\left[  0,T\right]  }$ induced by the inner product%
\[
\left\langle X,Y\right\rangle _{\Lambda}:=\mathbb{E}%
{\displaystyle\int_{0}^{T}}
\mathrm{Tr}\left(  X_{r}^{\ast}Y_{r}\right)  dr.
\]

\end{enumerate}

We note that $\mathbb{D}_{d}^{2}\subset\Lambda_{d}^{2}$.

\begin{condition}
[H1]\label{H1}The terminal datum $\eta\in L^{2}(\Omega,\mathcal{F}%
_{T},\mathbb{P};\mathbb{R}^{d})$.
\end{condition}

\begin{condition}
[H2]\label{H2}The generator $F\left(  \cdot,\cdot,y,z\right)  :\Omega
\times\left[  0,T\right]  \rightarrow\mathbb{R}^{d}$ is $\mathcal{F}_{t}%
$-progressively measurable for every $\left(  y,z\right)  \in\mathbb{R}%
^{d}\times\mathbb{R}^{d\times k}$, and there exist $L,\ell\in L^{2}\left(
0,T;\mathbb{R}_{+}\right)  $ such that:

\begin{itemize}
\item[$\left(  i\right)  $] \textit{Lipschitz conditions}: for all
$y,y^{\prime}\in\mathbb{R}^{d},\;z,z^{\prime}\in\mathbb{R}^{d\times
k},\;d\mathbb{P}\otimes dt$-$a.e.:$%
\begin{equation}
\left\vert F(t,y^{\prime},z)-F(t,y,z)\right\vert \leq L\left(  t\right)
|y^{\prime}-y|\quad\text{and}\quad|F(t,y,z^{\prime})-F(t,y,z)|\leq\ell\left(
t\right)  |z^{\prime}-z| \label{lip_F}%
\end{equation}

\item[$\left(  ii\right)  $] \textit{Boundedness condition}: $\mathbb{E}%
{\displaystyle\int\nolimits_{0}^{T}}
\left\vert F\left(  t,0,0\right)  \right\vert ^{2}dt<+\infty$.
\end{itemize}
\end{condition}

\begin{condition}
[H3]\label{H3}$\mathcal{R}:\mathcal{M}_{d}^{2}\longrightarrow\Lambda_{d\times
k}^{2}\quad$is a mapping such that:

\begin{itemize}
\item[$\left(  j\right)  $] $\mathcal{R}(M)\equiv0,$ if $M\equiv0,$
$\mathbb{P}$-a.s.;

\item[$\left(  jj\right)  $] there exists $C_{\mathcal{R}}>0$ such that, for
any $0\leq s<t\leq T$ and for any $M,N\in\mathcal{M}_{d}^{2}~,$%
\begin{equation}
\mathbb{E}%
{\displaystyle\int_{s}^{t}}
\left\vert \mathcal{R}_{r}(M)-\mathcal{R}_{r}(N)\right\vert ^{2}dr\leq
C_{\mathcal{R}}^{2}\,\mathbb{E}%
{\displaystyle\int_{s+}^{t}}
d\left[  M-N\right]  _{r}, \label{lip_R}%
\end{equation}
where $\left[  M-N\right]  $ is the quadratic variation (\ref{qv-def-m}) of
the martingale $M-N$, a c\`{a}dl\`{a}g nondecreasing adapted process.
\end{itemize}
\end{condition}

Since the increments of a square integrable martingale are orthogonal in
$L^{2}$, the RHS of (\ref{lip_R}) may equally be written as:%
\begin{equation}
\mathbb{E}\int_{s+}^{t}d\left[  M-N\right]  _{r}=\mathbb{E}\left[  M-N\right]
_{t}-\mathbb{E}\left[  M-N\right]  _{s}=\mathbb{E}\left\vert M_{t}%
-N_{t}\right\vert ^{2}-\mathbb{E}\left\vert M_{s}-N_{s}\right\vert ^{2}.
\label{brk incr}%
\end{equation}
Assumption $\left(  \ref{H3}-\left(  jj\right)  \right)  $ states precisely
that, for all martingales $M,N\in\mathcal{M}_{d}^{2}$, the deterministic
measure $\mathbb{E}\left\vert \mathcal{R}_{r}(M)-\mathcal{R}_{r}(N)\right\vert
^{2}dr$ is absolutely continuous with respect to $\mathbb{E}d\left[
M-N\right]  _{r}$, with density bounded by $C_{\mathcal{R}}^{2}$. Let us
observe that $\left(  jj\right)  $, written for $s=0$ and $t=T$, already gives%
\[
\left\vert \left\vert \mathcal{R}(M)-\mathcal{R}(N)\right\vert \right\vert
_{\Lambda_{d\times k}^{2}}^{2}\leq C_{\mathcal{R}}^{2}\,\mathbb{E}\left\vert
M_{T}-N_{T}\right\vert ^{2}=C_{\mathcal{R}}^{2}\left\Vert M-N\right\Vert
_{\mathcal{M}}^{2}~,
\]
i.e. $\mathcal{R}:\mathcal{M}_{d}^{2}\longrightarrow\Lambda_{d\times k}%
^{2}\quad$is a Lipschitz mapping.$\smallskip$

In order to see that the class of such mappings is rich, and that it is by no
means restricted to the extraction of the integrand of an It\^{o} integral
with respect to a Brownian motion, the interested reader in invited to consult
Negru\c{t}, R\u{a}\c{s}canu and Rotenstein \cite{Negrut/Rascanu/Rotenstein:26}%
, where some motivating examples are provided. Also, one can find there the
proof for the following result.

\begin{lemma}
\label{L1_gR}Assumption \ref{H3}$-\left(  jj\right)  $ is equivalent
to:\newline$\left(  jj^{\prime}\right)  \quad$There exists $C_{\mathcal{R}}>0$
such that, for every bounded Borel measurable function $g:\left[  0,T\right]
\rightarrow\mathbb{R}_{+}~$, all $M,N\in\mathcal{M}_{d}^{2}$ and all $0\leq
s<t\leq T$,%
\begin{equation}
\mathbb{E}\int_{s}^{t}g\left(  r\right)  \left\vert \mathcal{R}_{r}%
(M)-\mathcal{R}_{r}(N)\right\vert ^{2}dr\leq C_{\mathcal{R}}^{2}%
\,\mathbb{E}\int_{(s,t]}g\left(  r\right)  d\left[  M-N\right]  _{r}~.
\label{lip_gR}%
\end{equation}

\end{lemma}

Under the above assumptions on $F$ and $\mathcal{R}$ we will deduce that%
\[
\left(  Y,M\right)  \longrightarrow\int_{0}^{\cdot}F(s,Y_{s},\mathcal{R}%
_{s}(M))ds
\]
is a continuous mapping from $\Lambda_{d}^{2}\times\mathcal{M}_{d}^{2}$ into
${\mathcal{{\mathbb{D}}}}_{d}^{2}~$, as follows from the next lemma, whose
proof is a direct consequence of the Cauchy--Schwarz inequality together with
Assumptions \ref{H2} and \ref{H3} and is therefore omitted.

\begin{lemma}
\label{L2_F} For all $(Y,M)$, $\left(  Z,N\right)  \in\Lambda_{d}^{2}%
\times\mathcal{M}_{d}^{2}~$:
\[
\mathbb{E}\sup\limits_{t\in\left[  0,T\right]  }\left\vert
{\displaystyle\int_{0}^{t}}
\left(  F(s,Y_{s},\mathcal{R}_{s}(M))-F(s,Z_{s},\mathcal{R}_{s}(N))\right)
ds\right\vert ^{2}\leq C\left[  T\mathbb{E}\sup\limits_{s\in\left[
0,T\right]  }\left\vert Y_{s}-Z_{s}\right\vert ^{2}+\mathbb{E}\left\vert
M_{T}-N_{T}\right\vert ^{2}\right]  ;
\]
where
\[
C:=2\int_{0}^{T}L^{2}\left(  s\right)  ds+2C_{\mathcal{R}}^{2}\int_{0}^{T}%
\ell^{2}\left(  s\right)  ds;
\]
in particular, for $\left(  Z,N\right)  =\left(  0,0\right)  $,%
\begin{equation}%
\begin{array}
[c]{l}%
\mathbb{E}\sup\limits_{t\in\left[  0,T\right]  }\left\vert
{\displaystyle\int_{0}^{t}}
F(s,Y_{s},\mathcal{R}_{s}(M))ds\right\vert ^{2}\leq\mathbb{E}\left(
{\displaystyle\int_{0}^{T}}
\left\vert F(s,Y_{s},\mathcal{R}_{s}(M))\right\vert ds\right)  ^{2}\medskip\\
\quad\quad\quad\quad\quad\leq2C\times\left(  \mathbb{E}%
{\displaystyle\int_{0}^{T}}
\left\vert Y_{s}\right\vert ^{2}ds+\mathbb{E}\left\vert M_{T}\right\vert
^{2}\right)  +2\mathbb{E}\left(
{\displaystyle\int_{0}^{T}}
\left\vert F(s,0,0)\right\vert ds\right)  ^{2}%
\end{array}
\label{le-F2}%
\end{equation}

\end{lemma}

We give now a result (similar to Theorem 4.1 from Liang, Lyons, Qian
\cite{Liang/Lyons/Qian:11}) concerning the behavior of the solution for a BSDE
driven by a martingale and featuring a functional representation of the third
variable of the driver $F$. For the complete proof of this result, we refer to
Negru\c{t}, R\u{a}\c{s}canu and Rotenstein \cite{Negrut/Rascanu/Rotenstein:26}.

\begin{proposition}
\label{lip_FR}\textit{Under Assumptions \ref{H1}, \ref{H2}\ and \ref{H3},
there exists a unique pair }$(Y,M)\in\mathbb{D}_{d}^{2}\times\mathcal{M}%
_{d}^{2}~$\textit{, which is the solution of the equation }%
\begin{equation}
Y_{t}=\eta+%
{\displaystyle\int_{t}^{T}}
F\left(  r,Y_{r},\mathcal{R}_{r}(M)\right)  dr-(M_{T}-M_{t}),\text{ }%
t\in\left[  0,T\right]  ,\text{ a.s.,} \label{eq_FR}%
\end{equation}
\textit{and for some constant }$C=C(T,L,\ell,C_{\mathcal{R}})$\textit{, }%
\begin{equation}
\mathbb{E}\sup_{0\leq t\leq T}|Y_{t}|^{2}+\mathbb{E}\left\vert M_{T}%
\right\vert ^{2}\leq C\left(  \mathbb{E}|\eta|^{2}+\mathbb{E}\int_{0}%
^{T}|F(r,0,0)|^{2}dr\right)  . \label{estim FR}%
\end{equation}

\end{proposition}

In what follows we first revisit the results of Theorem 3.5 from Bensoussan,
Li and Yam \cite{Bensoussan/Li/Yam:2018}. Based on the estimates obtained in
Proposition \ref{Lemma with important bounds}, it is proved in
\cite{Negrut/Rascanu/Rotenstein:26} that, within a rigorous c\`{a}dl\`{a}g
framework, this oblique reflected BSDE admits a unique solution, but with the
obliquity produced by a symmetric matrix $H$, satisfying the assumptions
introduced by Gassous, R\u{a}\c{s}canu and Rotenstein
\cite{Gassous/Rascanu/Rotenstein:12}, \cite{Gassous/Rascanu/Rotenstein:15}.
After this, in Section \ref{main section}, we take a bidimensional obstacle
problem, consider $H$ to be a rotation matrix and develop our main result.

\section{BSDEs with oblique subgradients and martingale noise}

We are interested in the study of the following problem.

\begin{problem}
Show that there exists a unique triple $\left(  Y,M,U\right)  $, belonging to
a suitable space, such that%
\begin{equation}
\left\{
\begin{array}
[c]{l}%
Y_{t}+%
{\displaystyle\int_{t}^{T}}
H_{r}U_{r}dr=\eta+%
{\displaystyle\int_{t}^{T}}
F\left(  r,Y_{r},\mathcal{R}_{r}(M)\right)  dr-(M_{T}-M_{t}),\text{ }\forall
t\in\left[  0,T\right]  ,\text{ }\mathbb{P}\text{-}a.s.,\smallskip\\
U_{r}\in\partial\varphi\left(  Y_{r}\right)  ,\text{ }d\mathbb{P\otimes
}dr\text{-a.e.}%
\end{array}
\right.  \label{def of solution}%
\end{equation}
The exact meaning of the solution is given in Definition \ref{def of sol}.
\end{problem}

For now, assume that $\eta$, $F$ and $\mathcal{R}$ satisfy the classical
Assumptions \ref{H1}, \ref{H2}, \ref{H3}, respectively, and we impose the
following further assumptions, in order to highlight the working framework
considered until the present study.

\begin{condition}
[H4]\label{H4}The transforming term $H(\cdot,\cdot):\Omega\times\left[
0,T\right]  \rightarrow\mathbb{R}^{d\times d}$ is an $\mathcal{F}_{t}%
$-progressively measurable, symmetric matrix-valued process and there exist
constants $a_{H},b_{H},c_{H}>0$ such that $H_{\cdot}\left(  \omega\right)
=\left(  h_{i,j}\left(  \omega,\cdot\right)  \right)  _{d\times d}\in
C^{1}\left(  \left[  0,T\right]  ;\mathbb{R}^{d\times d}\right)  $ with the
operatorial norm satisfying $\left\vert \left\vert dH_{r}\left(
\omega\right)  /dr\right\vert \right\vert _{op}\leq c_{H}$, for every$\,r\in
\left[  0,T\right]  ,$ $\mathbb{P}$-a.s.. Moreover, for all $t\in\left[
0,T\right]  $ and $u\in\mathbb{R}^{d}$, $\mathbb{P}$-$a.s.$,%
\begin{equation}
a_{H}\left\vert u\right\vert ^{2}\leq\left\langle H_{t}u,u\right\rangle \leq
b_{H}\left\vert u\right\vert ^{2}. \label{h2}%
\end{equation}

\end{condition}

\begin{condition}
[H5]\label{H5}The function $\varphi:\mathbb{R}^{d}\rightarrow(-\infty
,+\infty]$ is a proper lower semicontinuous convex function and the terminal
datum satisfies $\mathbb{E}\,\varphi\left(  \eta\right)  <+\infty.$
\end{condition}

\noindent Denote by $\partial\varphi$ the subdifferential operator of
$\varphi$, i.e.%
\[
\partial\varphi\left(  x\right)  :=\{\hat{x}\in\mathbb{R}^{d}:\left\langle
\hat{x},y-x\right\rangle +\varphi\left(  x\right)  \leq\varphi\left(
y\right)  ,\text{ for all }y\in\mathbb{R}^{d}\}
\]
and define $Dom\left(  \partial\varphi\right)  :=\{x\in\mathbb{R}^{d}%
:\partial\varphi\left(  x\right)  \neq\emptyset\}$. Let us use the notation
$(x,\hat{x})\in\partial\varphi$ in order to express that $x\in Dom\left(
\partial\varphi\right)  $ and $\hat{x}\in\partial\varphi\left(  x\right)  $.

\begin{definition}
The vector given by the quantity $H_{t}\hat{x}$, with $\hat{x}\in
\partial\varphi\left(  x\right)  $, is called an \textit{oblique subgradient}.
\end{definition}

\begin{definition}
\label{def of sol}Given a stochastic basis $\left(  \Omega,\mathcal{F}%
,\mathbb{P},\mathbb{F}=\{\mathcal{F}_{t}\}_{t\geq0}\right)  $, we say that a
triplet $\left(  Y,M,U\right)  $ is a strong solution for the oblique
reflected BSVI (\ref{def of solution}) if $\left(  Y,M,U\right)
\in{\mathcal{{\mathbb{D}}}}_{d}^{0}\times\mathcal{M}_{d}^{0}\times\Lambda
_{d}^{0}$ and, $\mathbb{P}$-a.s.,%
\[
Y_{t}+%
{\displaystyle\int_{t}^{T}}
H_{r}U_{r}dr=\eta+%
{\displaystyle\int_{t}^{T}}
F\left(  r,Y_{r},\mathcal{R}_{r}(M)\right)  dr-(M_{T}-M_{t}),\text{ }\forall
t\in\left[  0,T\right]  .
\]
Moreover, for every progressively measurable stochastic process $v$, and any
$0\leq s\leq t\leq T$,%
\[
\mathbb{E}\int_{s}^{t}\left\langle v_{r}-Y_{r},U_{r}\right\rangle
dr+\mathbb{E}\int_{s}^{t}\varphi\left(  Y_{r}\right)  dr\leq\mathbb{E}\int
_{s}^{t}\varphi\left(  v_{r}\right)  dr,
\]
i.e. $U_{r}\in\partial\varphi\left(  Y_{r}\right)  ,$ $d\mathbb{P\otimes}dr$-a.e.
\end{definition}

\begin{theorem}
[Refined version of Theorem 3.5 from Bensoussan, Li, Yam
\cite{Bensoussan/Li/Yam:2018}]\label{Main result1} Assume that we situate in
the working framework described above, and let the Assumptions \ref{H1},
\ref{H2}, \ref{H3}, \ref{H4} and \ref{H5} be satisfied. Assume,\ also, that
$\ell\left(  t\right)  \equiv\ell$ is a positive constant satisfying
$a_{H}>\ell^{2}C_{\mathcal{R}}^{2}/2$. Then Problem \ref{def of solution}
admits a unique solution in the sense of Definition \ref{def of sol}.
Moreover, $\left(  Y,M,U\right)  \in{\mathcal{{\mathbb{D}}}}_{d}^{2}%
\times\mathcal{M}_{d}^{2}\times\Lambda_{d}^{2}$~.
\end{theorem}

As a first step in proving this theorem, we employ the classical Moreau-Yosida
technique for penalizing the equation. The a priori estimates which can be
obtained for the approximating sequence of solutions do not require the
symmetry of $H$, remaining valid under weaker assumptions for this matrix.
More precisely, one can replace here Assumption \ref{H4} by the following
weaker one.

\begin{condition}
[H4W]\label{H4W}The transforming term $H(\cdot,\cdot):\Omega\times\left[
0,T\right]  \rightarrow\mathbb{R}^{d\times d}$ is a progressively measurable,
matrix-valued stochastic process and there exist some constants $a_{H}%
,b_{H}>0$ such that, for all $u\in\mathbb{R}^{d}$,%
\[
a_{H}\left\vert u\right\vert ^{2}\leq\left\langle H_{t}u,u\right\rangle
\leq\left\Vert H_{t}\right\Vert _{op}\left\vert u\right\vert ^{2}\leq
b_{H}\left\vert u\right\vert ^{2},\quad d\mathbb{P}\otimes dt\text{-a.e.}%
\]

\end{condition}

Let $0<\varepsilon\leq1$ and consider the approximating BSDE, driven by a
martingale term:%
\begin{equation}
Y_{t}^{\varepsilon}+%
{\displaystyle\int_{t}^{T}}
H_{r}\nabla\varphi_{\varepsilon}\left(  Y_{r}^{\varepsilon}\right)  dr=\eta+%
{\displaystyle\int_{t}^{T}}
F\left(  r,Y_{r}^{\varepsilon},\mathcal{R}_{r}(M^{\varepsilon})\right)
dr-(M_{T}^{\varepsilon}-M_{t}^{\varepsilon}),\quad\forall t\in\left[
0,T\right]  , \label{approximating eq for general case}%
\end{equation}
with $\varphi_{\varepsilon}$ being the Moreau-Yosida regularization of the
proper convex lower-semicontinuous function $\varphi$ (see Lemma \ref{conv}).
We have the following result, which is, in fact, \textit{Milestone 1} of the
proof of Theorem \ref{Main result1}.

\begin{proposition}
\label{Lemma with important bounds} Let Assumptions \ref{H1}, \ref{H2},
\ref{H3}, \ref{H5}, and \ref{H4W} be satisfied, and assume, as announced in
Theorem \ref{Main result1}, that $\ell\left(  t\right)  \equiv\ell$ is a
positive constant such that the following compatibility condition holds:
$a_{H}>\ell^{2}C_{\mathcal{R}}^{2}/2$. Then, the penalized equation
(\ref{approximating eq for general case}) admits a unique solution $\left(
Y^{\varepsilon},M^{\varepsilon}\right)  \in{\mathcal{{\mathbb{D}}}}_{d}%
^{2}\times\mathcal{M}_{d}^{2}$ and the following estimates hold: there exists
a positive constant $C$, independent of $\varepsilon$, such that,%
\begin{equation}
\left\{
\begin{array}
[c]{ll}%
\left(  j\right)  \quad & \mathbb{E}\sup\limits_{t\in\left[  0,T\right]
}|Y_{t}^{\varepsilon}|^{2}+\mathbb{E}\sup\limits_{r\in\left[  0,T\right]
}\left\vert M_{r}^{\varepsilon}\right\vert ^{2}\leq C,\medskip\\
\left(  jj\right)  \quad & \mathbb{E}%
{\displaystyle\int_{0}^{T}}
\left\vert \nabla\varphi_{\varepsilon}(Y_{r}^{\varepsilon})\right\vert
^{2}dr+\mathbb{E}%
{\displaystyle\int_{0}^{T}}
\left\vert \mathcal{R}_{r}(M^{\varepsilon})\right\vert ^{2}dr\leq C,\medskip\\
\left(  jjj\right)  \quad & \mathbb{E}%
{\displaystyle\int_{0}^{T}}
\left\vert Y_{r}^{\varepsilon}-J_{\varepsilon}\left(  Y_{r}^{\varepsilon
}\right)  \right\vert ^{2}dr\leq C\varepsilon^{2},\medskip\\
\left(  jv\right)  \quad & \mathbb{E}%
{\displaystyle\int_{0}^{T}}
\left\vert \varphi\left(  J_{\varepsilon}\left(  Y_{r}^{\varepsilon}\right)
\right)  \right\vert dr\leq C\text{.}%
\end{array}
\right.  \label{estimates for the approx eq.}%
\end{equation}

\end{proposition}

For the proof of this result and of Theorem \ref{Main result1}, we refer to
Negru\c{t}, R\u{a}\c{s}canu and Rotenstein \cite{Negrut/Rascanu/Rotenstein:26}.

\section{Oblique reflection driven by a rotation matrix\label{main section}}

We consider the bidimensional multivalued oblique reflected BSDE: for all
$t\in\left[  0,T\right]  ,$%
\begin{equation}
\left\{
\begin{array}
[c]{l}%
Y_{t}+%
{\displaystyle\int_{t}^{T}}
\Theta_{r}U_{r}dr=\eta+%
{\displaystyle\int_{t}^{T}}
F\left(  r,Y_{r},\mathcal{R}_{r}(M)\right)  dr-(M_{T}-M_{t}),\quad
\mathbb{P}\text{-a.s.,}\medskip\\
U_{r}\in\partial I_{K}(Y_{r}),\text{ \ }d\mathbb{P\otimes}dr\text{-a.e. on
}\Omega\times\left[  0,T\right]  .
\end{array}
\right.  \label{r1}%
\end{equation}
Consider $d=2$ and Assumptions \ref{H1}, \ref{H2}, \ref{H3} are fulfilled.
Usually, until the present study, the symmetry of the matrix $H$ is essential
for proving the convergence in $\mathbb{D}_{d}^{2}$ of the approximate
solution, given by the Moreau-Yosida penalizations, $Y^{\varepsilon}$. By
abandoning the symmetry assumption of $H$, we limit the analysis to the case
$d=2$ and $H$ being a rotation matrix.

Remark that the symmetry of $H$ is assumed throughout the oblique subgradient
literature: in the seminal papers of Gassous, R\u{a}\c{s}canu and Rotenstein
\cite{Gassous/Rascanu/Rotenstein:12} on forward variational inequalities,
where the Lyapunov function is given by $\left\langle \left(  [H(x)]^{-1}%
+[H(\hat{x})]^{-1}\left(  x-\hat{x}\right)  ,x-\hat{x}\right)  \right\rangle
$, and also in \cite{Gassous/Rascanu/Rotenstein:15}, which treated the first
approach to BSVIs, where it is given by $||H_{s}^{-1/2}(Y_{s}-\tilde{Y}%
_{s})||$, and so on; neither works without $H=H^{\ast}$, and neither produces
a tangential term. The alternative used in the Skorokhod problem literature is
not the symmetry, but a smallness condition on the obliqueness - the spectral
condition $\sigma\left(  V\right)  <1$ of Ramasubramanian
\cite{Ramasubramanian:02}, or the \textquotedblleft set $B$\textquotedblright%
\ condition of Dupuis and Ishii \cite{Dupuis/Ishii:93} - and it is a
requirement of this second kind, here bearing on the angle $\theta$, that will
appear below. For obliquely reflected BSDEs in a convex domain, with a
direction of reflection depending on time and on the solution, see Chassagneux
and Richou \cite{Chassagneux/Richou:20}. However, even here, the existing
arguments still rely essentially on the symmetry of the obliqueness
matrix.\medskip

\textit{The proof of the main result requires three distinct types of
compatibility. The first concerns the ellipticity of the two-point kernel
}$\left(  C_{2}:\text{(\ref{r15})}\right)  $\textit{, the second the sign of
the boundary reflection term }$\left(  C_{0}:\text{(\ref{C0})}\right)
$\textit{, and the third the absorption of the second-order defect generated
by the c\`{a}dl\`{a}g It\^{o} formula }$\left(  C_{3}:\text{(\ref{C3}%
)}\right)  $\textit{.}\medskip

Obviously, the classical Assumption \ref{H4} is replaced by

\begin{condition}
[H4']\label{H4'} $H:\left[  0,T\right]  \rightarrow\mathbb{R}^{2\times2}$ is a
rotation matrix%
\[
H_{r}=\Theta_{r}:=\left(
\begin{array}
[c]{cc}%
\cos\theta_{r} & -\sin\theta_{r}\medskip\\
\sin\theta_{r} & \cos\theta_{r}%
\end{array}
\right)  =\cos\theta_{r}\,I+\sin\theta_{r}\,J,
\]
where
\[
I=\left(
\begin{array}
[c]{cc}%
1 & 0\\
0 & 1
\end{array}
\right)  ,\quad\quad J=\left(
\begin{array}
[c]{cc}%
0 & -1\\
1 & 0
\end{array}
\right)  =-J^{T},
\]
and $\theta\in C^{1}([0,T];(0,\pi/2)),$ with the assumption that%
\[
a_{\theta}:=\min\limits_{r\in\left[  0,T\right]  }\left(  \cos\theta
_{r}\right)  >0.
\]

\end{condition}

We remark that, for every $u$,%
\begin{equation}
a_{\theta}\left\vert u\right\vert ^{2}\leq\cos\theta_{r}\left\vert
u\right\vert ^{2}=\left\langle H_{r}u,u\right\rangle \leq\left\vert
u\right\vert ^{2} \label{r2}%
\end{equation}
Indeed, since $\left\langle Ju,u\right\rangle =0,$ for every $u$, and
$0<\cos\theta_{r}<1$,%
\[
\left\langle H_{r}u,u\right\rangle =\cos\theta_{r}\left\langle
Iu,u\right\rangle =\cos\theta_{r}\left\vert u\right\vert ^{2}\leq\left\vert
u\right\vert ^{2}.
\]
The Assumption \ref{H5} on the function $\varphi$ is replaced by the
corresponding key assumption for a particular $\varphi,$ which is give by our
obstacle problem on $K$.

\begin{condition}
[H5']\label{H5'}$\varphi:=I_{K}:\mathbb{R}^{2}$ $\rightarrow\left[
0,\infty\right]  $ is the convexity indicator of a bounded closed convex set
$K\subset\mathbb{R}^{2}$ with $\operatorname{int}K\neq\emptyset$, and
characterized by%
\begin{equation}
\operatorname{int}K=\{x\in\mathbb{R}^{2}:\psi\left(  x\right)  <0\},\qquad
bd\left(  K\right)  =\partial K=\{x\in\mathbb{R}^{2}:\psi\left(  x\right)
=0\},\qquad|\nabla\psi(x)|=1\ \text{on }\partial K, \label{r3}%
\end{equation}
where $\psi:\mathbb{R}^{2}\rightarrow\mathbb{R}$ is a convex and smooth enough
function. More precisely, we complete \ref{H5'} by fixing the constants
\begin{equation}
a_{\psi}>0,\qquad b_{\psi}\geq1,\qquad D_{K}:=\operatorname{diam}K,\qquad
b_{\ast}:=b_{\psi}\left(  1+b_{\psi}D_{K}\right)  , \label{r7}%
\end{equation}
such that, on a bounded open neighborhood $\mathcal{V}$ of $K$,
\begin{equation}
a_{\psi}I\leq D^{2}\psi\left(  x\right)  \leq b_{\psi}I,\qquad\left\vert
\left\vert D^{3}\psi\left(  x\right)  \right\vert \right\vert \leq b_{\psi}.
\label{r8}%
\end{equation}
Thus $\psi$ is uniformly convex on $\mathcal{V}$ and of class $C^{3}$ with
bounded derivatives there.

The terminal datum satisfies $\eta\in K$, $a.s..$
\end{condition}

For our obstacle problem%
\[
\varphi\left(  y\right)  :=\left\{
\begin{array}
[c]{ll}%
0, & \text{if }y\in K,\medskip\\
+\infty & \text{if }y\in\mathbb{R}^{2}\backslash K
\end{array}
\right.
\]
and the instruments appearing in the Moreau-Yosida penalization approach are:%
\[
\varphi_{\varepsilon}\left(  y\right)  =\dfrac{1}{2\varepsilon}d_{K}%
^{2}\left(  y\right)  \text{,}\quad\quad\text{and}\quad\quad\nabla
\varphi_{\varepsilon}\left(  y\right)  =\dfrac{y-\pi_{K}\left(  y\right)
}{\varepsilon}.
\]
where $\pi_{K}:\mathbb{R}^{2}\rightarrow K$ is the orthogonal projection on
$K$ and $d_{K}\left(  y\right)  =\left\vert y-\pi_{K}\left(  y\right)
\right\vert $ is the distance from $y$ to $K$.

\begin{remark}
The condition $U_{r}\in\partial I_{K}(Y_{r})$ on $\left[  0,T\right]  $ means
that $Y_{r}\in K$, $d\mathbb{P}\otimes dr$-a.e., and%
\[%
{\displaystyle\int_{s}^{t}}
\left\langle v_{r}-Y_{r},U_{r}\right\rangle dr\leq0,\quad\text{for all }0\leq
s\leq t\leq T\text{ and all }v\in C(\left[  0,T\right]  ;K).
\]

\end{remark}

Let $0<\varepsilon\leq1$. By Proposition \ref{Lemma with important bounds}
there exists a unique pair $\left(  Y^{\varepsilon},M^{\varepsilon}\right)
\in{\mathcal{{\mathbb{D}}}}_{d}^{2}\times\mathcal{M}_{d}^{2}$ which is the
solution of the approximating equation, for all $t\in\left[  0,T\right]  $,%
\begin{equation}
\left\{
\begin{array}
[c]{l}%
Y_{t}^{\varepsilon}+%
{\displaystyle\int_{t}^{T}}
\Theta_{r}U_{r}^{\varepsilon}dr=\eta+%
{\displaystyle\int_{t}^{T}}
F\left(  r,Y_{r}^{\varepsilon},\mathcal{R}_{r}(M^{\varepsilon})\right)
dr-(M_{T}^{\varepsilon}-M_{t}^{\varepsilon}),\quad\mathbb{P}\text{-a.s.\ ,}%
\medskip\\
U_{r}^{\varepsilon}=\dfrac{1}{\varepsilon}\left(  Y_{r}^{\varepsilon}-\pi
_{K}\left(  Y_{r}^{\varepsilon}\right)  \right)
\end{array}
\right.  \label{r4}%
\end{equation}
and the following estimates hold with a positive constant $C$ independent of
$\varepsilon$:%
\begin{equation}
\left\{
\begin{array}
[c]{ll}%
\left(  a\right)  \quad & \mathbb{E}\sup\limits_{t\in\left[  0,T\right]
}|Y_{t}^{\varepsilon}|^{2}+\mathbb{E}\sup\limits_{r\in\left[  0,T\right]
}\left\vert M_{r}^{\varepsilon}\right\vert ^{2}\leq C\medskip\\
\left(  b\right)  \quad & \mathbb{E}%
{\displaystyle\int_{0}^{T}}
\left\vert U_{r}^{\varepsilon}\right\vert ^{2}dr+\mathbb{E}%
{\displaystyle\int_{0}^{T}}
\left\vert \mathcal{R}_{r}(M^{\varepsilon})\right\vert ^{2}dr\leq C,\medskip\\
\left(  c\right)  \quad & \mathbb{E}%
{\displaystyle\int_{0}^{T}}
d_{K}^{2}\left(  Y_{r}^{\varepsilon}\right)  dr\leq C\varepsilon^{2}.
\end{array}
\right.  \label{r5}%
\end{equation}
Since the rotation matrix $\Theta$ is not symmetric, we cannot follow the
classical steps to prove the convergence of the sequence $\left(
Y^{\varepsilon},M^{\varepsilon}\right)  $. For a symmetric matrix $H$ the
employed test function was $|H_{r}^{-1/2}\left(  x-y\right)  |^{2}$, a
quadratic form in $x-y$ whose Hessian does not depend on $\left(  x,y\right)
$, and the reflection terms were disposed of by the monotonicity of
$\partial\varphi$. For a rotation this fails at the very first step: if
$x\in\partial K$ and $\mathbf{n}\left(  x\right)  $ is the outward unit normal
at $x$, then
\begin{equation}
\left\langle \nabla_{x}\left\vert x-y\right\vert ^{2},\Theta_{r}%
\mathbf{n}\left(  x\right)  \right\rangle =2\cos\theta_{r}\left\langle
x-y,\mathbf{n}\left(  x\right)  \right\rangle +2\sin\theta_{r}\left\langle
x-y,J\mathbf{n}\left(  x\right)  \right\rangle . \label{r6}%
\end{equation}
The first term is nonnegative, by the convexity of $K$. The second one is the
\textit{tangential} component of the oblique direction; it is of first order
in $\left\vert x-y\right\vert $ and of arbitrary sign, and it is precisely
what destroys every estimate of the approach when $H$ is of symmetric type.
The remedy is to replace $\left\vert x-y\right\vert ^{2}$ by a non-quadratic
test function designed so that this term cancels identically.

\subsection{The structure of the geometric framework\label{rot-geometry}}

We work under the hypotheses that Assumptions \ref{H1}, \ref{H2}, \ref{H3},
\ref{H4'} and \ref{H5'} hold.

\begin{lemma}
[Geometry of the set $K$]\label{K-geometry}Let $K$ and $\psi$ be as in
Assumption \ref{H5'}, and let $a_{\psi},b_{\psi},b_{\ast}$ be as in
(\ref{r7})--(\ref{r8}). Let $x\in bd\left(  K\right)  $ and $y\in K$. Then

\begin{itemize}
\item[$\left(  a\right)  $] $\left\vert \nabla\psi\left(  z\right)
\right\vert \leq1$ for every $z\in K$;

\item[$\left(  b\right)  $] $0\leq1-\left\langle \nabla\psi\left(  x\right)
,\nabla\psi\left(  y\right)  \right\rangle \leq b_{\ast}\left\vert
x-y\right\vert $;

\item[$\left(  c\right)  $] $\left\vert \left\langle J\nabla\psi\left(
y\right)  ,\nabla\psi\left(  x\right)  \right\rangle \right\vert \leq b_{\psi
}\left\vert x-y\right\vert $;

\item[$\left(  d\right)  $] $\left\langle x-y,\nabla\psi\left(  x\right)
\right\rangle \geq-\psi\left(  y\right)  +\dfrac{a_{\psi}}{2}\left\vert
x-y\right\vert ^{2}\geq\dfrac{a_{\psi}}{2}\left\vert x-y\right\vert ^{2}%
\geq0.$
\end{itemize}
\end{lemma}

\begin{proof}
$\left(  a\right)  $ Let $z\in K$ with $\nabla\psi\left(  z\right)  \neq0$,
and let $z^{\prime}\in bd\left(  K\right)  $ be the first exit point of the
half-line $z+t\nabla\psi\left(  z\right)  $, $t\geq0$. The monotonicity of
$\nabla\psi$ gives, for $z^{\prime}=z+t\nabla\psi\left(  z\right)  $,
$\left\langle \nabla\psi\left(  z^{\prime}\right)  -\nabla\psi\left(
z\right)  ,\nabla\psi\left(  z\right)  \right\rangle \geq0$, hence $\left\vert
\nabla\psi\left(  z^{\prime}\right)  \right\vert \left\vert \nabla\psi\left(
z\right)  \right\vert \geq\left\vert \nabla\psi\left(  z\right)  \right\vert
^{2}$ and, by (\ref{r3}), $\left\vert \nabla\psi\left(  z\right)  \right\vert
\leq\left\vert \nabla\psi\left(  z^{\prime}\right)  \right\vert =1$.

$\left(  b\right)  $ With $\mathrm{m}:=\nabla\psi\left(  y\right)  $ and
$\mathrm{n}:=\nabla\psi\left(  x\right)  $, so that $\left\vert \mathrm{n}%
\right\vert =1$ and $\left\vert \mathrm{m}\right\vert \leq1$ by $\left(
a\right)  $,
\[
1-\left\langle \mathrm{n},\mathrm{m}\right\rangle =\dfrac{1}{2}\left\vert
\mathrm{n}-\mathrm{m}\right\vert ^{2}+\dfrac{1}{2}\left(  1-\left\vert
\mathrm{m}\right\vert ^{2}\right)  \geq0.
\]
Moreover $\left\vert \mathrm{n}-\mathrm{m}\right\vert \leq\left\vert
\left\vert D^{2}\psi\right\vert \right\vert \left\vert x-y\right\vert \leq
b_{\psi}\left\vert x-y\right\vert $ by (\ref{r8}), and
\[
1-\left\vert \mathrm{m}\right\vert ^{2}\leq2\left(  1-\left\vert
\mathrm{m}\right\vert \right)  =2\left(  \left\vert \mathrm{n}\right\vert
-\left\vert \mathrm{m}\right\vert \right)  \leq2\left\vert \mathrm{n}%
-\mathrm{m}\right\vert \leq2b_{\psi}\left\vert x-y\right\vert .
\]
If we also use $\left\vert \mathrm{n}-\mathrm{m}\right\vert ^{2}\leq b_{\psi
}\left\vert x-y\right\vert \left\vert n-m\right\vert \leq b_{\psi}^{2}%
D_{K}\left\vert x-y\right\vert $, by adding the terms, we obtain%
\[
1-\left\langle \mathrm{n},\mathrm{m}\right\rangle \leq\dfrac{1}{2}b_{\psi}%
^{2}D_{K}\left\vert x-y\right\vert +b_{\psi}\left\vert x-y\right\vert \leq
b_{\ast}\left\vert x-y\right\vert .
\]

$\left(  c\right)  $ Since $\left\langle J\nabla\psi\left(  x\right)
,\nabla\psi\left(  x\right)  \right\rangle =0$, one has $\left\langle
J\nabla\psi\left(  y\right)  ,\nabla\psi\left(  x\right)  \right\rangle
=\left\langle J\left(  \nabla\psi\left(  y\right)  -\nabla\psi\left(
x\right)  \right)  ,\nabla\psi\left(  x\right)  \right\rangle $, and it
suffices to use $\left\vert \left\vert J\right\vert \right\vert _{op}=1$ and
(\ref{r8}) and $\left(  c\right)  $ follows.

$\left(  d\right)  $ The $a_{\psi}$-uniform convexity (\ref{r8}) of $\psi$
gives, by the Taylor formula, $\psi\left(  y\right)  \geq\psi\left(  x\right)
+\left\langle \nabla\psi\left(  x\right)  ,y-x\right\rangle +\frac{a_{\psi}%
}{2}\left\vert x-y\right\vert ^{2}$. Since $x\in bd\left(  K\right)  $ and
$y\in K$ it follows $\psi\left(  x\right)  =0$ and $\psi\left(  y\right)
\leq0$, by (\ref{r3}). Consequently,%
\[
\left\langle \nabla\psi\left(  x\right)  ,x-y\right\rangle \geq-\psi\left(
y\right)  +\frac{a_{\psi}}{2}\left\vert x-y\right\vert ^{2},
\]
that is $\left(  d\right)  $ holds and the proof is now complete. \hfill
\end{proof}

\begin{remark}
Only the behavior of $\nabla\psi$ near $K$ will matter, while the penalized
solutions $Y^{\varepsilon}$ constructed below need not take their values in
$K$. We therefore work with a globally bounded \textit{normal field}: let
$\mathbf{n}:\mathbb{R}^{2}\rightarrow\mathbb{R}^{2}$ be of class $C^{2}$ with
\begin{equation}
\mathbf{n}=\nabla\psi\ \text{ on a neighborhood of }K,\quad\quad\left\vert
\mathbf{n}\right\vert \leq1,\quad\quad\left\vert \left\vert D\mathbf{n}%
\right\vert \right\vert +\left\vert \left\vert D^{2}\mathbf{n}\right\vert
\right\vert \leq Cb_{\psi}=:C_{\psi}\quad\text{on }\mathbb{R}^{2}. \label{r9}%
\end{equation}
Such a field exists since, by Lemma \ref{K-geometry}-$\left(  a\right)  $, one
has $\left\vert \nabla\psi\right\vert \leq1$ on $K$, so it suffices to leave
$\nabla\psi$ unchanged on $\{\left\vert \nabla\psi\right\vert \leq1\}\supset
K$ and to compose it with a smooth radial truncation at level $1$ outside; on
the truncation region $\left\vert \nabla\psi\right\vert \geq1$, so the
composition is $C^{2}$ and (\ref{r8}) gives the stated bounds. On $K$,
$\mathbf{n}$ is the gradient of $\psi$, and on $bd\left(  K\right)  $ it is
the outward unit normal. We write
\begin{equation}
\mathbf{n}_{x,y}:=\mathbf{n}\left(  x\right)  +\mathbf{n}\left(  y\right)
,\qquad\left\vert \mathbf{n}_{x,y}\right\vert \leq2. \label{def on nxy}%
\end{equation}

\end{remark}

\subsection{The key geometrical kernel: the test function\label{rot-test}}

For $x,y\in\mathbb{R}^{2}$ let us define the following bilinear forms, with
$\mathbf{n}_{x,y}$ given by (\ref{def on nxy}):%
\[
P\left(  x,y\right)  :=\left\langle x-y,\mathbf{n}_{x,y}\right\rangle
\quad\text{and}\quad Q\left(  x,y\right)  :=\left\langle x-y,J\mathbf{n}%
_{x,y}\right\rangle ,
\]
and, for every $r\in\left[  0,T\right]  $, define the key ingredient for the
entire study:%
\begin{equation}
\Phi_{r}\left(  x,y\right)  :=\left\vert x-y\right\vert ^{2}-\dfrac{1}{2}%
\tan\theta_{r}\,P\left(  x,y\right)  Q\left(  x,y\right)  . \label{r10}%
\end{equation}
It is easy to observe that $\Phi_{r}$ is \textit{symmetric}, $\Phi_{r}\left(
y,x\right)  =\Phi_{r}\left(  x,y\right)  $, since both $P$ and $Q$ change
their sign when $x$ and $y$ are interchanged. According to (\ref{def on nxy}),
$\left\vert \mathbf{n}_{x,y}\right\vert \leq2$ and, consequently, we have
$\max\{\left\vert P\left(  x,y\right)  \right\vert ,\left\vert Q\left(
x,y\right)  \right\vert \}\leq2\left\vert x-y\right\vert $. Therefore, from
(\ref{r9}), $\Phi_{r}$ is of class $C^{2}$ in $\left(  x,y\right)  $ and%
\begin{equation}
\left\vert \nabla_{x}\Phi_{r}\right\vert +\left\vert \nabla_{y}\Phi
_{r}\right\vert \leq c_{\Phi}\left\vert x-y\right\vert ,\quad\quad\left\vert
\left\vert D_{\left(  x,y\right)  }^{2}\Phi_{r}\right\vert \right\vert \leq
c_{\Phi},\quad\quad\left\vert \partial_{r}\Phi_{r}\right\vert \leq c_{\Phi
}\left\vert \left\vert \theta^{\prime}\right\vert \right\vert _{\infty
}\left\vert x-y\right\vert ^{2}, \label{r11}%
\end{equation}
where $c_{\Phi}=c_{\Phi}\left(  b_{\psi}\right)  >0$ is a constant that may
change from line to line. The first two inequalities from (\ref{r11}) follow
from (\ref{r9}), the boundedness of $\left\vert \left\vert Dn\right\vert
\right\vert $ and $\left\vert \left\vert D^{2}n\right\vert \right\vert $. The
last bound follows from
\[
\partial_{r}\Phi_{r}=\frac{\partial}{\partial_{r}}\Phi_{r}=-\frac
{\theta^{\prime}\left(  r\right)  }{2\cos^{2}\theta\left(  r\right)  }P\left(
x,y\right)  Q\left(  x,y\right)  .
\]
In particular $\nabla_{x}\Phi_{r}$ and $\nabla_{y}\Phi_{r}$ are globally
Lipschitz in $\left(  x,y\right)  $, uniformly with respect to $r$, with a
constant still denoted $c_{\Phi}$. Finally, since%
\[
\left\{
\begin{array}
[c]{l}%
\nabla_{x}P\left(  x,y\right)  +\nabla_{y}P\left(  x,y\right)  =\left(
D\mathbf{n}\left(  x\right)  +D\mathbf{n}\left(  y\right)  \right)  \left(
x-y\right)  ,\medskip\\
\nabla_{x}Q\left(  x,y\right)  +\nabla_{y}Q\left(  x,y\right)  =J\left(
D\mathbf{n}\left(  x\right)  +D\mathbf{n}\left(  y\right)  \right)  \left(
x-y\right)
\end{array}
\right.
\]
and $P\left(  x,y\right)  ,Q\left(  x,y\right)  =\mathcal{O}\left(  \left\vert
x-y\right\vert \right)  $, we obtain the important estimate%
\begin{equation}
\left\vert \nabla_{x}\Phi_{r}+\nabla_{y}\Phi_{r}\right\vert \leq c_{\Phi
}\left\vert x-y\right\vert ^{2}, \label{r12}%
\end{equation}
which proves to be crucial at Step 3 from Theorem \ref{rot-thm-conv}.

\begin{lemma}
[Two-sided bound; the case $d=2$ is mandatary]\label{rot-two}For all
$x,y\in\mathbb{R}^{2}$ and all $r\in\left[  0,T\right]  $, we have
\[
\left(  1-\tan\theta_{r}\right)  \left\vert x-y\right\vert ^{2}\leq\Phi
_{r}\left(  x,y\right)  \leq\left(  1+\tan\theta_{r}\right)  \left\vert
x-y\right\vert ^{2}.
\]

\end{lemma}

\begin{proof}
If $\mathbf{n}_{x,y}=0$ then $P\left(  x,y\right)  =Q\left(  x,y\right)  =0$
and there is nothing to prove. If $\mathbf{n}_{x,y}\neq0$, let us consider the
unity-norm vectors $e_{1}:=\mathbf{n}_{x,y}/\left\vert \mathbf{n}%
_{x,y}\right\vert $ and $e_{2}:=J\mathbf{n}_{x,y}/\left\vert \mathbf{n}%
_{x,y}\right\vert $. It is easy to observe that the set $\left\{  e_{1}%
,e_{2}\right\}  $ is an orthonormal basis of $\mathbb{R}^{2}$, because
$\left\vert J\mathbf{n}_{x,y}\right\vert =\left\vert \mathbf{n}_{x,y}%
\right\vert $ and $\left\langle J\mathbf{n}_{x,y},\mathbf{n}_{x,y}%
\right\rangle =0$. The coordinates of $x-y$ with respect to this basis, are%
\[
\left\langle x-y,e_{1}\right\rangle =\frac{P\left(  x,y\right)  }{\left\vert
\mathbf{n}_{x,y}\right\vert }\quad\text{and}\quad\left\langle x-y,e_{2}%
\right\rangle =\frac{Q\left(  x,y\right)  }{\left\vert \mathbf{n}%
_{x,y}\right\vert }.
\]
The Parseval identity $\left\vert x-y\right\vert ^{2}=\left\langle
x-y,e_{1}\right\rangle ^{2}+\left\langle x-y,e_{2}\right\rangle ^{2}$ leads to%
\begin{equation}
P^{2}\left(  x,y\right)  +Q^{2}\left(  x,y\right)  =\left\vert \mathbf{n}%
_{x,y}\right\vert ^{2}\left\vert x-y\right\vert ^{2}. \label{r13}%
\end{equation}
Hence, since $\left\vert \mathbf{n}_{x,y}\right\vert \leq2$,%
\[
\left\vert P\left(  x,y\right)  Q\left(  x,y\right)  \right\vert \leq\dfrac
{1}{2}\left(  P^{2}\left(  x,y\right)  +Q^{2}\left(  x,y\right)  \right)
=\dfrac{1}{2}\left\vert \mathbf{n}_{x,y}\right\vert ^{2}\left\vert
x-y\right\vert ^{2}\leq2\left\vert x-y\right\vert ^{2},
\]
which gives%
\[
\left\vert \frac{1}{2}\tan\theta_{r}P\left(  x,y\right)  Q\left(  x,y\right)
\right\vert \leq\tan\theta_{r}\left\vert x-y\right\vert ^{2}.
\]
Adding $\left\vert x-y\right\vert ^{2}$, from (\ref{r10}), the conclusion
follows.\hfill
\end{proof}

\begin{remark}
The identity (\ref{r13}) is the only place where $d=2$ is used, and it is
decisive. In $\mathbb{R}^{d}$ with $d\geq3$ the pair $\left(  J\mathbf{n}%
_{x,y},\mathbf{n}_{x,y}\right)  $ no longer spans the space, and the product
$PQ$ is no longer controlled by $\left\vert x-y\right\vert ^{2}$ with the
constant $\frac{1}{2}\left\vert \mathbf{n}_{x,y}\right\vert ^{2}$. Therefore,
the construction collapses. This is the precise sense in which the oblique
reflection driven by a rotation find it's place into a two-dimensional
obstacle problem.
\end{remark}

\begin{lemma}
[Cancellation of the tangential term]\label{rot-cancel}There is an absolute
constant $c_{0}\leq16$ such that, for all $r\in\left[  0,T\right]  $, $x\in
bd\left(  K\right)  $ and $y\in K$,
\begin{equation}
\left\langle \nabla_{x}\Phi_{r}\left(  x,y\right)  ,\Theta_{r}\mathbf{n}%
\left(  x\right)  \right\rangle \geq\dfrac{2\cos2\theta_{r}}{\cos\theta_{r}%
}\left\langle x-y,\mathbf{n}\left(  x\right)  \right\rangle -c_{0}\,b_{\ast
}\tan\theta_{r}\left\vert x-y\right\vert ^{2}. \label{r14}%
\end{equation}
By the symmetry of $\Phi_{r}$ the same estimate holds for $\left\langle
\nabla_{y}\Phi_{r}\left(  x,y\right)  ,\Theta_{r}\mathbf{n}\left(  y\right)
\right\rangle $ when $y\in bd\left(  K\right)  $ and $x\in K$, with
$\left\langle y-x,\mathbf{n}\left(  y\right)  \right\rangle $ taking the role
of $\left\langle x-y,\mathbf{n}\left(  x\right)  \right\rangle $.
\end{lemma}

\begin{proof}
\footnote{From now on, when there is no risk of confusion, we will use the
terms without writing their arguments. For example, we use $P=P\left(
x,y\right)  ,$ $Q=Q\left(  x,y\right)  ,$ $\theta=\theta_{r}~$, and so on.}Let
us denote $\lambda:=-\frac{1}{2}\tan\theta$, $\mathrm{n}:=\mathbf{n}\left(
x\right)  $, $\mathrm{m}:=\mathbf{n}\left(  y\right)  $, $g:=\Theta
_{r}\mathrm{n}=\cos\theta\mathrm{n}+\sin\theta J\mathrm{n}$, and
\[
a:=\left\langle x-y,\mathrm{n}\right\rangle ,\quad b:=\left\langle
x-y,J\mathrm{n}\right\rangle ,\quad a^{\prime}:=\left\langle x-y,\mathrm{m}%
\right\rangle ,\quad b^{\prime}:=\left\langle x-y,J\mathrm{m}\right\rangle .
\]
We have $P=a+a^{\prime}$, $Q=b+b^{\prime}~$and $a^{2}+b^{2}=\left\vert
x-y\right\vert ^{2}$, since $\{\mathrm{n,}J\mathrm{n}\}$ is an orthonormal
base on $bd\left(  K\right)  ,$ $\left\vert \mathrm{n}\right\vert =1$. By
(\ref{r8}), $\left\vert a^{\prime}-a\right\vert +\left\vert b^{\prime
}-b\right\vert \leq2b_{\psi}\left\vert x-y\right\vert ^{2}$.

Denote $c:=\left\langle \mathrm{m},\mathrm{n}\right\rangle $ and $s:=\langle
J\mathrm{m},\mathrm{n}\rangle$. According to Lemma \ref{K-geometry}$~\left(
b\right)  ,\left(  c\right)  $,
\[
0\leq1-c\leq b_{\ast}\left\vert x-y\right\vert \quad\quad\text{and}\quad
\quad\left\vert s\right\vert \leq b_{\psi}\left\vert x-y\right\vert .
\]
Our aim is to obtain now the gradient $\nabla_{x}\Phi_{r}$. Differentiating,
$\nabla_{x}P=\mathbf{n}_{x,y}+D\mathbf{n}\left(  x\right)  \left(  x-y\right)
$ and $\nabla_{x}Q=J\mathbf{n}_{x,y}-D\mathbf{n}\left(  x\right)  J\left(
x-y\right)  $ and one obtain:%
\begin{equation}
\nabla_{x}\Phi_{r}=\underbrace{2\left(  x-y\right)  }_{\mathrm{I}}%
+\underbrace{\lambda\left(  Q\mathbf{n}_{x,y}+PJ\mathbf{n}_{x,y}\right)
}_{\mathrm{II}}+\underbrace{\lambda\left(  QD\mathbf{n}\left(  x\right)
\left(  x-y\right)  -PD\mathbf{n}\left(  x\right)  J\left(  x-y\right)
\right)  }_{\mathrm{III}}. \label{grad of FI}%
\end{equation}
Consider each term separately and we multiply them with $g=\Theta
_{r}\mathrm{n}=\Theta_{r}\mathbf{n}\left(  x\right)  $ in order to get the
left hand side (LHS) from the inequality (\ref{r14}). We have:

\begin{itemize}
\item $\left\langle \mathrm{I},g\right\rangle =2\left\langle x-y,\cos
\theta\mathrm{n}+\sin\theta J\mathrm{n}\right\rangle =2\cos\theta\left\langle
x-y,\mathrm{n}\right\rangle +2\sin\theta\left\langle x-y,J\mathrm{n}%
\right\rangle =2\cos\theta\cdot a+2\sin\theta\cdot b$.

\item We know that $J^{T}=-J$, $J^{2}=-I$ and $\left\vert J\mathrm{n}%
\right\vert =1$. Also, $\langle J\mathrm{n},\mathrm{n}\rangle=0$, $\langle
J\mathrm{m},J\mathrm{n}\rangle=\langle\mathrm{m},\mathrm{n}\rangle=:c$ and
$\langle\mathrm{m},J\mathrm{n}\rangle=:-s$. As consequence,$\smallskip$

$\left\langle \mathbf{n}_{x,y},g\right\rangle =\left\langle n+m,\cos
\theta\mathrm{n}+\sin\theta J\mathrm{n}\right\rangle =\cos\theta
\cdot(1+c)-\sin\theta\cdot s~$;$\smallskip$

$\left\langle J\mathbf{n}_{x,y},g\right\rangle =\left\langle Jn+Jm,\cos
\theta\mathrm{n}+\sin\theta J\mathrm{n}\right\rangle =\cos\theta\cdot
s+\sin\theta\cdot(1+c)~$;

Multiply the above relations with $Q=b+b^{\prime}$ and, respectively, with
$P=a+a^{\prime}$, we obtain$\smallskip$

$\left\langle \mathrm{II},g\right\rangle =\lambda\left(  b+b^{\prime}\right)
\left(  2\cos\theta-\cos\theta\cdot(1-c)-\sin\theta\cdot s\right)
+\lambda\left(  a+a^{\prime}\right)  \left(  2\sin\theta+\cos\theta\cdot
s-\sin\theta\cdot(1-c)\right)  $

From Lemma \ref{K-geometry}~$\left(  b\right)  ,\left(  c\right)  $, we have%
\[
\left\{
\begin{array}
[c]{l}%
0\leq1-c\leq b_{\ast}\left\vert x-y\right\vert ,\quad\quad\left\vert
s\right\vert \leq b_{\psi}\left\vert x-y\right\vert \leq b_{\ast}\left\vert
x-y\right\vert \medskip\\
\left\vert a^{\prime}-a\right\vert +\left\vert b^{\prime}-b\right\vert
\leq2b_{\psi}\left\vert x-y\right\vert ^{2}.
\end{array}
\right.
\]
Therefore, replacing $b+b^{\prime}$ by $2b$ and $a+a^{\prime}$ by $2a$ in
$\left\langle \mathrm{II},g\right\rangle $ produces an error bounded by:%
\[
2\left\vert \lambda\right\vert \left(  2\cos\theta+2\sin\theta\right)
\cdot2b_{\psi}\left\vert x-y\right\vert ^{2}\leq8\left\vert \lambda\right\vert
b_{\psi}\left\vert x-y\right\vert ^{2}.
\]
The terms which contain the factors $s$ and $1-c$ are also bounded, due to the
fact that we have $\left\vert a\right\vert ,\left\vert b\right\vert
\leq\left\vert x-y\right\vert $ and $\left\vert a+a^{\prime}\right\vert
,\left\vert b+b^{\prime}\right\vert \leq3\left\vert x-y\right\vert $ by%
\[
2\cdot3\left\vert x-y\right\vert \cdot\left\vert \lambda\right\vert \cdot
b_{\ast}\left\vert x-y\right\vert =6\left\vert \lambda\right\vert b_{\ast
}\left\vert x-y\right\vert ^{2}.
\]
As consequence,
\[
\left\langle \mathrm{II},g\right\rangle \geq4\lambda\cos\theta\cdot
b+4\lambda\sin\theta\cdot a-14\left\vert \lambda\right\vert b_{\ast}\left\vert
x-y\right\vert ^{2}.
\]

\item Finally, for the third term from (\ref{grad of FI}), we have

$\left\vert \left\langle \mathrm{III},g\right\rangle \right\vert
\leq\left\vert \lambda\right\vert \left(  \left\vert P\right\vert +\left\vert
Q\right\vert \right)  b_{\psi}\left\vert x-y\right\vert \leq\left\vert
\lambda\right\vert \left(  2\left\vert x-y\right\vert +2\left\vert
x-y\right\vert \right)  b_{\psi}\left\vert x-y\right\vert \leq2\left\vert
\lambda\right\vert b_{\ast}\left\vert x-y\right\vert ^{2}.$
\end{itemize}

In order to conclude the proof, from the formulas of $\left\langle
\mathrm{I},g\right\rangle $, $\left\langle \mathrm{II},g\right\rangle $, the
coefficient of $b$ is:%
\[
2\sin\theta+4\lambda\cos\theta=0,\quad\text{due to the choice of }%
\lambda=-\frac{1}{2}\tan\theta.
\]
The coefficient of $a$ becomes, also due to the choice of $\lambda$,%
\[
2\cos\theta+4\lambda\sin\theta=2\cos\theta-2\dfrac{\sin^{2}\theta}{\cos\theta
}=\dfrac{2\cos2\theta}{\cos\theta}%
\]
Since $16\left\vert \lambda\right\vert =8\tan\theta$, we clearly observe that
(\ref{r14}) follows, for $c_{0}=16$. Indeed, if we count the errors order
brought by second and third group, we have:%
\[
14\left\vert \lambda\right\vert b_{\ast}\left\vert x-y\right\vert
^{2}+2\left\vert \lambda\right\vert b_{\ast}\left\vert x-y\right\vert
^{2}=16b_{\ast}\left\vert x-y\right\vert ^{2}.
\]
The proof of Lemma \ref{rot-cancel} is now complete.\hfill
\end{proof}

\begin{remark}
[Provenance of $\Phi$: a classical device, in two-component form]%
\label{rot-provenance}The function $\Phi_{r}$ is a quadratic form in $x-y$,
with an $\left(  x,y\right)  $-dependent coefficient. Indeed, $\left\langle
x-y,a\right\rangle \left\langle x-y,b\right\rangle =\left\langle
\mathrm{sym}\left(  a\otimes b\right)  \left(  x-y\right)  ,x-y\right\rangle $
with $\mathrm{sym}\left(  a\otimes b\right)  :=\frac{1}{2}\left(  ab^{\ast
}+ba^{\ast}\right)  $, so
\begin{equation}
\Phi_{r}\left(  x,y\right)  =\left\langle A_{r}\left(  x,y\right)  \left(
x-y\right)  ,x-y\right\rangle \quad\text{and}\quad A_{r}\left(  x,y\right)
=I-\dfrac{1}{2}\tan\theta_{r}\mathrm{sym}\left(  \mathbf{n}_{x,y}\otimes
J\mathbf{n}_{x,y}\right)  , \label{r14a}%
\end{equation}
a symmetric matrix. On the diagonal $\mathbf{n}_{x,x}=2\mathbf{n}\left(
x\right)  $ and
\begin{equation}
A_{r}\left(  x,x\right)  =I-\tan\theta(r)\left(  \mathbf{n}\otimes
J\mathbf{n}+J\mathbf{n}\otimes\mathbf{n}\right)  \quad\text{and}\quad
A_{r}\left(  x,x\right)  \Theta_{r}\mathbf{n}\left(  x\right)  =\dfrac
{\cos2\theta_{r}}{\cos\theta_{r}}\mathbf{n}\left(  x\right)  , \label{r14b}%
\end{equation}
the second identity following from $\left\langle J\mathbf{n},\mathbf{n}%
\right\rangle =0$ and $\left\vert \mathbf{n}\right\vert =\left\vert
J\mathbf{n}\right\vert =1$. Thus (\ref{r14b}) \textit{is} exactly Lemma
\ref{rot-cancel} in algebraic form, and it produces the constant $2\cos
2\theta/\cos\theta$. In the bidimensional framework, $\{\mathbf{n}%
,J\mathbf{n}\}$, the eigenvalues of $A_{r}\left(  x,x\right)  $ are $1\pm
\tan\theta$, so that the inequality $A_{r}\geq\alpha_{\theta}I$ is precisely
$\left(  C_{1}\right)  $, and (\ref{r13}) is Parseval's identity in the
orthonormal basis $\left\{  \left\vert \mathbf{n}_{x,y}\right\vert
^{-1}\mathbf{n}_{x,y},\left\vert \mathbf{n}_{x,y}\right\vert ^{-1}%
J\mathbf{n}_{x,y}\right\}  $.

Written in this way, $\Phi$ is recognizable as the two-point form of a
classical device. We have a symmetric, uniformly elliptic matrix field $A$
with
\[
A\left(  x\right)  \gamma\left(  x\right)  =c\left(  x\right)  \mathbf{n}%
\left(  x\right)  ,\qquad c\left(  x\right)  >0,\qquad\text{on }bd\left(
K\right)  ,
\]
$\gamma$ denoting the direction of reflection. Only the sign of the factor $c$
is used; here, by (\ref{r14b}), $c=\cos2\theta/\cos\theta$, and $c>0$ is
exactly $\left(  C_{1}\right)  $, whereas Barles and DaLio
\cite{Barles/DaLio:06} normalizes $c\equiv1$ - which, as Remark
\ref{rot-C1-artifact} will soon show, is possible for every $\theta\in\left(
0,\pi/4\right)  $. Such an $A$ is used inside the function $\left\langle
A\left(  x-y\right)  ,x-y\right\rangle $ in order to obtain the disappearance
from $\left\langle \nabla_{x}\Phi,\gamma\left(  x\right)  \right\rangle $ of
the first order tangential term. In the viscosity solution literature for
oblique derivative and nonlinear Neumann problems this is standard: it is
assumption $(H3a)$ from Barles and Da Lio \cite{Barles/DaLio:06}
(\textquotedblleft there exists a Lipschitz continuous function
$A:O\rightarrow S^{n}$ with $A\geq c_{0}Id$, for some $c_{0}>0$ such that
$A(x)\gamma(x)=n(x)$ for every $x\in\partial O$\textquotedblright), where it
is attributed to P.-L.\ Lions and to Lions and Sznitman
\cite{Lions/Sznitman:84} and where the resulting boundary estimate is exactly
(\ref{r14}). The test function from Barles \cite{Barles:99} contains the same
product between a normal increment and an oblique increment. Also, when
dealing with the Skorokhod problem, the auxiliary function $g$ of Dupuis and
Ishii \cite{Dupuis/Ishii:93} obeys the same specification, although no closed
form is given there. The planar parametrization $\Theta_{r}\mathbf{n}%
=\cos\theta\mathbf{n}+\sin\theta J\mathbf{n}$ is classical as well; it is
written under the form $v_{\theta}=\mathbf{n}+\tan\theta\mathbf{t}$,
$\mathbf{t}=J\mathbf{n}$, in Burdzy, Chen, Marshall and Ramanan
\cite{Burdzy/Chen/Marshall/Ramanan:17}.

We therefore claim no novelty for $\Phi$ itself. What is specific here is its
realization and its use: the explicit two-dimensional closed form, the
symmetrization through $\mathbf{n}\left(  x\right)  +\mathbf{n}\left(
y\right)  $, rather than through a midpoint or a mollification, which makes
$\Phi_{r}\left(  x,y\right)  =\Phi_{r}\left(  y,x\right)  $ and is what allows
both arguments to reach $bd\left(  K\right)  $, as they must when a backward
equation is compared with itself. We have the transfer of the device to a
multivalued backward equation driven by a martingale on a general filtration,
where the second order term of Lemma \ref{rot-hess} and the weight of
Subsection \ref{rot-conv} are the real difficulty.
\end{remark}

\subsection{Rotation angle compatibility constraints\label{rot-constraints}}

We introduce some constraints on the rotation angle, in order to make the
upcoming results valid. The proofs below use exactly the following two
properties of $\theta$:%
\begin{equation}
\left\{
\begin{array}
[c]{ll}%
\left(  C_{1}\right)  : & \alpha_{\theta}:=1-\max\limits_{r\in\left[
0,T\right]  }\tan\theta_{r}>0;\medskip\\
\left(  C_{2}\right)  : & \left\langle \nabla_{x}\Phi_{r}\left(  x,y\right)
,\Theta_{r}\mathbf{n}\left(  x\right)  \right\rangle \geq0,\quad\text{for all
}r\in\left[  0,T\right]  ,\ x\in bd\left(  K\right)  ,\ y\in K.
\end{array}
\right.  \label{r15}%
\end{equation}
Condition $\left(  C_{1}\right)  $ states that $\operatorname{Im}\left(
\theta\right)  =\left(  0,\pi/4\right)  $. If we revisit Lemma \ref{rot-two},
the condition makes, in fact, $\Phi_{r}$ equivalent to $\left\vert
x-y\right\vert ^{2}$. More precisely,%
\begin{equation}
\alpha_{\theta}\left\vert x-y\right\vert ^{2}\leq\Phi_{r}\left(  x,y\right)
\leq2\left\vert x-y\right\vert ^{2}, \label{r16}%
\end{equation}
and it is exactly the condition $\cos2\theta>0$, under which the surviving
normal term from Lemma \ref{rot-cancel} has the right sign. Condition $\left(
C_{2}\right)  $ is the quantitative form of the following compatibility
assumption: the curvature of $bd\left(  K\right)  $ is dominating the
tangential slide. We now present a result which gives sufficient conditions
for (\ref{r15}) to hold.

\begin{lemma}
[Sufficient criteria for $\left(  C_{1}\right)  $ and $\left(  C_{2}\right)
$]\label{rot-suff}The following assertions are valid.

\begin{itemize}
\item[$\left(  a\right)  $] If $a_{\psi}\cos2\theta_{r}\geq16b_{\ast}%
\sin\theta_{r}~$, for all $r\in\left[  0,T\right]  $, then $\left(
C_{2}\right)  $ holds. Moreover, the inequality $a_{\psi}\leq b_{\psi}\leq
b_{\ast}~$, implies $\tan\theta_{r}\leq\frac{1}{8}$, and, therefore, condition
$\left(  C_{1}\right)  $ also holds.

\item[$\left(  b\right)  $] If we consider $K$ to be the closed Euclidean ball
$K=\bar{B}\left(  z_{0},R\right)  $ and $\psi\left(  x\right)  =\frac{1}%
{2R}\left(  \left\vert x-z_{0}\right\vert ^{2}-R^{2}\right)  $, then $\left(
C_{2}\right)  $ holds under the constraint $\tan\theta_{r}+\tan^{2}\theta
_{r}\leq1$, that is, $0<\tan\theta_{r}\leq\tfrac{1}{2}\left(  \sqrt
{5}-1\right)  $, i.e. $\theta\leq31.7^{\circ}$.
\end{itemize}
\end{lemma}

\begin{proof}
$\left(  a\right)  $ Combining Lemma \ref{rot-cancel} with Lemma
\ref{K-geometry}$\left(  d\right)  $, we obtain:%
\[
\left\langle \nabla_{x}\Phi_{r},\Theta_{r}\mathbf{n}\left(  x\right)
\right\rangle \geq\dfrac{2\cos2\theta_{r}}{\cos\theta_{r}}\left\langle
x-y,,\mathbf{n}\left(  x\right)  \right\rangle -16b_{\ast}\tan\theta
_{r}\left\vert x-y\right\vert ^{2},
\]
but $\left\langle x-y,\mathbf{n}\left(  x\right)  \right\rangle \geq
\tfrac{a_{\psi}}{2}\left\vert x-y\right\vert ^{2}\geq0.$ Substituting,%
\[
\left\langle \nabla_{x}\Phi_{r},\Theta_{r}\mathbf{n}\left(  x\right)
\right\rangle \geq\left(  \dfrac{a_{\psi}\cos2\theta}{\cos\theta}-16b_{\ast
}\tan\theta\right)  \left\vert x-y\right\vert ^{2},
\]
and it suffices to multiply by $\cos\theta$ to obtain $\left(  a\right)  $,
under the assumption $a_{\psi}\cos2\theta_{r}\geq16b_{\ast}\sin\theta_{r}~$,
for all $r\in\left[  0,T\right]  $.

$\left(  b\right)  $ Take $z_{0}=0$ and $R=1$, the inequality being invariant
under the linear transformation $x\mapsto z_{0}+Rx$. Then $\mathbf{n}\left(
x\right)  =x$ on $bd\left(  K\right)  =bd(\bar{B}\left(  0,1\right)  )$,
$\mathbf{n}_{x,y}=x+y$, and for $\left\vert x\right\vert =1$,%
\begin{equation}
P=\left\vert x\right\vert ^{2}-\left\vert y\right\vert ^{2}=1-\left\vert
y\right\vert ^{2},\qquad Q=\left\langle x-y,J\left(  x+y\right)  \right\rangle
=2\left\langle x,Jy\right\rangle =:2w. \label{forms of P and Q}%
\end{equation}
Observe that $\nabla_{x}P=2x$ and $\nabla_{x}Q=2Jy$, which implies
\[
\nabla_{x}\Phi_{r}=2\left(  x-y\right)  -\tan\theta\left(  Qx+PJy\right)  .
\]
The intention is to make the scalar product with $g=\Phi_{r}=\cos\theta\cdot
x+\sin\theta\cdot Jx$. Denote $c:=\left\langle x,y\right\rangle $ and, using
$\left\langle x,g\right\rangle =\cos\theta$, we get, due to the skew-symmetry
of $J$ and (\ref{forms of P and Q}),%
\[
\left\{
\begin{array}
[c]{l}%
\langle Jy,g\rangle=\left\langle Jy,\cos\theta\cdot x+\sin\theta\cdot
Jx\right\rangle =w\cos\theta+c\sin\theta\quad\text{and}\medskip\\
\langle x-y,g\rangle=\left\langle x-y,\cos\theta\cdot x+\sin\theta\cdot
Jx\right\rangle =(1-c)\cos\theta+w\sin\theta.
\end{array}
\right.
\]
Observe that the two terms $\pm2w\sin\theta$ cancel and it yields%
\[
\left\langle \nabla_{x}\Phi_{r},g\right\rangle =2\cos\theta\cdot
(1-c)-\tan\theta\cdot\left(  1-\left\vert y\right\vert ^{2}\right)  \left(
w\cos\theta+c\sin\theta\right)  .
\]
The equality $y=cx-wJx$ implies $c^{2}+w^{2}=\left\vert y\right\vert ^{2}%
\leq1$, so $1-\left\vert y\right\vert ^{2}\leq1-c^{2}$. If we divide by
$\cos\theta$, it is enough to have%
\[
2\left(  1-c\right)  \geq\left(  1-c^{2}\right)  \left(  \tan\theta\cdot
w+\tan^{2}\theta\cdot c\right)  ,
\]
which, when the RHS member is positive ($1>c$), implies that is is enough to
ask, as a sufficient condition,%
\[
\left(  1+c\right)  \left(  \tan\theta\cdot w+\tan^{2}\theta\cdot c\right)
\leq2\left(  \tan\theta+\tan^{2}\theta\right)  \leq2.
\]
The proof is now complete.\hfill
\end{proof}

\begin{remark}
\label{rot-sharp}Conditions $\left(  C_{1}\right)  $--$\left(  C_{2}\right)  $
are sufficient, not necessary, and the constant $16\,b_{\ast}$ of Lemma
\ref{rot-suff}$\left(  a\right)  $, obtained by a rough counting of the error
terms, is far from being an optimal one. Note also that $\left(  C_{2}\right)
$ forces $bd\left(  K\right)  $ to be uniformly convex, through Lemma
\ref{K-geometry}$\left(  d\right)  $, whereas oblique reflection on a
half-space is harmless: the method pays for its generality.
\end{remark}

\begin{remark}
[$\left(  C_{1}\right)  $ is not intrinsic]\label{rot-C1-artifact}The
restriction $\tan\theta<1$ (i.e., the $\left(  C_{1}\right)  $ condition) is
an artifact which assures that the leading term of our explicit given function
$\Phi$ is exactly $\left\vert x-y\right\vert ^{2}$. In our bidimensional
framework, a symmetric matrix $A$ satisfying $A\geq c_{0}I$ and $A\Theta
_{r}\mathbf{n}=\mathbf{n}$ exists for every angle $\theta\in\left(
0,\pi/2\right)  $. More precisely, with respect tot the orthonormal base
$\{\mathbf{n},J\mathbf{n}\}$, one can construct the nonsingular matrix%
\[
A:=\left(
\begin{array}
[c]{cc}%
\dfrac{1+\tan\theta\sin\theta}{\cos\theta}\medskip & -\tan\theta\\
-\tan\theta & 1
\end{array}
\right)  ,\quad\quad\text{with}\quad\quad\det A=\dfrac{1}{\cos\theta}>0.
\]
For $A$, the smallest eigenvalue behaves like $\cos\theta$, as $\theta
\nearrow\pi/2$: it degenerates, but it does not vanish at $\pi/4$. Replacing
$I-\tan\theta\left(  \mathbf{n}\otimes J\mathbf{n}+J\mathbf{n}\otimes
\mathbf{n}\right)  $ from relation (\ref{r14b}) with this $A$, symmetrized in
$\left(  x,y\right)  $ and truncated as in (\ref{r9}), should remove the
constraint $\left(  C_{1}\right)  $ altogether, at the heavier price of an
ellipticity constant degenerating with $\theta$ and of heavier second order
estimates. We have not carried this out.
\end{remark}

Under the assumption (\ref{r15}-$\left(  C_{1}\right)  $), we provide some
estimates which permit the control of the second order terms from the
approximating equations associated to our obstacle problem. This is the most
difficult to control, since we need a small fraction as $\alpha_{\theta}/2$,
from (\ref{r17}), for the absorption of the coupling terms. It remains a
second order defect, proportional with $\left\vert x-y\right\vert ^{2}\left(
\left\vert h\right\vert ^{2}+|h^{\prime}|^{2}\right)  $, multiplied by an
explicit given positive constant, $\kappa_{\theta}$~.

\begin{lemma}
[Second order terms lower bound]\label{rot-hess}Assume $\left(  C_{1}\right)
$ takes places and set%
\begin{equation}
\theta_{\max}:=\max\limits_{r\in\left[  0,T\right]  }\theta\left(  r\right)
,\qquad\kappa_{\theta}:=2b_{\psi}^{2}\tan\theta_{\max}\left(  1+\dfrac
{8\tan\theta_{\max}}{\alpha_{\theta}}\right)  . \label{r16a}%
\end{equation}
Then, for all $x,y\in\mathbb{R}^{2}$, all $r\in\left[  0,T\right]  $ and all
direction of variation $h,h^{\prime}\in\mathbb{R}^{2}$,
\begin{equation}
\dfrac{1}{2}D_{\left(  x,y\right)  }^{2}\Phi_{r}\left(  x,y\right)  \left[
\left(  h,h^{\prime}\right)  ,\left(  h,h^{\prime}\right)  \right]  \geq
\dfrac{\alpha_{\theta}}{2}\left\vert h-h^{\prime}\right\vert ^{2}%
-\kappa_{\theta}\left\vert x-y\right\vert ^{2}\left(  \left\vert h\right\vert
^{2}+|h^{\prime}|^{2}\right)  . \label{r17}%
\end{equation}

\end{lemma}

\begin{proof}
Let us denote $u:=x-y$, $k:=h-h^{\prime}$ and define the directional
derivative of $P\left(  x,y\right)  =\left\langle x-y,\mathbf{n}\left(
x\right)  +\mathbf{n}\left(  y\right)  \right\rangle $ on the direction
$\left(  h,h^{\prime}\right)  :$%
\[
\mathbb{D}P:=DP\cdot\left(  h,h^{\prime}\right)  =\underbrace{\left\langle
k,\mathbf{n}_{x,y}\right\rangle }_{p_{1}}+\underbrace{\left\langle
u,D\mathbf{n}\left(  x\right)  h+D\mathbf{n}\left(  y\right)  h^{\prime
}\right\rangle }_{p_{2}}.
\]
In a similar manner, let us construct $\mathbb{D}Q:=q_{1}+q_{2}~,$ with
$q_{1}:=\left\langle k,J\mathbf{n}_{x,y}\right\rangle $:%
\[
\mathbb{D}Q:=DQ\cdot\left(  h,h^{\prime}\right)  =\underbrace{\left\langle
k,J\mathbf{n}_{x,y}\right\rangle }_{q_{1}}+\underbrace{\left\langle u,J\left(
D\mathbf{n}\left(  x\right)  h+D\mathbf{n}\left(  y\right)  h^{\prime}\right)
\right\rangle }_{q_{2}}.
\]
The second derivative of the product $PQ$ is given by%
\[
D^{2}\left(  PQ\right)  \left[  \left(  h,h^{\prime}\right)  \right]
^{2}=2\left(  \mathbb{D}P\right)  \left(  \mathbb{D}Q\right)  +QD^{2}P\left[
\left(  h,h^{\prime}\right)  ^{2}\right]  +PD^{2}Q\left[  \left(  h,h^{\prime
}\right)  ^{2}\right]  .
\]
By (\ref{r9}), $\max\{\left\vert p_{2}\right\vert ,\left\vert q_{2}\right\vert
\}\leq b_{\psi}\left\vert u\right\vert \left(  \left\vert h\right\vert
+|h^{\prime}|\right)  $, and%
\[
D^{2}P\left[  \left(  h,h^{\prime}\right)  ^{2}\right]  =2\left\langle
k,D\mathbf{n}\left(  x\right)  h+D\mathbf{n}\left(  y\right)  h^{\prime
}\right\rangle +\left\langle u,D^{2}\mathbf{n}\left(  x\right)  \left[
h,h\right]  +D^{2}\mathbf{n}\left(  y\right)  \left[  h^{\prime},h^{\prime
}\right]  \right\rangle .
\]
By estimating it, we obtain%
\[
\left\vert D^{2}P\left[  \left(  h,h^{\prime}\right)  ^{2}\right]  \right\vert
\leq2b_{\psi}\left\vert k\right\vert \left(  \left\vert h\right\vert
+|h^{\prime}|\right)  +b_{\psi}\left\vert u\right\vert \left(  \left\vert
h\right\vert ^{2}+|h^{\prime}|^{2}\right)  ,
\]
and, similarly, for $Q$. Since
\begin{equation}
\dfrac{1}{2}D^{2}\Phi_{r}\left[  \left(  h,h^{\prime}\right)  ^{2}\right]
=\left\vert k\right\vert ^{2}-\dfrac{\tan\theta}{4}\left(  2\mathbb{D}%
P\mathbb{D}Q+QD^{2}P\left[  \left(  h,h^{\prime}\right)  ^{2}\right]
+PD^{2}Q\left[  \left(  h,h^{\prime}\right)  ^{2}\right]  \right)  ,
\label{sec order term to est}%
\end{equation}
we estimate each of the three terms separately. First, exactly as in the proof
of Lemma \ref{rot-two},
\[
\left\vert p_{1}q_{1}\right\vert \leq\frac{1}{2}\left(  p_{1}^{2}+q_{1}%
^{2}\right)  =\frac{1}{2}\left\vert \mathbf{n}_{x,y}\right\vert ^{2}\left\vert
k\right\vert ^{2}\leq2\left\vert k\right\vert ^{2}\quad\Longrightarrow
\quad\left\vert \frac{\tan\theta}{2}p_{1}q_{1}\right\vert \leq\tan
\theta\left\vert k\right\vert ^{2}.
\]
Substract it from $\left\vert k\right\vert ^{2}$ and we obtain a positive
coefficient $\left(  1-\tan\theta\right)  \left\vert k\right\vert ^{2}%
\geq\alpha_{\theta}\left\vert k\right\vert ^{2}$, inequality which will be
used to close the proof.

\noindent Secondly, using $\max\{\left\vert p_{1}\right\vert ,\left\vert
q_{1}\right\vert \}\leq2\left\vert k\right\vert $, we obtain%
\[
\dfrac{\tan\theta}{2}\left\vert p_{1}q_{2}+p_{2}q_{1}+p_{2}q_{2}\right\vert
\leq\tan\theta\left(  2b_{\psi}\left\vert u\right\vert \left\vert k\right\vert
\left(  \left\vert h\right\vert +|h^{\prime}|\right)  +\dfrac{b_{\psi}^{2}}%
{2}\left\vert u\right\vert ^{2}\left(  \left\vert h\right\vert +|h^{\prime
}|\right)  ^{2}\right)  .
\]
Thirdly, by $\max\{\left\vert P\right\vert ,\left\vert Q\right\vert
\}\leq2\left\vert u\right\vert $, it follows%
\[
\dfrac{\tan\theta}{4}\left\vert QD^{2}P+PD^{2}Q\right\vert \leq\tan
\theta\left(  2b_{\psi}\left\vert u\right\vert \left\vert k\right\vert \left(
\left\vert h\right\vert +|h^{\prime}|\right)  +b_{\psi}\left\vert u\right\vert
^{2}\left(  \left\vert h\right\vert ^{2}+|h^{\prime}|^{2}\right)  \right)  .
\]
We add the above estimates, use $\left(  \left\vert h\right\vert +|h^{\prime
}|\right)  ^{2}\leq2\left\vert h\right\vert ^{2}+2|h^{\prime}|^{2}$ ,
$b_{\psi}\geq1$ and, finally, we invoke the Young's inequality, with
$\varepsilon=\alpha_{\theta}/2$, applied under the form%
\[
4b_{\psi}\tan\theta\left\vert u\right\vert \left\vert k\right\vert \left(
\left\vert h\right\vert +|h^{\prime}|\right)  \leq\varepsilon\left\vert
k\right\vert ^{2}+\frac{4b_{\psi}^{2}\tan^{2}\theta}{\varepsilon}\left\vert
u\right\vert ^{2}\left(  \left\vert h\right\vert +|h^{\prime}|\right)  ^{2}%
\]
From (\ref{sec order term to est}) we obtain%
\[%
\begin{array}
[c]{l}%
\dfrac{1}{2}D^{2}\Phi_{r}\left[  \left(  h,h^{\prime}\right)  ^{2}\right]
\geq\left(  1-\tan\theta\right)  \left\vert k\right\vert ^{2}-4b_{\psi}%
\tan\theta\left\vert u\right\vert \left\vert k\right\vert \left(  \left\vert
h\right\vert +|h^{\prime}|\right)  \medskip\\
\quad\quad\quad\quad\quad-2b_{\psi}^{2}\tan\theta\left\vert u\right\vert
^{2}\left(  \left\vert h\right\vert ^{2}+|h^{\prime}|^{2}\right)  \medskip\\
\quad\quad\quad\quad\quad\geq\left(  1-\tan\theta\right)  \left\vert
k\right\vert ^{2}-\varepsilon\left\vert k\right\vert ^{2}-\dfrac{8b_{\psi}%
^{2}\tan^{2}\theta}{\varepsilon}\left\vert x-y\right\vert ^{2}\left(
\left\vert h\right\vert ^{2}+|h^{\prime}|^{2}\right)  \medskip\\
\quad\quad\quad\quad\quad-2b_{\psi}^{2}\tan\theta\left\vert x-y\right\vert
^{2}\left(  \left\vert h\right\vert ^{2}+|h^{\prime}|^{2}\right)  \medskip\\
\quad\quad\quad\quad\quad=\left(  1-\tan\theta-\dfrac{\alpha_{\theta}}%
{2}\right)  \left\vert h-h^{\prime}\right\vert ^{2}-2b_{\psi}^{2}\tan
\theta\left(  \dfrac{8\tan\theta}{\alpha_{\theta}}+1\right)  \left\vert
x-y\right\vert ^{2}\left(  \left\vert h\right\vert ^{2}+|h^{\prime}%
|^{2}\right)
\end{array}
\]
Since $1-\tan\theta\geq\alpha_{\theta}~$, we deduce (\ref{r17}), with
$\kappa_{\theta}$ having exactly the form specified by (\ref{r16a}).\hfill
\end{proof}

\begin{remark}
\label{rot-kappa}Two features of (\ref{r17}) will prove to be crucial for
obtaining a priori estimates of the solution for the obstacle problem. First,
the whole second order defect is carried by the constant $\kappa_{\theta}$,
which is proportional to $\tan\theta_{\max}$ and vanishes with it. For
$\theta\equiv0$ the test function is $\left\vert x-y\right\vert ^{2}$, its
Hessian is constant, and there is no defect at all. Everything that will be
required below beyond $\left(  C_{1}\right)  $ and $\left(  C_{2}\right)  $ is
therefore a smallness, but acceptable, requirement on the obliqueness angle,
given the geometry of $K$ and the data. Secondly, (\ref{r17}) holds for every
pair $\left(  x,y\right)  \in\mathbb{R}^{2}\times\mathbb{R}^{2}$, due to the
role played by the global truncation on $\mathbb{R}^{2}$, formula (\ref{r9}).
\end{remark}

\subsection{Convergence of the Moreau-Yosida approximating
solutions\label{rot-conv}}

Two features brought in the classical approach by the presence of a symmetric
matrix are lost here, and both have to be compensated. The first is that
$\left(  Y^{\varepsilon},M^{\varepsilon}\right)  $ need not remain in a fixed
bounded set, so the estimate of Lemma \ref{rot-hess} must hold globally and
this is the reason why the normal field $\mathbf{n}$ of (\ref{r9}) was
truncated. The second, and the serious one, is that the It\^{o} correction
produced by $\Phi$ is not a fixed quadratic form. Indeed, according to Lemma
\ref{rot-hess} it leaves behind the term%
\begin{equation}
2\kappa_{\theta}{%
{\displaystyle\int_{t}^{T}}
}|Y_{r-}^{\varepsilon,\delta}|^{2}d\left(  \left[  M^{\varepsilon}\right]
+[M^{\delta}]\right)  _{r}~,\quad\text{with }Y_{r-}^{\varepsilon,\delta
}:=Y_{r-}^{\varepsilon}-Y_{r-}^{\delta}~,\label{r18a}%
\end{equation}
that is, a Gronwall inequality with respect to a random measure, whereas
Assumption \ref{H3} is an assumption in mean and admits only deterministic
weights (Lemma \ref{L1_gR}). The term (\ref{r18a}) is irreducible, since it
comes from $QD^{2}P+PD^{2}Q$ and from the cross products $p_{1}q_{2}%
+p_{2}q_{1},$ from the proof of Lemma \ref{rot-hess}. It survives even when
$\psi$ is quadratic.

A possible solution for a remedy is the one introduced by Chassagneux,
Nadtochiy and Richou \cite{Chassagneux/Nadtochiy/Richou:22}, and it consists
of two complementary steps. They introduced a technique based on a weighted
martingale-exponential approach. Since $\left[  M^{\varepsilon}\right]
+\left[  M^{\delta}\right]  $ is a random stochastic measure, one can not
apply the classical Gronwall inequality after we took the expectation, as we
make in the symmetric scenario. One should apply It\^{o}'s formula to the
weighted process $\Gamma_{t}\Phi_{t}\left(  Y_{t}^{\varepsilon},Y_{t}^{\delta
}\right)  $, where the weight
\[
\Gamma_{t}=\exp\left(  \nu_{\theta}\left(  \left[  M^{\varepsilon}\right]
_{t}+[M^{\delta}]_{t}\right)  +others\right)
\]
contains the random measure itself. The term (\ref{r18a}) is then cancelled
\textit{pathwise}, by construction. The price of the weight is the
integrability of $\Gamma_{T}$. It is paid by an exponential-moment inequality
for increasing processes, in Lemma \ref{garsia}.

\noindent In order to continue with the existence result, three new
assumptions are added, all of them on the given data.

\begin{condition}
[H2']\label{H2'}In addition to Assumption \ref{H2}, suppose $L,\ell\in
L^{\infty}\left(  0,T;\mathbb{R}_{+}\right)  $ and
\[
f_{\infty}^{2}:=\left\vert \left\vert {\int_{0}^{T}}\left\vert F\left(
r,0,0\right)  \right\vert ^{2}dr\right\vert \right\vert _{L^{\infty}\left(
\Omega\right)  }<+\infty.
\]

\end{condition}

\begin{remark}
No boundedness is required of the terminal datum: $\eta\in K,$ $\mathbb{P}%
$-a.s.\ and $K$ is bounded, so $\eta\in L^{\infty}$ automatically. Assumption
\ref{H2'} is the analogue of the boundedness hypotheses of Chassagneux,
Nadtochiy, and Richou \cite[Assumption 2.1]{Chassagneux/Nadtochiy/Richou:22}
and it is used only in Proposition \ref{rot-apriori}.
\end{remark}

\begin{condition}
[H3']\label{H3'}In addition to Assumption \ref{H3}, there exists
$C_{\mathcal{R}}>0$ such that, for every bounded nonnegative predictable
weight process $g$ and all $M,N\in\mathcal{M}_{d}^{2}~$,
\begin{equation}
\mathbb{E}{%
{\displaystyle\int_{0}^{T}}
}g_{r}\left\vert \mathcal{R}_{r}(M)-\mathcal{R}_{r}(N)\right\vert ^{2}dr\leq
C_{\mathcal{R}}^{2}\mathbb{E}{%
{\displaystyle\int_{0}^{T}}
}g_{r}\,d\left[  M-N\right]  _{r}. \label{lip_pR}%
\end{equation}

\end{condition}

Assumption \ref{H3'} strengthens Lemma \ref{L1_gR} by allowing bounded
nonnegative predictable weights instead of deterministic Borel weights. This
is necessary because the weighted estimate of Theorem \ref{rot-thm-conv}
involves the predictable process $\exp\left(  V_{r-}\right)  $, constructed
from the quadratic variations of the approximating martingales. The assumption
will therefore allow us to estimate%
\[
\mathbb{E}\int_{t}^{T}e^{V_{r-}}|\mathcal{R}_{r}^{\varepsilon,\delta}|^{2}dr,
\]
by the corresponding weighted quadratic variation of $M^{\varepsilon,\delta}$.

\begin{lemma}
\label{H3'-facts}$\left(  a\right)  $ Assumption \ref{H3'} implies Assumption
\ref{H3}-$\left(  jj\right)  $, with the same constant.\newline$\left(
b\right)  $ Under Assumption \ref{H3'}, for every stopping time $\tau\leq T$
and all $M,N\in\mathcal{M}_{d}^{2}~$,
\begin{equation}
\mathbb{E}^{\mathcal{F}_{\tau}}{%
{\displaystyle\int_{\tau}^{T}}
}\left\vert \mathcal{R}_{r}(M)-\mathcal{R}_{r}(N)\right\vert ^{2}dr\leq
C_{\mathcal{R}}^{2}\mathbb{E}^{\mathcal{F}_{\tau}}\left(  \left[  M-N\right]
_{T}-\left[  M-N\right]  _{\tau}\right)  ,\quad\mathbb{P}\text{-a.s.}
\label{lip_cR}%
\end{equation}
$\left(  c\right)  $ The martingale representation for the noise satisfy
Assumption \ref{H3'}, with the same constant $C_{\mathcal{R}}=1$ as in the
Assumption \ref{H3}.
\end{lemma}

\begin{proof}
$\left(  a\right)  $ Take $g=\mathbf{1}_{(s,t]}$ in (\ref{lip_pR}). The point
$\left(  a\right)  $ follows directly.

$\left(  b\right)  $ Let $A\in\mathcal{F}_{\tau}$ and $g_{r}:=\mathbf{1}%
_{A}\mathbf{1}_{(\tau,T]}\left(  r\right)  $, which is predictable and
bounded. Then (\ref{lip_pR}) can be written as%
\[
\mathbb{E}\left(  \mathbf{1}_{A}\int_{\tau}^{T}\left\vert \mathcal{R}%
_{r}(M)-\mathcal{R}_{r}(N)\right\vert ^{2}dr\right)  \leq C_{\mathcal{R}}%
^{2}\mathbb{E}\left(  \mathbf{1}_{A}\left(  \left[  M-N\right]  _{T}-\left[
M-N\right]  _{\tau}\right)  \right)  ,
\]
and $A$ is arbitrary in $\mathcal{F}_{\tau}$. Since $A$ is arbitrarily chosen,
$\left(  b\right)  $ follows.

$\left(  c\right)  $ Let $M=\int_{0}^{\cdot}Z\,dW$ and $N=\int_{0}^{\cdot
}Z^{\prime}dW$ with $W$ being a normal martingale. This implies $\mathcal{R}%
(M)-\mathcal{R}(N)=Z-Z^{\prime}$. Since we have the essential condition that
$g$ is predictable and bounded (from Assumption \ref{H3'}) and $\left[
M-N\right]  -\left\langle M-N\right\rangle $ is a uniform integrable
martingale,%
\[
\mathbb{E}\int_{0}^{T}g_{r}d\left[  M-N\right]  _{r}=\mathbb{E}\int_{0}%
^{T}g_{r}d\left\langle M-N\right\rangle _{r}=\mathbb{E}\int_{0}^{T}%
g_{r}\left\vert Z_{r}-Z_{r}^{\prime}\right\vert ^{2}dr.
\]
Hence (\ref{lip_pR}) holds with equality. The proof is now complete.\hfill
\bigskip
\end{proof}

Finally we impose the additional geometrical structure that $K$ admits a
uniform interior ball condition, whose radius $\rho$ dominates the tangential
slide. Fix $y_{0}\in\operatorname{int}K$ and $\rho>0$ such that $\bar
{B}\left(  y_{0},\rho\right)  \subset K$, and denote%
\[
R_{K}:=\max\limits_{z\in K}\left\vert z-y_{0}\right\vert \leq D_{K},\quad\quad
c_{K}:=\left\vert y_{0}\right\vert +R_{K},\quad\quad j:=D_{K}+2.
\]
We ask the following geometrical constraint%
\begin{equation}
\left(  C_{0}\right)  :\quad\quad\delta_{0}:=\min\limits_{r\in\left[
0,T\right]  }\left(  \rho\cos\theta_{r}-R_{K}\sin\theta_{r}\right)
>0,\quad\text{i.e.}\quad\tan\theta_{\max}<\dfrac{\rho}{R_{K}}. \label{C0}%
\end{equation}

\begin{remark}
For $K=\bar{B}\left(  z_{0},R\right)  $ and $y_{0}=z_{0}$ one has $\rho
=R_{K}=R$, so that $\left(  C_{0}\right)  $ asks $\tan\theta_{\max}<1$ and is
implied by $\left(  C_{1}\right)  $. In general, $\left(  C_{0}\right)  $
refers to the fact that the obliqueness angle is smaller than the aperture
under which the interior ball is seen from the boundary. Similarly to the
Assumption $\left(  C_{2}\right)  $, it is a requirement on $\theta$ relative
to the shape of $K$, it is a geometrical compatibility wit the domain.
\end{remark}

\subsubsection*{Uniform a priori estimates for the penalizing solutions}

Everything rests on the following Proposition, which is the substitute for the
discarded assumption ($C_{3}$). All the constants it produces are depending
from $\eta,F,\mathcal{R},K,\theta$ alone, they are $\varepsilon-$independent.

\begin{proposition}
[Uniform bounds]\label{rot-apriori}Let all the Assumptions \ref{H1}, \ref{H2},
\ref{H2'}, \ref{H3}, \ref{H3'}, \ref{H4'}, \ref{H5'} and $\left(
\mathrm{C}_{0}\right)  $ hold, and define the positive constants%
\begin{equation}
\left\{
\begin{array}
[c]{l}%
c_{6}:=3f_{\infty}^{2}+6\left\Vert L\right\Vert _{\infty}^{2}c_{K}%
^{2}T+3\left\Vert \ell\right\Vert _{\infty}^{2}C_{\mathcal{R}}^{2}R_{K}%
^{2}~,\quad\quad c_{7}:=6\left\Vert \ell\right\Vert _{\infty}^{2}%
C_{\mathcal{R}}^{2}R_{K}\sqrt{T}~,\medskip\\
\bar{\Psi}:=\left(  c_{7}+\sqrt{c_{7}^{2}+2c_{6}}\right)  ^{2},\quad
\Lambda^{2}:=R_{K}^{2}+2R_{K}\sqrt{T\bar{\Psi}}+\dfrac{\bar{\Psi}}{2a_{\theta
}},\quad j:=D_{K}+2.
\end{array}
\right.  \label{def of constants1}%
\end{equation}
Let, finally, $\varepsilon_{1}\in(0,1]$ be defined by
\begin{equation}
\varepsilon_{1}:=\min\left\{  1,\dfrac{a_{\theta}}{\bar{\Psi}},\varepsilon
_{1}^{\prime}\right\}  ,\quad\text{where}\quad\varepsilon_{1}^{\prime}%
:=\sup\left\{  \varepsilon\in(0,1]:\dfrac{6\left\Vert L\right\Vert _{\infty
}^{2}\varepsilon^{2}}{a_{\theta}^{2}}+\dfrac{3\left\Vert \ell\right\Vert
_{\infty}^{2}C_{\mathcal{R}}^{2}\varepsilon}{2a_{\theta}}\leq\dfrac{1}%
{2}\right\}  . \label{def of constants2}%
\end{equation}
Then, for every $\varepsilon\in(0,\varepsilon_{1}]$ and every stopping time
$\tau\leq T$, $\mathbb{P}$-a.s., we have:%
\begin{equation}
\left\{
\begin{array}
[c]{ll}%
\left(  a\right)  \quad & \mathbb{E}^{\mathcal{F}_{\tau}}{%
{\displaystyle\int_{\tau}^{T}}
}\left\vert F\left(  r,Y_{r}^{\varepsilon},\mathcal{R}_{r}(M^{\varepsilon
})\right)  \right\vert ^{2}dr\leq\bar{\Psi}~;\medskip\\
\left(  b\right)  \quad & \sup\limits_{t\in\left[  0,T\right]  }d_{K}\left(
Y_{t}^{\varepsilon}\right)  \leq\sqrt{\dfrac{\varepsilon\bar{\Psi}}{a_{\theta
}}}\leq\sqrt{\varepsilon}\leq1~;\quad\text{hence}\quad\left\{
\begin{array}
[c]{l}%
\sup\limits_{t\in\left[  0,T\right]  }\left\vert Y_{t}^{\varepsilon
}\right\vert \leq c_{K}+1~;\smallskip\\
\sup\limits_{r\in\left[  0,T\right]  }\left\vert \Delta M_{r}^{\varepsilon
}\right\vert \leq j~;
\end{array}
\right. \\
\left(  c\right)  \quad & \mathbb{E}^{\mathcal{F}_{\tau}}{%
{\displaystyle\int_{\tau}^{T}}
}\left\vert U_{r}^{\varepsilon}\right\vert ^{2}dr\leq\dfrac{\bar{\Psi}%
}{a_{\theta}^{2}}\quad\text{and}\quad\mathbb{E}^{\mathcal{F}_{\tau}}{%
{\displaystyle\int_{\tau}^{T}}
}\left\vert \mathcal{R}_{r}(M^{\varepsilon})\right\vert ^{2}dr\leq
C_{\mathcal{R}}^{2}\Lambda^{2}~;\medskip\\
\left(  d\right)  \quad & \mathbb{E}^{\mathcal{F}_{\tau}}\left(  \left[
M^{\varepsilon}\right]  _{T}-\left[  M^{\varepsilon}\right]  _{\tau}\right)
\leq\Lambda^{2},\quad\text{that is}\quad\left\vert \left\vert M^{\varepsilon
}\right\vert \right\vert _{BMO}^{2}\leq\Lambda^{2}.
\end{array}
\right.  \label{r18}%
\end{equation}
The same bounds, with $d_{K}\left(  Y\right)  \equiv0$ in $\left(  b\right)  $
and $j$ replaced by $D_{K}$, hold for every solution $\left(  Y,M,U\right)
\in{\mathcal{{\mathbb{D}}}}_{2}^{2}\times\mathcal{M}_{2}^{2}\times\Lambda
_{2}^{2}$ of (\ref{r1}).
\end{proposition}

\begin{proof}
Denote $F_{r}^{\varepsilon}:=F\left(  r,Y_{r}^{\varepsilon},\mathcal{R}%
_{r}(M^{\varepsilon})\right)  $, $\bar{Y}_{r}^{\varepsilon}:=\pi_{K}\left(
Y_{r}^{\varepsilon}\right)  $ and, for a stopping time $\tau\leq T$,
\[
\Psi_{\tau}:=\mathbb{E}^{\mathcal{F}_{\tau}}{%
{\displaystyle\int_{\tau}^{T}}
}\left\vert F_{r}^{\varepsilon}\right\vert ^{2}dr,\quad\quad\text{and}%
\quad\quad X_{\tau}:=\mathbb{E}^{\mathcal{F}_{\tau}}\left(  \left[
M^{\varepsilon}\right]  _{T}-\left[  M^{\varepsilon}\right]  _{\tau}\right)
,
\]
both finite $\mathbb{P}$-a.s.\ by (\ref{r5}).

\smallskip\noindent\textbf{Step 1: the two It\^{o} identities.} By
(\ref{r5})$\left(  a\right)  $ and the Cauchy--Schwarz inequality,%
\[
\mathbb{E}\left(  {\int_{0+}^{T}}\left\vert Y_{r-}^{\varepsilon}%
-y_{0}\right\vert ^{2}d\left[  M^{\varepsilon}\right]  _{r}\right)  ^{1/2}%
\leq\left(  \mathbb{E}\sup\limits_{r\in\left[  0,T\right]  }\left\vert
Y_{r}^{\varepsilon}-y_{0}\right\vert ^{2}\right)  ^{1/2}\left(  \mathbb{E}%
\left[  M^{\varepsilon}\right]  _{T}\right)  ^{1/2}<\infty,
\]
since, by the Burkholder--Davis--Gundy inequality (\ref{BDG1}), written with
$p=1$, $\left(  \mathbb{E}\left[  M^{\varepsilon}\right]  _{T}\right)
^{1/2}<+\infty$. Consequently, $\int_{0+}^{\cdot}\left\langle Y_{r-}%
^{\varepsilon}-y_{0},dM_{r}^{\varepsilon}\right\rangle $ is a uniformly
integrable martingale. Since $\varepsilon\left\vert U_{r}^{\varepsilon
}\right\vert =d_{K}\left(  Y_{r}^{\varepsilon}\right)  \leq\left\vert
Y_{r}^{\varepsilon}\right\vert +\max_{z\in K}\left\vert z\right\vert $, then
$\int_{0+}^{\cdot}\left\langle \varepsilon U_{r-}^{\varepsilon},dM_{r}%
^{\varepsilon}\right\rangle $ is also a uniformly integrable martingale.

Firstly, apply now backward It\^{o}'s formula to $\left\vert \cdot
-y_{0}\right\vert ^{2}$, together with $Y_{T}^{\varepsilon}=\eta\in K.$ Take
the conditional expectation with respect to $\mathcal{F}_{\tau}$ and it gives,
due to the notation from the beginning of the proof and to the fact that
$\int_{0+}^{\cdot}\left\langle Y_{r-}^{\varepsilon}-y_{0},dM_{r}^{\varepsilon
}\right\rangle $ is a uniformly integrable martingale,%
\begin{equation}
\left\vert Y_{\tau}^{\varepsilon}-y_{0}\right\vert ^{2}+X_{\tau}%
=\mathbb{E}^{\mathcal{F}_{\tau}}\left\vert \eta-y_{0}\right\vert
^{2}+2\mathbb{E}^{\mathcal{F}_{\tau}}{%
{\displaystyle\int_{\tau}^{T}}
}\left\langle Y_{r}^{\varepsilon}-y_{0},F_{r}^{\varepsilon}\right\rangle
dr-2\mathbb{E}^{\mathcal{F}_{\tau}}{%
{\displaystyle\int_{\tau}^{T}}
}\left\langle Y_{r}^{\varepsilon}-y_{0},\Theta_{r}U_{r}^{\varepsilon
}\right\rangle dr. \label{r18b}%
\end{equation}

Secondly, it is well known that the distance function $f:=\frac{1}{2}d_{K}%
^{2}$ is a convex regular function, with $\nabla f\left(  y\right)  =y-\pi
_{K}\left(  y\right)  =\varepsilon U^{\varepsilon}$ at $y=Y^{\varepsilon}$.
From the forward It\^{o}-Meyer's formula for c\`{a}dl\`{a}g semimartingales
and sufficiently regular convex functions yields, if we apply it to $f\left(
Y^{\varepsilon}\right)  $, between $\tau$ and $T$,%
\begin{equation}%
\begin{array}
[c]{lll}%
f\left(  \eta\right)  & = & f\left(  Y_{\tau}^{\varepsilon}\right)  +{%
{\displaystyle\int_{\tau}^{T}}
}\left\langle \nabla f\left(  Y_{r-}^{\varepsilon}\right)  ,dY_{r}%
^{\varepsilon}\right\rangle +\dfrac{1}{2}{%
{\displaystyle\int_{\tau}^{T}}
}D^{2}f\left(  Y_{r-}^{\varepsilon}\right)  \,d\left[  Y^{\varepsilon
,c}\right]  _{r}\medskip\\
& + &
{\displaystyle\sum\limits_{\tau<r\leq T}}
\left[  f\left(  Y_{r}^{\varepsilon}\right)  -f\left(  Y_{r-}^{\varepsilon
}\right)  -\left\langle \nabla f\left(  Y_{r-}^{\varepsilon}\right)  ,\Delta
Y_{r}^{\varepsilon}\right\rangle \right]  ,
\end{array}
\label{second Ito int}%
\end{equation}
where the last two terms are nonnegative, by the convexity of $f$, and
$f\left(  \eta\right)  =0$. Indeed, for the last sum, the convexity of $f$
gives $f\left(  Y_{r}^{\varepsilon}\right)  \geq f\left(  Y_{r-}^{\varepsilon
}\right)  +\left\langle \nabla f\left(  Y_{r-}^{\varepsilon}\right)  ,\Delta
Y_{r}^{\varepsilon}\right\rangle $, with $\Delta Y_{r}^{\varepsilon}%
=Y_{r}^{\varepsilon}-Y_{r-}^{\varepsilon}$~. From (\ref{second Ito int}) we
obtain the following fundamental inequality:%
\[%
\begin{array}
[c]{l}%
0\geq f\left(  Y_{\tau}^{\varepsilon}\right)  +{%
{\displaystyle\int_{\tau}^{T}}
}\left\langle \nabla f\left(  Y_{r-}^{\varepsilon}\right)  ,dY_{r}%
^{\varepsilon}\right\rangle =f\left(  Y_{\tau}^{\varepsilon}\right)  +{%
{\displaystyle\int_{\tau}^{T}}
}\left\langle \varepsilon U_{r-}^{\varepsilon},-F_{r}^{\varepsilon}%
dr+\Theta_{r}U_{r}^{\varepsilon}dr+dM_{r}^{\varepsilon}\right\rangle
\medskip\\
\quad\quad\quad=f\left(  Y_{\tau}^{\varepsilon}\right)  -\varepsilon{%
{\displaystyle\int_{\tau}^{T}}
}\left\langle U_{r}^{\varepsilon},F_{r}^{\varepsilon}\right\rangle
dr+\varepsilon{%
{\displaystyle\int_{\tau}^{T}}
}\left\langle U_{r}^{\varepsilon},\Theta_{r}U_{r}^{\varepsilon}\right\rangle
dr+\varepsilon{%
{\displaystyle\int_{\tau}^{T}}
}\left\langle U_{r-}^{\varepsilon},dM_{r}^{\varepsilon}\right\rangle
\end{array}
\]
It follows%
\[
f\left(  Y_{\tau}^{\varepsilon}\right)  +\varepsilon{%
{\displaystyle\int_{\tau}^{T}}
}\left\langle U_{r}^{\varepsilon},\Theta_{r}U_{r}^{\varepsilon}\right\rangle
dr\leq\varepsilon{%
{\displaystyle\int_{\tau}^{T}}
}\left\langle U_{r}^{\varepsilon},F_{r}^{\varepsilon}\right\rangle
dr-\varepsilon{%
{\displaystyle\int_{\tau}^{T}}
}\left\langle U_{r-}^{\varepsilon},dM_{r}^{\varepsilon}\right\rangle .
\]
According to (\ref{r2}), we have $\left\langle U_{r}^{\varepsilon},\Theta
_{r}U_{r}^{\varepsilon}\right\rangle \geq a_{\theta}\left\vert U_{r}%
^{\varepsilon}\right\vert ^{2}$. Take $\mathbb{E}^{\mathcal{F}_{\tau}}$ and
recall the definition of $f.$ We get, since the last term is a uniformly
integrable martingale and it vanishes,%
\begin{equation}
\dfrac{1}{2}d_{K}^{2}\left(  Y_{\tau}^{\varepsilon}\right)  +\varepsilon
a_{\theta}\mathbb{E}^{\mathcal{F}_{\tau}}{%
{\displaystyle\int_{\tau}^{T}}
}\left\vert U_{r}^{\varepsilon}\right\vert ^{2}dr\leq\varepsilon
\mathbb{E}^{\mathcal{F}_{\tau}}{%
{\displaystyle\int_{\tau}^{T}}
}\left\vert U_{r}^{\varepsilon}\right\vert \left\vert F_{r}^{\varepsilon
}\right\vert dr\overset{Young}{\leq}\dfrac{\varepsilon a_{\theta}}%
{2}\mathbb{E}^{\mathcal{F}_{\tau}}{\int_{\tau}^{T}}\left\vert U_{r}%
^{\varepsilon}\right\vert ^{2}dr+\dfrac{\varepsilon}{2a_{\theta}}\Psi_{\tau},
\label{r18c}%
\end{equation}
where $\Psi_{\tau}=\mathbb{E}^{\mathcal{F}_{\tau}}{%
{\textstyle\int_{\tau}^{T}}
}\left\vert F_{r}^{\varepsilon}\right\vert ^{2}dr$ is exactly the one given at
the beginning of the proof. Summarizing, the first term from the RHS of
(\ref{r18c}) is absorbed in the LHS, the integral remains with the coefficient
$\varepsilon a_{\theta}/2$ and we can conclude that the following boundedness
take place:%
\begin{equation}
d_{K}^{2}\left(  Y_{\tau}^{\varepsilon}\right)  \leq\dfrac{\varepsilon
}{a_{\theta}}\Psi_{\tau}\quad\quad\text{and}\quad\quad\mathbb{E}%
^{\mathcal{F}_{\tau}}{%
{\displaystyle\int_{\tau}^{T}}
}\left\vert U_{r}^{\varepsilon}\right\vert ^{2}dr\leq\dfrac{\Psi_{\tau}%
}{a_{\theta}^{2}}. \label{r18d}%
\end{equation}

\smallskip\noindent\textbf{Step 2: the reflection term control, and a
cancellation.} If $U_{r}^{\varepsilon}\neq0$ then $\bar{Y}_{r}^{\varepsilon
}:=\pi_{K}\left(  Y_{r}^{\varepsilon}\right)  \in bd\left(  K\right)  $ and
$U_{r}^{\varepsilon}=\left\vert U_{r}^{\varepsilon}\right\vert \mathbf{n}%
\left(  \bar{Y}_{r}^{\varepsilon}\right)  $, while $Y_{r}^{\varepsilon}%
=\bar{Y}_{r}^{\varepsilon}+\varepsilon U_{r}^{\varepsilon}$. For $z\in
bd\left(  K\right)  $ the convexity of $K$ and the geometrical constraint
$\bar{B}\left(  y_{0},\rho\right)  \subset K$ give $\left\langle
z-y_{0},\mathbf{n}\left(  z\right)  \right\rangle \geq\rho$, while $\left\vert
\left\langle z-y_{0},J\mathbf{n}\left(  z\right)  \right\rangle \right\vert
\leq R_{K}$. As consequence, by the Assumption $\left(  C_{0}\right)  $,%
\[
\left\langle z-y_{0},\Theta_{r}\mathbf{n}\left(  z\right)  \right\rangle
=\cos\theta\left\langle z-y_{0},\mathbf{n}\left(  z\right)  \right\rangle
+\sin\theta\left\langle z-y_{0},J\mathbf{n}\left(  z\right)  \right\rangle
\geq\rho\cos\theta-R_{K}\sin\theta\geq\delta_{0}\geq0.
\]
Consequently
\[
\left\langle Y_{r}^{\varepsilon}-y_{0},\Theta_{r}U_{r}^{\varepsilon
}\right\rangle =\left\vert U_{r}^{\varepsilon}\right\vert \left\langle \bar
{Y}_{r}^{\varepsilon}-y_{0},\Theta_{r}\mathbf{n}\left(  \bar{Y}_{r}%
^{\varepsilon}\right)  \right\rangle +\varepsilon\left\vert U_{r}%
^{\varepsilon}\right\vert ^{2}\left\langle \mathbf{n}\left(  \bar{Y}%
_{r}^{\varepsilon}\right)  ,\Theta_{r}\mathbf{n}\left(  \bar{Y}_{r}%
^{\varepsilon}\right)  \right\rangle \geq\delta_{0}\left\vert U_{r}%
^{\varepsilon}\right\vert +\varepsilon a_{\theta}\left\vert U_{r}%
^{\varepsilon}\right\vert ^{2},
\]
so that the last term of (\ref{r18b}) is bounded from above as we see below:%
\[
-2\mathbb{E}^{\mathcal{F}_{\tau}}{%
{\displaystyle\int_{\tau}^{T}}
}\left\langle Y_{r}^{\varepsilon}-y_{0},\Theta_{r}U_{r}^{\varepsilon
}\right\rangle dr\leq-2\delta_{0}\mathbb{E}^{\mathcal{F}_{\tau}}\int_{\tau
}^{T}\left\vert U_{r}^{\varepsilon}\right\vert dr-2\varepsilon a_{\theta
}\mathbb{E}^{\mathcal{F}_{\tau}}\int_{\tau}^{T}\left\vert U_{r}^{\varepsilon
}\right\vert ^{2}dr.
\]
On the other hand $\left\vert Y_{r}^{\varepsilon}-y_{0}\right\vert
\leq\left\vert \bar{Y}_{r}^{\varepsilon}-y_{0}\right\vert +\left\vert
Y_{r}^{\varepsilon}-\bar{Y}_{r}^{\varepsilon}\right\vert \leq R_{K}%
+d_{K}\left(  Y_{r}^{\varepsilon}\right)  =R_{K}+\varepsilon\left\vert
U_{r}^{\varepsilon}\right\vert $ and, by Young's inequality, $2\varepsilon
\left\vert U_{r}^{\varepsilon}\right\vert \left\vert F_{r}^{\varepsilon
}\right\vert \leq2\varepsilon a_{\theta}\left\vert U_{r}^{\varepsilon
}\right\vert ^{2}+\frac{\varepsilon}{2a_{\theta}}\left\vert F_{r}%
^{\varepsilon}\right\vert ^{2}$. The two terms in $\varepsilon a_{\theta
}\left\vert U^{\varepsilon}\right\vert ^{2}$ cancel, $\left\vert \left\vert
\eta-y_{0}\right\vert \right\vert _{\infty}\leq R_{K},$ and (\ref{r18b})
becomes, after we ignore the positive term $\left\vert Y_{\tau}^{\varepsilon
}-y_{0}\right\vert ^{2}$ from the LHS:%
\begin{align*}
X_{\tau}  &  \leq\mathbb{E}^{\mathcal{F}_{\tau}}\left\vert \eta-y_{0}%
\right\vert ^{2}+2\mathbb{E}^{\mathcal{F}_{\tau}}{%
{\displaystyle\int_{\tau}^{T}}
}\left\langle Y_{r}^{\varepsilon}-y_{0},F_{r}^{\varepsilon}\right\rangle
dr-2\mathbb{E}^{\mathcal{F}_{\tau}}{%
{\displaystyle\int_{\tau}^{T}}
}\left\langle Y_{r}^{\varepsilon}-y_{0},\Theta_{r}U_{r}^{\varepsilon
}\right\rangle dr\\
&  \leq R_{K}^{2}+\mathbb{E}^{\mathcal{F}_{\tau}}{%
{\displaystyle\int_{\tau}^{T}}
}\left(  {2}R_{K}\left\vert F_{r}^{\varepsilon}\right\vert +2\varepsilon
\left\vert U_{r}^{\varepsilon}\right\vert \left\vert F_{r}^{\varepsilon
}\right\vert \right)  dr-2\delta_{0}\mathbb{E}^{\mathcal{F}_{\tau}}\int_{\tau
}^{T}\left\vert U_{r}^{\varepsilon}\right\vert dr-2\varepsilon a_{\theta
}\mathbb{E}^{\mathcal{F}_{\tau}}\int_{\tau}^{T}\left\vert U_{r}^{\varepsilon
}\right\vert ^{2}dr\\
&  \leq R_{K}^{2}+2R_{K}\mathbb{E}^{\mathcal{F}_{\tau}}{%
{\displaystyle\int_{\tau}^{T}}
}\left\vert F_{r}^{\varepsilon}\right\vert dr+\frac{\varepsilon}{2a_{\theta}%
}\mathbb{E}^{\mathcal{F}_{\tau}}{%
{\displaystyle\int_{\tau}^{T}}
}\left\vert F_{r}^{\varepsilon}\right\vert ^{2}dr-2\delta_{0}\mathbb{E}%
^{\mathcal{F}_{\tau}}\int_{\tau}^{T}\left\vert U_{r}^{\varepsilon}\right\vert
dr\\
&  =R_{K}^{2}+2R_{K}\mathbb{E}^{\mathcal{F}_{\tau}}{%
{\displaystyle\int_{\tau}^{T}}
}\left\vert F_{r}^{\varepsilon}\right\vert dr+\frac{\varepsilon}{2a_{\theta}%
}\Psi_{\tau}\leq R_{K}^{2}+2R_{K}\sqrt{T\Psi_{\tau}}+\dfrac{\varepsilon
}{2a_{\theta}}\Psi_{\tau}%
\end{align*}
We obtained
\begin{equation}
X_{\tau}\leq R_{K}^{2}+2R_{K}\sqrt{T\Psi_{\tau}}+\dfrac{\varepsilon
}{2a_{\theta}}\Psi_{\tau}. \label{r18e}%
\end{equation}
\textit{This cancellation is the key point of the proof:} it is what prevents
the penalization from entering the $BMO$ estimate at first order.

\smallskip\noindent\textbf{Step 3: closing the loop.} The objective of this
step is to close the estimate loop for $\Psi_{\tau}=\mathbb{E}^{\mathcal{F}%
_{\tau}}{%
{\textstyle\int_{\tau}^{T}}
}\left\vert F_{r}^{\varepsilon}\right\vert ^{2}dr$, by proving that
$\Psi_{\tau}\leq\bar{\Psi},$ $\mathbb{P}$-a.s., uniformly with respect to
$\varepsilon\in(0,\varepsilon_{1}],$ where these final elements are given by
(\ref{def of constants1}) and (\ref{def of constants2}). By Assumption
\ref{H2}, $\left\vert F_{r}^{\varepsilon}\right\vert ^{2}\leq3\left(
\left\vert F\left(  r,0,0\right)  \right\vert ^{2}+L_{r}^{2}\left\vert
Y_{r}^{\varepsilon}\right\vert ^{2}+\ell_{r}^{2}\left\vert \mathcal{R}%
_{r}(M^{\varepsilon})\right\vert ^{2}\right)  $. For the state term
$\left\vert Y_{r}^{\varepsilon}\right\vert ^{2}$, since $y_{0}\in int\left(
K\right)  $ and $\left\vert y_{0}\right\vert +R_{K}=c_{K},$ we have
$\left\vert Y_{r}^{\varepsilon}\right\vert ^{2}\leq2c_{K}^{2}+2d_{K}%
^{2}\left(  Y_{r}^{\varepsilon}\right)  =2c_{K}^{2}+2\varepsilon^{2}\left\vert
U_{r}^{\varepsilon}\right\vert ^{2}$. We integrate from $\tau$ to $T$ and take
$\mathbb{E}^{\mathcal{F}_{\tau}}$ and it yields%
\begin{equation}
\Psi_{\tau}=\mathbb{E}^{\mathcal{F}_{\tau}}{%
{\displaystyle\int_{\tau}^{T}}
}\left\vert F_{r}^{\varepsilon}\right\vert ^{2}dr\leq3f_{\infty}%
^{2}+6\left\vert \left\vert L\right\vert \right\vert _{\infty}^{2}c_{K}%
^{2}T+6\left\vert \left\vert L\right\vert \right\vert _{\infty}^{2}%
\varepsilon^{2}\mathbb{E}^{\mathcal{F}_{\tau}}{%
{\displaystyle\int_{\tau}^{T}}
}\left\vert U_{r}^{\varepsilon}\right\vert ^{2}dr+3\ell_{r}^{2}\mathbb{E}%
^{\mathcal{F}_{\tau}}{%
{\displaystyle\int_{\tau}^{T}}
}\left\vert \mathcal{R}_{r}(M^{\varepsilon})\right\vert ^{2}dr.
\label{estim for FI tau}%
\end{equation}
Formula (\ref{r18d}) gives $\mathbb{E}^{\mathcal{F}_{\tau}}{%
{\textstyle\int_{\tau}^{T}}
}\left\vert U_{r}^{\varepsilon}\right\vert ^{2}dr\leq\tfrac{\Psi_{\tau}%
}{a_{\theta}^{2}}$, while Lemma \ref{H3'-facts}-$\left(  b\right)  $, with
$N=0$, assures that we have $\mathbb{E}^{\mathcal{F}_{\tau}}{%
{\textstyle\int_{\tau}^{T}}
}\left\vert \mathcal{R}_{r}(M^{\varepsilon})\right\vert ^{2}dr\leq
C_{\mathcal{R}}^{2}\mathbb{E}^{\mathcal{F}_{\tau}}\left(  \left[
M^{\varepsilon}\right]  _{T}-\left[  M^{\varepsilon}\right]  _{\tau}\right)
=C_{\mathcal{R}}^{2}X_{\tau}.$ Insert these two estimates into
(\ref{estim for FI tau}) to get%
\[
\Psi_{\tau}\leq3f_{\infty}^{2}+6\left\Vert L\right\Vert _{\infty}^{2}c_{K}%
^{2}T+\dfrac{6\left\Vert L\right\Vert _{\infty}^{2}\varepsilon^{2}}{a_{\theta
}^{2}}\Psi_{\tau}+3\left\Vert \ell\right\Vert _{\infty}^{2}C_{\mathcal{R}}%
^{2}X_{\tau}.
\]
Finally, inserting (\ref{r18e}) into it, we obtain, for the positive constants
$c_{6}$ and $c_{7}$ defined in (\ref{def of constants1}),%
\[
\Psi_{\tau}\leq c_{6}+c_{7}\sqrt{\Psi_{\tau}}+\left(  \dfrac{6\left\Vert
L\right\Vert _{\infty}^{2}\varepsilon^{2}}{a_{\theta}^{2}}+\dfrac{3\left\Vert
\ell\right\Vert _{\infty}^{2}C_{\mathcal{R}}^{2}\varepsilon}{2a_{\theta}%
}\right)  \Psi_{\tau}.
\]
For $\varepsilon\leq\varepsilon_{1}^{\prime}$ the bracket is at most $\frac
{1}{2}$ (from (\ref{def of constants2})). Therefore, after the absorption in
the LHS, it remains that $\Psi_{\tau}\leq2c_{6}+2c_{7}\sqrt{\Psi_{\tau}}$.
Since $\Psi_{\tau}<\infty$, $\mathbb{P}$-a.s., we have a second degree
equation, which gives the fixed explicit representation of $\bar{\Psi}$ given
by (\ref{def of constants1}):%
\[
\Psi_{\tau}\leq\left(  c_{7}+\sqrt{c_{7}^{2}+2c_{6}}\right)  ^{2}=:\bar{\Psi
},\quad\mathbb{P}\text{-a.s.,}%
\]
i.e. (\ref{r18})-$\left(  a\right)  $ holds. Then (\ref{r18d}) gives
(\ref{r18})-$\left(  b\right)  $ --- the bound $\sqrt{\varepsilon\bar{\Psi
}/a_{\theta}}\leq1$ holding because $\varepsilon\leq a_{\theta}/\bar{\Psi}$
--- and the first half of $\left(  c\right)  $, while (\ref{r18e}) gives
$\left(  d\right)  $. The second half of $\left(  c\right)  $ follows from
Lemma \ref{H3'-facts}-$\left(  b\right)  ,\left(  d\right)  $. For the jumps,
the finite variation part of $Y^{\varepsilon}$ is continuous, so $\Delta
M_{r}^{\varepsilon}=\Delta Y_{r}^{\varepsilon}$, and both $Y_{r}^{\varepsilon
}$ and $Y_{r-}^{\varepsilon}$ lie within distance at most $1$ of $K$. As
consequence, $\left\vert \Delta M_{r}^{\varepsilon}\right\vert =\left\vert
\Delta Y_{r}^{\varepsilon}\right\vert \leq diam\left(  K\right)
+1+1=D_{K}+2=:j$.

From (\ref{r18d}) we have $\mathbb{E}^{\mathcal{F}_{\tau}}{%
{\textstyle\int_{\tau}^{T}}
}\left\vert U_{r}^{\varepsilon}\right\vert ^{2}dr\leq\tfrac{\Psi_{\tau}%
}{a_{\theta}^{2}}$. The estimate (\ref{r18e}) gives, since $\Psi_{\tau}\leq$
$\bar{\Psi}$ and $\varepsilon\leq1$,%
\[
X_{\tau}\leq R_{K}^{2}+2R_{K}\sqrt{T\Psi_{\tau}}+\dfrac{\varepsilon
}{2a_{\theta}}\Psi_{\tau}\leq R_{K}^{2}+2R_{K}\sqrt{T\bar{\Psi}}+\dfrac
{1}{2a_{\theta}}\bar{\Psi}=:\Lambda^{2}.
\]

\smallskip\noindent\textbf{Step 4: for the solutions of (\ref{r1}).} If
$\left(  Y,M,U\right)  $ solves (\ref{r1}) then $Y_{r}\in K$, $d_{K}\left(
Y\right)  \equiv0$, $\left\vert \Delta M_{r}\right\vert =\left\vert \Delta
Y_{r}\right\vert \leq D_{K}$, and $U_{r}\in\partial I_{K}\left(  Y_{r}\right)
=N_{K}\left(  Y_{r}\right)  $. Therefore, $U_{r}=\left\vert U_{r}\right\vert
\mathbf{n}\left(  Y_{r}\right)  $ whenever $U_{r}\neq0,$that is when we lie on
$bd\left(  K\right)  $. In this situation, all the three previous steps apply
verbatim with $\varepsilon=0$. Steps 1 and 3 then give (\ref{r18e}) and
$\Psi_{\tau}\leq c_{6}+c_{7}\sqrt{\Psi_{\tau}}$ without any restriction, and
the conclusion follows as before.$\smallskip$

The proof of Proposition \ref{rot-apriori} is now complete.\hfill
\end{proof}

\subsubsection*{Exponential moments}

The following classical inequality, Garsia's lemma (see Dellacherie and Meyer
\cite[Chapter\ VI, pp. 105--107]{Dellacherie/Meyer:80} or Kazamaki
\cite[Chapter\ 2]{Kazamaki:94}) converts a $BMO$ bound into exponential
integrability. It is stated for increasing processes, which is exactly what is
needed in our study. The weight $\Gamma$ is built on $\left[  M^{\varepsilon
}\right]  $, and no passage to the predictable bracket $\left\langle
M^{\varepsilon}\right\rangle $, nor any form of the John--Nirenberg inequality
for martingales, is required.

\begin{lemma}
[Garsia]\label{garsia}Let $A$ be an adapted, c\`{a}dl\`{a}g, nondecreasing
process with $A_{0}=0$ and $A_{T}\in L^{1}$, and let $c>0$ be such that
\[
\mathbb{E}^{\mathcal{F}_{\tau}}\left(  A_{T}-A_{\tau-}\right)  \leq
c,\quad\quad\mathbb{P}\text{-a.s., for every stopping time }\tau\leq T\text{.}%
\]
Then $\mathbb{E}A_{T}^{n}\leq n!c^{n}$ for every $n\geq1$, and
\[
\mathbb{E}e^{\lambda A_{T}}\leq\dfrac{1}{1-\lambda c},\quad\quad\text{for
every }\lambda\in\left(  0,1/c\right)  .
\]

\end{lemma}

\begin{proof}
[Sketch]The pathwise inequality $A_{T}^{n}\leq n\int_{0}^{T}\left(
A_{T}-A_{s-}\right)  ^{n-1}dA_{s}$ and the section theorem give, by induction
on $n$, $\mathbb{E}^{\mathcal{F}_{\tau}}\left(  A_{T}-A_{\tau-}\right)
^{n}\leq n!c^{n},$ for every stopping time $\tau$. Summing the series yields
the second assertion. Indeed, $\mathbb{E}e^{\lambda A_{T}}=\sum_{n=0}^{\infty
}\frac{\lambda^{n}\mathbb{E}A_{T}^{n}}{n!}\leq\sum_{n=0}^{\infty}\left(
\lambda c\right)  ^{n}=\tfrac{1}{1-\lambda c},$ for $\lambda\in\left(
0,1/c\right)  .$\hfill
\end{proof}

\begin{corollary}
\label{rot-expmom}Consider that we situate under the assumptions of
Proposition \ref{rot-apriori}, keeping all the notations introduced there, and
all the statements taking place. Then, for every $\varepsilon\in
(0,\varepsilon_{1}]$,

\begin{itemize}
\item[$\left(  a\right)  $] $\mathbb{E}\exp\left(  \lambda\left[
M^{\varepsilon}\right]  _{T}\right)  \leq\left(  1-\lambda\Lambda_{\ast}%
^{2}\right)  ^{-1}$, for every $\lambda\in\left(  0,1/\Lambda_{\ast}%
^{2}\right)  $, where $\Lambda_{\ast}^{2}:=\Lambda^{2}+2j^{2}$;

\item[$\left(  b\right)  $] $\mathbb{E}\left(  {%
{\displaystyle\int_{0}^{T}}
}\left\vert U_{r}^{\varepsilon}\right\vert ^{2}dr\right)  ^{q}\leq
\Gamma\left(  q+1\right)  \left(  \dfrac{\bar{\Psi}}{a_{\theta}^{2}}\right)
^{q}$, for every $q\geq1$;

\item[$\left(  c\right)  $] $\mathbb{E}\exp\left(  \lambda{%
{\displaystyle\int_{0}^{T}}
}\left\vert F\left(  r,Y_{r}^{\varepsilon},\mathcal{R}_{r}(M^{\varepsilon
})\right)  \right\vert dr\right)  <+\infty$, for every $\lambda>0$.
\end{itemize}
\end{corollary}

\begin{proof}
Consider $\varepsilon\in(0,\varepsilon_{1}]$.

\noindent$\left(  a\right)  $ We will apply Lemma \ref{garsia} to $A=\left[
M^{\varepsilon}\right]  $. For the exponential moments of the quadratic
variation $\left[  M^{\varepsilon}\right]  _{T}$, for a stopping time $\tau$,%
\[
\mathbb{E}^{\mathcal{F}_{\tau}}\left(  \left[  M^{\varepsilon}\right]
_{T}-\left[  M^{\varepsilon}\right]  _{\tau-}\right)  =\mathbb{E}%
^{\mathcal{F}_{\tau}}\left(  \left[  M^{\varepsilon}\right]  _{T}-\left[
M^{\varepsilon}\right]  _{\tau}\right)  +\left\vert \Delta M_{\tau
}^{\varepsilon}\right\vert ^{2}\leq\Lambda^{2}+j^{2}\leq\Lambda_{\ast}^{2}~,
\]
by Proposition \ref{rot-apriori}-$\left(  b\right)  ,\left(  d\right)  $.
Apply now Lemma \ref{garsia}, with $c:=\Lambda_{\ast}^{2}~$.

\noindent$\left(  b\right)  $ Apply Lemma \ref{garsia} to the continuous
increasing process $A=\int_{0}^{\cdot}\left\vert U_{r}^{\varepsilon
}\right\vert ^{2}dr$, for which $c=\frac{\bar{\Psi}}{a_{\theta}^{2}}$, using
the result provided by (\ref{r18})-$\left(  c\right)  $. Indeed, $A$ has no
jumps, so $A_{\tau-}=A_{\tau}$ and%
\[
\mathbb{E}^{\mathcal{F}_{\tau}}\left(  A_{T}-A_{\tau}\right)  =\mathbb{E}%
^{\mathcal{F}_{\tau}}\int_{\tau}^{T}\left\vert U_{r}^{\varepsilon}\right\vert
^{2}dr\leq\frac{\bar{\Psi}}{a_{\theta}^{2}}=:c.
\]
Moreover, for $n\in\mathbb{N}^{\ast}$, Lemma \ref{garsia} assures that
$\mathbb{E}A_{T}^{n}\leq n!c^{n}=\Gamma\left(  n+1\right)  c^{n}.$ For any
real $q\geq1$, by interpolation, via H\"{o}lder's inequality, the conclusion follows.

\noindent$\left(  c\right)  $ By Assumption \ref{H2'} and (\ref{r18})-$\left(
b\right)  $ (i.e. $\sup\nolimits_{t\in\left[  0,T\right]  }\left\vert
Y_{t}^{\varepsilon}\right\vert \leq c_{K}+1$), we obtain%
\[
{\int_{0}^{T}}\left\vert F_{r}^{\varepsilon}\right\vert dr\leq\sqrt
{T}f_{\infty}+\left\Vert L\right\Vert _{\infty}T\left(  c_{K}+1\right)
+\left\Vert \ell\right\Vert _{\infty}\sqrt{T}\,\mathcal{Z}^{1/2}%
,\quad\text{where}\quad\mathcal{Z}:={\int_{0}^{T}}\left\vert \mathcal{R}%
_{r}(M^{\varepsilon})\right\vert ^{2}dr.
\]
By (\ref{r18})-$\left(  c\right)  $, the continuous increasing process ${%
{\textstyle\int_{0}^{\cdot}}
}\left\vert \mathcal{R}_{r}(M^{\varepsilon})\right\vert ^{2}dr$ has the
conditional increment bounded from above by $C_{\mathcal{R}}^{2}\Lambda^{2}$.
One can apply now Lemma \ref{garsia}, and we have the boundedness
$\mathbb{E}e^{\lambda^{\prime}\mathcal{Z}}<\infty$, for $\lambda^{\prime
}<1/\left(  C_{\mathcal{R}}^{2}\Lambda^{2}\right)  $. Also, from the
elementary inequality $\lambda\sqrt{\mathcal{Z}}\leq\lambda^{\prime
}\mathcal{Z}+\frac{\lambda^{2}}{4\lambda^{\prime}}~$, it yields $\mathbb{E}%
e^{\lambda\sqrt{\mathcal{Z}}}<+\infty$, for every $\lambda>0$. The proof is
now complete.\hfill
\end{proof}

\subsubsection*{The smallness data condition}

We introduce another compatibility criterion, as one can see below. It
involves the constants used in the previous results, constants which permit
the control, by boundedness, of several terms. We impose:%
\begin{equation}
\left(  C_{3}\right)  \quad\quad64\kappa_{\theta}\Lambda_{\ast}^{2}%
<\alpha_{\theta},\quad\quad\text{where}\quad\quad\Lambda_{\ast}^{2}%
=\Lambda^{2}+2j^{2}, \label{C3}%
\end{equation}
with $\kappa_{\theta}:=2b_{\psi}^{2}\tan\theta_{\max}\left(  1+8\alpha
_{\theta}^{-1}\tan\theta_{\max}\right)  $ given by (\ref{r16a}) and
$\Lambda^{2},j$ given by Proposition \ref{rot-apriori}.

\begin{remark}
\label{rot-C3}$\left(  a\right)  $ Since $\kappa_{\theta}$ is directly
proportional to $\tan\theta_{\max}$ and $\Lambda_{\ast}^{2}$ does not depend
on $\theta$, except through $a_{\theta}\geq\cos\theta_{\max}$, condition
$\left(  C_{3}\right)  $ holds as soon as\emph{ }$\theta_{\max}$\emph{ }is
small enough, given the fixed elements $K,T$ and the data $\eta,F,\mathcal{R}%
$. It is a restriction of exactly the same nature as $\left(  C_{0}\right)
$-$\left(  C_{2}\right)  $, and, like them, it is checked on the given data.

$\left(  b\right)  $ It is the counterpart of the \textquotedblleft smallness
assumption\textquotedblright\ Chassagneux, Nadtochiy, and Richou
\cite[Assumption 2.1]{Chassagneux/Nadtochiy/Richou:22}, which makes their
$BMO$ norm of the martingale part small enough for the John-Nirenberg
inequality to produce the required exponential moment. In our study, the two
roles are distinguished: $\Lambda^{2}$ measures the noise, $\kappa_{\theta}$
the obliqueness, and it is their product that must be small enough.

$\left(  c\right)  $ For $\theta\equiv0$, one has $\kappa_{\theta}=0$, and
Assumption $\left(  C_{3}\right)  $ is trivial satisfied. The weight $\Gamma$
reduces to a deterministic exponential and the argument below is reduced to
the classical one, without nothing new involved.
\end{remark}

\subsubsection*{The weighted estimate for a Cauchy approach}

Theorem \ref{rot-thm-conv} is the main technical result of the paper. It
establishes the quantitative stability of the Moreau--Yosida approximations
and is the key step toward the existence result stated later in Theorem
\ref{rot-thm-exist}. For the convenience of the reader, we organized it to be
self contained, by splitting its proof into eight detailed steps. In
\textit{Step 1}, we apply the c\`{a}dl\`{a}g It\^{o} formula to the two-point
kernel $\Phi_{r}\left(  Y_{r}^{\varepsilon},Y_{r}^{\delta}\right)  $ and
estimate both the continuous and discontinuous second-order terms. In
\textit{Step 2}, the reflection terms are controlled by the boundary
compatibility condition $\left(  C_{2}\right)  $ and the Lipschitz continuity
of the gradient of $\Phi$, producing an error of order $\varepsilon+\delta$.
\textit{Step 3} estimates the generator terms and the time derivative of the
kernel. In \textit{Step 4}, all these bounds are combined with an exponential
weight constructed from the quadratic variations and the absolutely continuous
coefficients. This cancels the random-measure defect pathwise. \textit{Step 5}
establishes the integrability of the weight, by means of the uniform BMO
estimates and the exponential moment bounds of Corollary \ref{rot-expmom}. In
\textit{Step 6}, the weighted locality assumption on $\mathcal{R}$ is used to
control the term involving $\mathcal{R}_{r}^{\varepsilon,\delta}%
=\mathcal{R}_{r}(M^{\varepsilon})-\mathcal{R}_{r}(M^{\delta})$ and to absorb
it into the coercive martingale term. \textit{Step 7} yields the basic
estimates for $\sup_{t}\mathbb{E}\Phi_{t}$ and $\mathbb{E}\left[
M^{\varepsilon}-M^{\delta}\right]  _{T}$. Finally, \textit{Step 8} upgrades
these estimates, by "commuting" $\mathbb{E}$ with the supremum. This is done
by applying the Burkholder--Davis--Gundy inequality to the weighted local martingale.

\begin{theorem}
[Convergence of the penalized family]\label{rot-thm-conv}Let $d=2$ and let
Assumptions \ref{H1}, \ref{H2}, \ref{H2'}, \ref{H3}, \ref{H3'}, \ref{H4'} and
\ref{H5'} be satisfied, together with the compatibility conditions $\left(
C_{0}\right)  $-$\left(  C_{3}\right)  $. Then there exists a positive
constant $C>0$, such that, for all $\varepsilon,\delta\in(0,\varepsilon_{1}]$,%
\begin{equation}
\mathbb{E}\sup\limits_{t\in\left[  0,T\right]  }\left\vert Y_{t}^{\varepsilon
}-Y_{t}^{\delta}\right\vert ^{2}+\mathbb{E}\left[  M^{\varepsilon}-M^{\delta
}\right]  _{T}\leq C\left(  \varepsilon+\delta\right)  . \label{r25}%
\end{equation}
Consequently $\left(  Y^{\varepsilon}\right)  _{\varepsilon}$ is a Cauchy
sequence in ${\mathcal{{\mathbb{D}}}}_{2}^{2}$, $\left(  M^{\varepsilon
}\right)  _{\varepsilon}$ is Cauchy in $\mathcal{M}_{2}^{2}$, and $\left(
U^{\varepsilon}\right)  _{\varepsilon}$ is bounded in $\Lambda_{2}^{2}$.
\end{theorem}

\begin{proof}
Fix $\varepsilon,\delta\in(0,\varepsilon_{1}]$ and introduce the notations:%
\[
Y^{\varepsilon,\delta}:=Y^{\varepsilon}-Y^{\delta},\quad M^{\varepsilon
,\delta}:=M^{\varepsilon}-M^{\delta},\quad\mathcal{R}_{r}^{\varepsilon,\delta
}:=\mathcal{R}_{r}(M^{\varepsilon})-\mathcal{R}_{r}(M^{\delta}),\quad
F_{r}^{\varepsilon}:=F\left(  r,Y_{r}^{\varepsilon},\mathcal{R}_{r}%
(M^{\varepsilon})\right)  ,
\]%
\[
\bar{Y}_{r}^{\varepsilon}:=\pi_{K}\left(  Y_{r}^{\varepsilon}\right)  \in
K\quad\quad\text{and}\quad\quad\varphi_{t}:=\Phi_{t}(Y_{t}^{\varepsilon}%
,Y_{t}^{\delta}).
\]
More precisely, the real-valued process $\varphi$ is given by the
deterministic, class $C^{1,2}\left(  \left[  0,T\right]  \times\mathbb{R}%
^{2}\times\mathbb{R}^{2}\right)  $ test function introduced in (\ref{r10}),
i.e.,%
\[
\varphi_{t}=\Phi_{t}(Y_{t}^{\varepsilon},Y_{t}^{\delta})=\Phi(t,Y_{t}%
^{\varepsilon},Y_{t}^{\delta})=\left\vert Y_{t}^{\varepsilon}-Y_{t}^{\delta
}\right\vert ^{2}-\dfrac{1}{2}\tan\theta_{t}\cdot P(Y_{t}^{\varepsilon}%
,Y_{t}^{\delta})Q(Y_{t}^{\varepsilon},Y_{t}^{\delta}).
\]
Set also the constants and processes:%
\begin{equation}
\left\{
\begin{array}
[c]{l}%
\gamma:=\dfrac{\alpha_{\theta}}{8C_{\mathcal{R}}^{2}}\quad\quad\text{and}%
\quad\quad\nu_{\theta}:=\dfrac{4\kappa_{\theta}}{\alpha_{\theta}},\medskip\\
V_{t}:=2\nu_{\theta}\left(  \left[  M^{\varepsilon}\right]  _{t}+\left[
M^{\delta}\right]  _{t}\right)  +\dfrac{2c_{\Phi}}{\alpha_{\theta}}{%
{\displaystyle\int_{0}^{t}}
}\left(  L_{r}+\dfrac{c_{\Phi}\ell_{r}^{2}}{2\gamma}+\left\vert \left\vert
\theta^{\prime}\right\vert \right\vert _{\infty}+\left\vert F_{r}^{\delta
}\right\vert \right)  dr,\quad\Gamma_{t}:=e^{V_{t}}\geq1.
\end{array}
\right.  \label{r19}%
\end{equation}
The process $V_{\cdot}$ is adapted, c\`{a}dl\`{a}g, nondecreasing, $V_{0}=0$,
and $\Gamma_{\cdot-}$ is predictable.$\smallskip$

From (\ref{r4}), the dynamic of the $\varepsilon$-penalizing process is given
by $dY_{r}^{\varepsilon}=-F_{r}^{\varepsilon}dr+\Theta_{r}U_{r}^{\varepsilon
}dr+dM_{r}^{\varepsilon}$, and, similarly, for the index $\delta$. Remark that
$\varphi_{T}=\Phi_{T}\left(  \eta,\eta\right)  =0$. It\^{o}'s formula for a
c\`{a}dl\`{a}g semimartingale vector $\left(  Y^{\varepsilon},Y^{\delta
}\right)  $, applied to the $C^{1,2}$-class function $\Phi$ gives, under
integral form, from $t$ to $T$:%
\begin{equation}
\varphi_{t}=\int_{t}^{T}\left(  \left\langle \nabla_{x}\Phi_{r}\left(
Y_{r-}^{\varepsilon},Y_{r-}^{\delta}\right)  ,-dY_{r}^{\varepsilon
}\right\rangle +\left\langle \nabla_{y}\Phi_{r}\left(  Y_{r-}^{\varepsilon
},Y_{r-}^{\delta}\right)  ,-dY_{r}^{\delta}\right\rangle -\partial_{r}\Phi
_{r}\left(  Y_{r}^{\varepsilon},Y_{r}^{\delta}\right)  \right)  dr-\mathcal{I}%
_{t}^{T},\label{Ito for vector}%
\end{equation}
where $\mathcal{I}_{t}^{T}:=\mathcal{I}_{t}^{c,T}+\mathcal{I}_{t}^{d,T}$
represents the second order It\^{o}'s correction on $\left[  t,T\right]  $ for
c\`{a}dl\`{a}g semimartingales, obtained when inserting the dynamics of
$Y^{\varepsilon}$ and $Y^{\delta}$ into (\ref{Ito for vector}). It has two
parts: the contribution of the continuous part, $\mathcal{I}_{t}^{c,T}$, and
the contribution of the pure jumps, $\mathcal{I}_{t}^{d,T}$. The continuous
part is obtained from the quadratic variation of the continuous parts of the
martingale vector $\left(  M^{\varepsilon,c},M^{\delta,c}\right)  $, linked
with the Hessian matrix of the test function:%
\[
\mathcal{I}_{t}^{c,T}=\frac{1}{2}\int_{t}^{T}D_{\left(  x,y\right)  }^{2}%
\Phi_{r}\left(  Y_{r-}^{\varepsilon},Y_{r-}^{\delta}\right)  \left(
dM_{r}^{\varepsilon,c},dM_{r}^{\delta,c}\right)  ^{2}=\frac{1}{2}\int_{t}%
^{T}D_{\left(  x,y\right)  }^{2}\Phi_{r}\left(  Y_{r-}^{\varepsilon}%
,Y_{r-}^{\delta}\right)  d\left[  \left(  M^{\varepsilon,c},M^{\delta
,c}\right)  \right]  _{r}~.
\]
In details,%
\[
\left\{
\begin{array}
[c]{rcl}%
D_{\left(  x,y\right)  }^{2}\Phi_{r} & = & \left(
\begin{array}
[c]{cc}%
D_{xx}^{2}\Phi_{r} & D_{xy}^{2}\Phi_{r}\medskip\\
D_{yx}^{2}\Phi_{r} & D_{yy}^{2}\Phi_{r}%
\end{array}
\right)  \in\mathbb{R}^{4\times4}\quad\quad\text{and}\medskip\\
d\left[  \left(  M^{\varepsilon,c},M^{\delta,c}\right)  \right]  _{r} & = &
\left(
\begin{array}
[c]{cc}%
d\left[  M^{\varepsilon,c}\right]  _{r} & d\left[  M^{\varepsilon,c}%
,M^{\delta,c}\right]  _{r}\medskip\\
d\left[  M^{\delta,c},M^{\varepsilon,c}\right]  _{r} & d\left[  M^{\delta
,c}\right]  _{r}%
\end{array}
\right)  \in\mathbb{R}^{4\times4}%
\end{array}
\right.
\]
Since $Y^{\varepsilon}$ admits jumps produced only by the martingale process,
$\Delta Y_{r}^{\varepsilon}=\Delta M_{r}^{\varepsilon}~$, where we denoted
$\Delta Y_{r}^{\varepsilon}:=Y_{r}^{\varepsilon}-Y_{r-}^{\varepsilon}$ and,
similarly, for $\Delta M_{r}^{\varepsilon}=M_{r}^{\varepsilon}-M_{r-}%
^{\varepsilon}~.$ Therefore,%
\[%
\begin{array}
[c]{ccl}%
\mathcal{I}_{t}^{d,T} & = &
{\displaystyle\sum\limits_{t<r\leq T}}
\left[  \Phi_{r}\left(  Y_{r}^{\varepsilon},Y_{r}^{\delta}\right)  -\Phi
_{r}\left(  Y_{r-}^{\varepsilon},Y_{r-}^{\delta}\right)  \right.  \medskip\\
&  & -\left.  \left\langle \nabla_{x}\Phi_{r}\left(  Y_{r-}^{\varepsilon
},Y_{r-}^{\delta}\right)  ,\Delta M_{r}^{\varepsilon}\right\rangle
-\left\langle \nabla_{y}\Phi_{r}\left(  Y_{r-}^{\varepsilon},Y_{r-}^{\delta
}\right)  ,\Delta M_{r}^{\delta}\right\rangle \right]  .
\end{array}
\]
Insert now in (\ref{Ito for vector}) the dynamics of $Y^{\varepsilon}$ and
$Y^{\delta}$. We obtain%
\begin{equation}%
\begin{array}
[c]{lll}%
\varphi_{t} & = &
{\displaystyle\int_{t}^{T}}
\left\langle \nabla_{x}\Phi_{r}\left(  Y_{r}^{\varepsilon},Y_{r}^{\delta
}\right)  ,F_{r}^{\varepsilon}\right\rangle dr-%
{\displaystyle\int_{t}^{T}}
\left\langle \nabla_{x}\Phi_{r}\left(  Y_{r}^{\varepsilon},Y_{r}^{\delta
}\right)  ,\Theta_{r}U_{r}^{\varepsilon}\right\rangle dr\medskip\\
& + &
{\displaystyle\int_{t}^{T}}
\left\langle \nabla_{y}\Phi_{r}\left(  Y_{r}^{\varepsilon},Y_{r}^{\delta
}\right)  ,F_{r}^{\delta}\right\rangle dr-%
{\displaystyle\int_{t}^{T}}
\left\langle \nabla_{y}\Phi_{r}\left(  Y_{r}^{\varepsilon},Y_{r}^{\delta
}\right)  ,\Theta_{r}U_{r}^{\delta}\right\rangle dr\medskip\\
& - &
{\displaystyle\int_{t}^{T}}
\left\langle \nabla_{x}\Phi_{r}\left(  Y_{r-}^{\varepsilon},Y_{r-}^{\delta
}\right)  ,dM_{r}^{\varepsilon}\right\rangle -%
{\displaystyle\int_{t}^{T}}
\left\langle \nabla_{y}\Phi_{r}\left(  Y_{r-}^{\varepsilon},Y_{r-}^{\delta
}\right)  ,dM_{r}^{\delta}\right\rangle \medskip\\
& - &
{\displaystyle\int_{t}^{T}}
\partial_{r}\Phi_{r}\left(  Y_{r}^{\varepsilon},Y_{r}^{\delta}\right)
dr+\mathcal{I}_{t}^{T}.
\end{array}
\label{Ito expanded}%
\end{equation}
Written under differential form, it becomes (we omit, for the simplicity of
the presentation to mention the pair of arguments $\left(  Y_{r}^{\varepsilon
},Y_{r}^{\delta}\right)  $; the differentiation in (\ref{Ito expanded}) is
with respect to $t$, followed by the use of the running variable $r$):%
\begin{equation}
-d\varphi_{r}=\left[  \left\langle \nabla_{x}\Phi_{r},F_{r}^{\varepsilon
}\right\rangle +\left\langle \nabla_{y}\Phi_{r},F_{r}^{\delta}\right\rangle
-\partial_{r}\Phi_{r}\right]  dr-\left[  \left\langle \nabla_{x}\Phi
_{r},\Theta_{r}U_{r}^{\varepsilon}\right\rangle +\left\langle \nabla_{y}%
\Phi_{r},\Theta_{r}U_{r}^{\delta}\right\rangle \right]  dr-dm_{r}%
-d\mathcal{I}_{r},\label{r20}%
\end{equation}
where $m$ is a local martingale with $m_{0}=0$ and%
\[%
\begin{array}
[c]{ccl}%
dm_{r} & = & \left\langle \nabla_{x}\Phi_{r}\left(  Y_{r-}^{\varepsilon
},Y_{r-}^{\delta}\right)  ,dM_{r}^{\varepsilon}\right\rangle +\left\langle
\nabla_{y}\Phi_{r}\left(  Y_{r-}^{\varepsilon},Y_{r-}^{\delta}\right)
,dM_{r}^{\delta}\right\rangle \medskip\\
& = & \left\langle \nabla_{x}\Phi_{r},dM_{r}^{\varepsilon}\right\rangle
+\left\langle \nabla_{y}\Phi_{r},dM_{r}^{\delta}\right\rangle =\left\langle
\nabla_{x}\Phi_{r},dM_{r}^{\varepsilon}\pm dM_{r}^{\delta}\right\rangle
+\left\langle \nabla_{y}\Phi_{r},dM_{r}^{\delta}\right\rangle \medskip\\
& = & \left\langle \nabla_{x}\Phi_{r},d\left(  M^{\varepsilon,\delta}\right)
_{r}\right\rangle +\left\langle \nabla_{x}\Phi_{r}+\nabla_{y}\Phi_{r}%
,dM_{r}^{\delta}\right\rangle .
\end{array}
\]
We start below to deal with the controllable terms.

\smallskip\noindent\textbf{Step 1: the second order term control.} For its
continuous part, from Lemma \ref{rot-hess}, formula (\ref{r17}) applies for
the pairs $\left(  x,y\right)  =\left(  Y_{r-}^{\varepsilon},Y_{r-}^{\delta
}\right)  ,$ $\left(  h,h^{\prime}\right)  =(dM_{r}^{\varepsilon,c}%
,dM_{r}^{\delta,c})$. We have $x-y=Y_{r-}^{\varepsilon,\delta},$ $h-h^{\prime
}=d(M_{r}^{\varepsilon,\delta,c})$ and, as measures, we obtain $\left\vert
h\right\vert ^{2}+|h^{\prime}|^{2}\longleftrightarrow d\left[  M^{\varepsilon
,c}\right]  _{r}+d\left[  M^{\delta,c}\right]  _{r},$ $\left\vert h-h^{\prime
}\right\vert ^{2}\longleftrightarrow d\left[  M^{\varepsilon,\delta,c}\right]
_{r}.$ It yields the following lower bound for the integrand:%
\begin{equation}
\dfrac{1}{2}D_{\left(  x,y\right)  }^{2}\Phi_{r}\left(  Y_{r-}^{\varepsilon
},Y_{r-}^{\delta}\right)  \left(  dM_{r}^{\varepsilon,c},dM_{r}^{\delta
,c}\right)  ^{2}\geq\frac{\alpha_{\theta}}{2}d\left[  M^{\varepsilon,\delta
,c}\right]  _{r}-\kappa_{\theta}\left\vert Y_{r-}^{\varepsilon,\delta
}\right\vert ^{2}d\left(  \left[  M^{\varepsilon,c}\right]  +[M^{\delta
,c}]\right)  _{r}~.\label{cont part}%
\end{equation}

Similarly, for the jump at time $r$, Taylor's formula gives a point $\zeta
_{r}$ on the segment joining $\left(  Y_{r-}^{\varepsilon},Y_{r-}^{\delta
}\right)  $ to $\left(  Y_{r}^{\varepsilon},Y_{r}^{\delta}\right)  $, at which
the summand equals $\frac{1}{2}D^{2}\Phi_{r}\left(  \zeta_{r}\right)  [\left(
h_{r},h_{r}^{\prime}\right)  ^{2}]$, where $h_{r}:=\Delta M_{r}^{\varepsilon}$
and $h_{r}^{\prime}:=\Delta M_{r}^{\delta}$. Since the finite variation parts
of $Y^{\varepsilon}$ and $Y^{\delta}$ are continuous, at $\zeta_{r}=\left(
x,y\right)  $ one has%
\[
\left\vert x-y\right\vert \leq|Y_{r-}^{\varepsilon,\delta}|+\left\vert
h_{r}-h_{r}^{\prime}\right\vert ,
\]
and Lemma \ref{rot-hess}, together with $\left\vert h_{r}\right\vert
,\left\vert h_{r}^{\prime}\right\vert \leq j$, assured by Proposition
\ref{rot-apriori}-$\left(  b\right)  $) give%
\begin{equation}
\dfrac{1}{2}D^{2}\Phi_{r}\left(  \zeta_{r}\right)  \left[  \left(  h_{r}%
,h_{r}^{\prime}\right)  ^{2}\right]  \geq\left(  \dfrac{\alpha_{\theta}}%
{2}-4\kappa_{\theta}j^{2}\right)  \left\vert h_{r}-h_{r}^{\prime}\right\vert
^{2}-2\kappa_{\theta}|Y_{r-}^{\varepsilon,\delta}|^{2}\left(  \left\vert
h_{r}\right\vert ^{2}+\left\vert h_{r}^{\prime}\right\vert ^{2}\right)
,\label{jump part}%
\end{equation}
with $4\kappa_{\theta}j^{2}\leq\alpha_{\theta}/4$, according to the smallness
condition $\left(  C_{3}\right)  $. Adding the two estimates above,
(\ref{cont part}) and (\ref{jump part}), we get%
\begin{equation}
-d\mathcal{I}_{r}\leq-\dfrac{\alpha_{\theta}}{4}d[M^{\varepsilon,\delta}%
]_{r}+2\kappa_{\theta}|Y_{r-}^{\varepsilon,\delta}|^{2}d\left(  \left[
M^{\varepsilon}\right]  +[M^{\delta}]\right)  _{r}.\label{r21}%
\end{equation}
This is the term (\ref{r18a}), and it is the only one that is not compatible
with a deterministic Gronwall measure. At the first impression, the second
term from the RHS should have the coefficient $3$, instead of $2.$ However,
the involving coefficients from the two estimates, (\ref{cont part}) and
(\ref{jump part}), correspond to disjoints measures, one continuous, diffusive
and another one, discrete and with pure jumps. After the orthogonal
decomposition of the quadratic variation $d\left[  M\right]  _{r}$, one must
choose a uniform constant, suitable for both disjoint supports.

\smallskip\noindent\textbf{Step 2: the reflection terms control.} If
$U_{r}^{\varepsilon}\neq0$ then $\bar{Y}_{r}^{\varepsilon}=\pi_{K}\left(
Y_{r}^{\varepsilon}\right)  \in bd\left(  K\right)  $ and $U_{r}^{\varepsilon
}=\left\vert U_{r}^{\varepsilon}\right\vert \mathbf{n}\left(  \bar{Y}%
_{r}^{\varepsilon}\right)  $. Since $\bar{Y}_{r}^{\delta}\in K$, we revisit
the angle compatibility constraints section and the condition $\left(
C_{2}\right)  $ given by (\ref{r15}) can be applied with the point $\left(
\bar{Y}_{r}^{\varepsilon},\bar{Y}_{r}^{\delta}\right)  .$ It gives%
\[
\left\langle \nabla_{x}\Phi_{r}(\bar{Y}_{r}^{\varepsilon},\bar{Y}_{r}^{\delta
}),\Theta_{r}U_{r}^{\varepsilon}\right\rangle \geq0.
\]
Since $\nabla_{x}\Phi_{r}$ is globally Lipschitz with respect to both space
variables, with the positive constant $c_{\Phi}$ and $\left\vert
Y_{r}^{\varepsilon}-\bar{Y}_{r}^{\varepsilon}\right\vert =\varepsilon
\left\vert U_{r}^{\varepsilon}\right\vert $ and, similarly, $\left\vert
Y_{r}^{\delta}-\bar{Y}_{r}^{\delta}\right\vert =\delta\left\vert U_{r}%
^{\delta}\right\vert $, we obtain%
\begin{equation}
\left\vert \nabla_{x}\Phi_{r}(Y_{r}^{\varepsilon},Y_{r}^{\delta})-\nabla
_{x}\Phi_{r}(\bar{Y}_{r}^{\varepsilon},\bar{Y}_{r}^{\delta})\right\vert \leq
c_{\Phi}\left(  \varepsilon\left\vert U_{r}^{\varepsilon}\right\vert
+\delta|U_{r}^{\delta}|\right)  .\label{Lip estim}%
\end{equation}
Take the scalar product from (\ref{r20}) and it can be written as (recall that
$\Theta_{r}$ is a rotation matrix, so its isometry property gives $\left\vert
\Theta_{r}U_{r}^{\varepsilon}\right\vert =\left\vert U_{r}^{\varepsilon
}\right\vert $):%
\[%
\begin{array}
[c]{l}%
\left\langle \nabla_{x}\Phi_{r}\left(  Y_{r}^{\varepsilon},Y_{r}^{\delta
}\right)  ,\Theta_{r}U_{r}^{\varepsilon}\right\rangle =\left\langle \nabla
_{x}\Phi_{r}\left(  \bar{Y}_{r}^{\varepsilon},\bar{Y}_{r}^{\delta}\right)
,\Theta_{r}U_{r}^{\varepsilon}\right\rangle +\left\langle \nabla_{x}\Phi
_{r}\left(  Y_{r}^{\varepsilon},Y_{r}^{\delta}\right)  -\nabla_{x}\Phi
_{r}\left(  \bar{Y}_{r}^{\varepsilon},\bar{Y}_{r}^{\delta}\right)  ,\Theta
_{r}U_{r}^{\varepsilon}\right\rangle \medskip\\
\quad\quad\quad\quad\geq0-\left\vert \nabla_{x}\Phi_{r}\left(  Y_{r}%
^{\varepsilon},Y_{r}^{\delta}\right)  -\nabla_{x}\Phi_{r}\left(  \bar{Y}%
_{r}^{\varepsilon},\bar{Y}_{r}^{\delta}\right)  \right\vert \left\vert
U_{r}^{\varepsilon}\right\vert \geq-c_{\Phi}\left(  \varepsilon\left\vert
U_{r}^{\varepsilon}\right\vert +\delta\left\vert U_{r}^{\delta}\right\vert
\right)  \left\vert U_{r}^{\varepsilon}\right\vert ,\quad\text{from
(\ref{Lip estim}).}%
\end{array}
\]
Arguing symmetrically for the second, similar, term from (\ref{r20}), we
obtain%
\[
\left\{
\begin{array}
[c]{l}%
\left\langle \nabla_{x}\Phi_{r}\left(  Y_{r}^{\varepsilon},Y_{r}^{\delta
}\right)  ,\Theta_{r}U_{r}^{\varepsilon}\right\rangle \geq-c_{\Phi}\left(
\varepsilon\left\vert U_{r}^{\varepsilon}\right\vert +\delta\left\vert
U_{r}^{\delta}\right\vert \right)  \left\vert U_{r}^{\varepsilon}\right\vert
\medskip\\
\left\langle \nabla_{y}\Phi_{r}\left(  Y_{r}^{\varepsilon},Y_{r}^{\delta
}\right)  ,\Theta_{r}U_{r}^{\delta}\right\rangle \geq-c_{\Phi}\left(
\varepsilon\left\vert U_{r}^{\varepsilon}\right\vert +\delta\left\vert
U_{r}^{\delta}\right\vert \right)  \left\vert U_{r}^{\delta}\right\vert .
\end{array}
\right.
\]
Hence, due to the sign in front, the second bracket of (\ref{r20}) is bounded
from above by $d\mathcal{L}_{r}$, where%
\begin{equation}
\mathcal{L}_{t}:=2c_{\Phi}{%
{\displaystyle\int_{0}^{t}}
}\left(  \varepsilon\left\vert U_{r}^{\varepsilon}\right\vert ^{2}%
+\delta|U_{r}^{\delta}|^{2}+\left(  \varepsilon+\delta\right)  \left\vert
U_{r}^{\varepsilon}\right\vert |U_{r}^{\delta}|\right)  dr\leq3c_{\Phi}\left(
\varepsilon+\delta\right)  \mathcal{U}_{t}~,\label{r22}%
\end{equation}
where $\mathcal{U}_{t}:={\int_{0}^{t}}\left(  \left\vert U_{r}^{\varepsilon
}\right\vert ^{2}+\left\vert U_{r}^{\delta}\right\vert ^{2}\right)  dr.$

\smallskip\noindent\textbf{Step 3: the generator and the time derivative
control.} By an easy algebraic manipulation of the terms, we have%
\[
\left\langle \nabla_{x}\Phi_{r},F_{r}^{\varepsilon}\right\rangle +\left\langle
\nabla_{y}\Phi_{r},F_{r}^{\delta}\right\rangle =\left\langle \nabla_{x}%
\Phi_{r},F_{r}^{\varepsilon}-F_{r}^{\delta}\right\rangle +\left\langle
\nabla_{x}\Phi_{r}+\nabla_{y}\Phi_{r},F_{r}^{\delta}\right\rangle .
\]
Use now (\ref{r11}), (\ref{r12}), Assumption \ref{H2}-$\left(  i\right)  $ and
$|Y_{r}^{\varepsilon,\delta}|^{2}\leq\alpha_{\theta}^{-1}\varphi_{r}$, which
is (\ref{r16}). For the first bracket, we have%
\[%
\begin{array}
[c]{l}%
\left\langle \nabla_{x}\Phi_{r},F_{r}^{\varepsilon}-F_{r}^{\delta
}\right\rangle \leq\left\vert \nabla_{x}\Phi_{r}\right\vert \left\vert
F_{r}^{\varepsilon}-F_{r}^{\delta}\right\vert \overset{(\ref{r11})}{\leq
}c_{\Phi}|Y_{r}^{\varepsilon,\delta}|\left(  L_{r}|Y_{r}^{\varepsilon,\delta
}|+\beta_{r}|\mathcal{R}_{r}^{\varepsilon,\delta}|\right)  \medskip\\
\quad\quad\quad=c_{\Phi}L_{r}|Y_{r}^{\varepsilon,\delta}|^{2}+c_{\Phi}%
\beta_{r}|Y_{r}^{\varepsilon,\delta}||\mathcal{R}_{r}^{\varepsilon,\delta
}|\leq c_{\Phi}L_{r}|Y_{r}^{\varepsilon,\delta}|^{2}+\dfrac{c_{\Phi}^{2}%
\beta_{r}^{2}}{2\gamma}|Y_{r}^{\varepsilon,\delta}|^{2}+\dfrac{\gamma}%
{2}|\mathcal{R}_{r}^{\varepsilon,\delta}|^{2}\medskip\\
\quad\quad\quad=\left(  c_{\Phi}L_{r}+\dfrac{c_{\Phi}^{2}\beta_{r}^{2}%
}{2\gamma}\right)  |Y_{r}^{\varepsilon,\delta}|^{2}+\dfrac{\gamma}%
{2}|\mathcal{R}_{r}^{\varepsilon,\delta}|^{2}\overset{(C_{1})}{\leq}\dfrac
{1}{\alpha_{\theta}}\left(  c_{\Phi}L_{r}+\dfrac{c_{\Phi}^{2}\beta_{r}^{2}%
}{2\gamma}\right)  \varphi_{r}+\dfrac{\gamma}{2}|\mathcal{R}_{r}%
^{\varepsilon,\delta}|^{2}.
\end{array}
\]
For the second bracket, $\left\langle \nabla_{x}\Phi_{r}+\nabla_{y}\Phi
_{r},F_{r}^{\delta}\right\rangle $, the symmetry of $\Phi_{r}$ produces%
\[
\left\langle \nabla_{x}\Phi_{r}+\nabla_{y}\Phi_{r},F_{r}^{\delta}\right\rangle
\leq\left\vert \nabla_{x}\Phi_{r}+\nabla_{y}\Phi_{r}\right\vert \left\vert
F_{r}^{\delta}\right\vert \overset{(\ref{r12})}{\leq}c_{\Phi}\left\vert
Y_{r}^{\varepsilon,\delta}\right\vert ^{2}\left\vert F_{r}^{\delta}\right\vert
\leq\dfrac{c_{\Phi}}{\alpha_{\theta}}\left\vert F_{r}^{\delta}\right\vert
\varphi_{r}.
\]
For the temporal derivative term we use the definition of the test function
$\Phi_{r}\left(  x,y\right)  $ and (\ref{r11}). Hence,%
\[%
\begin{array}
[c]{ccl}%
-\partial_{r}\Phi_{r}\left(  Y_{r}^{\varepsilon},Y_{r}^{\delta}\right)   &
\leq & \left\vert \partial_{r}\Phi_{r}\left(  Y_{r}^{\varepsilon}%
,Y_{r}^{\delta}\right)  \right\vert =\dfrac{\left\vert \theta_{r}^{\prime
}\right\vert }{2\cos^{2}\theta_{r}}\left\vert P\left(  Y_{r}^{\varepsilon
},Y_{r}^{\delta}\right)  \right\vert \left\vert Q\left(  Y_{r}^{\varepsilon
},Y_{r}^{\delta}\right)  \right\vert \medskip\\
& \leq & c_{\Phi}\left\vert \left\vert \theta^{\prime}\right\vert \right\vert
_{\infty}|Y_{r}^{\varepsilon,\delta}|^{2}\overset{(\ref{r16})}{\leq}%
\dfrac{c_{\Phi}\left\vert \left\vert \theta^{\prime}\right\vert \right\vert
_{\infty}}{\alpha_{\theta}}\varphi_{r}.
\end{array}
\]
We now add the previous three estimates and it yields that the first bracket
from (\ref{r20}) is bounded from above by%
\begin{equation}
\dfrac{1}{2}\varphi_{r}dV_{r}^{\left(  1\right)  }+\dfrac{\gamma}%
{2}|\mathcal{R}_{r}^{\varepsilon,\delta}|^{2}dr,\quad\text{with}\quad
dV_{r}^{\left(  1\right)  }:=\dfrac{2c_{\Phi}}{\alpha_{\theta}}\left(
L_{r}+\dfrac{c_{\Phi}\ell_{r}^{2}}{2\gamma}+\left\vert \left\vert
\theta^{\prime}\right\vert \right\vert _{\infty}+\left\vert F_{r}^{\delta
}\right\vert \right)  dr,\label{r23}%
\end{equation}
$dV^{\left(  1\right)  }$ being the absolutely continuous part of $dV$ given
by (\ref{r19}).

\smallskip\noindent\textbf{Step 4: the weighted inequality.} Let us recall the
process $V$, introduced in (\ref{r19}):%
\[
V_{t}:=2\nu_{\theta}\left(  \left[  M^{\varepsilon}\right]  _{t}+[M^{\delta
}]_{t}\right)  +dV_{t}^{\left(  1\right)  },\quad\nu_{\theta}=\dfrac
{4\kappa_{\theta}}{\alpha_{\theta}}\quad\text{and}\quad\Gamma_{t-}:=e^{V_{t-}%
}\geq1.
\]
$V^{\left(  1\right)  }$ is the absolutely continuous part of $V$, which
controls the generator and the time derivative, while $\Gamma_{\cdot-}$ is the
left limit value, being a predictable process. Since $V_{t}$ is a process of
finite variation, $d\left(  \Gamma\varphi\right)  _{r}=\Gamma_{r-}d\varphi
_{r}+\varphi_{r}d\Gamma_{r}~$, while%
\begin{equation}
\Delta\Gamma_{r}=\Gamma_{r-}\left(  e^{\Delta V_{r}}-1\right)  \geq\Gamma
_{r-}\Delta V_{r}\quad\Longrightarrow\quad d\Gamma_{r}\geq\Gamma_{r-}%
dV_{r}~.\label{ineg dV dGamma}%
\end{equation}
The integration by parts formula for c\`{a}dl\`{a}g semimartingales, gives, if
applied to $-d\left(  \Gamma\varphi\right)  $, on $\left[  t,T\right]  $,%
\begin{equation}
\Gamma_{t}\varphi_{t}=\Gamma_{T}\varphi_{_{T}}+{\int_{t}^{T}}\Gamma
_{r-}\left(  -d\varphi_{r}\right)  -{\int_{t}^{T}}\varphi_{r}d\Gamma
_{r}\label{dyn of GammaFi}%
\end{equation}
Insert in the obtained formula the dynamic of $-d\varphi_{r}$, which is
provided by (\ref{r20}):%
\[%
\begin{array}
[c]{c}%
-d\varphi_{r}=\left[  \left\langle \nabla_{x}\Phi_{r},F_{r}^{\varepsilon
}\right\rangle +\left\langle \nabla_{y}\Phi_{r},F_{r}^{\delta}\right\rangle
-\partial_{r}\Phi_{r}\right]  dr-\left[  \left\langle \nabla_{x}\Phi
_{r},\Theta_{r}U_{r}^{\varepsilon}\right\rangle +\left\langle \nabla_{y}%
\Phi_{r},\Theta_{r}U_{r}^{\delta}\right\rangle \right]  dr-dm_{r}%
-d\mathcal{I}_{r}\medskip\\
=\left(  \text{\textsc{Generator}}\right)  dr-\left(
\text{\textsc{Reflection}}\right)  dr-dm_{r}-d\mathcal{I}_{r}%
\end{array}
\]
The dynamic (\ref{dyn of GammaFi}) becomes%
\[
\Gamma_{t}\varphi_{t}={%
{\displaystyle\int_{t}^{T}}
}\Gamma_{r-}\left(  \left(  \text{\textsc{Generator}}\right)  dr-\left(
\text{\textsc{Reflection}}\right)  dr-dm_{r}-d\mathcal{I}_{r}\right)
-{\int_{t}^{T}}\varphi_{r}d\Gamma_{r}.
\]
The bounds for the terms containing $\left(  \text{\textsc{Generator}}\right)
dr$ - (\ref{r22}, Step 2.) and $\left(  \text{\textsc{Reflection}}\right)  dr$
- (\ref{r23}, Step 3.) gives us%
\begin{equation}%
\begin{array}
[c]{ccl}%
\Gamma_{t}\varphi_{t} & \leq & \dfrac{\gamma}{2}%
{\displaystyle\int_{t}^{T}}
\Gamma_{r-}|\mathcal{R}_{r}^{\varepsilon,\delta}|^{2}dr+%
{\displaystyle\int_{t}^{T}}
\Gamma_{r-}d\mathcal{L}_{r}-%
{\displaystyle\int_{t}^{T}}
\Gamma_{r-}dm_{r}\medskip\\
& + & \dfrac{1}{2}%
{\displaystyle\int_{t}^{T}}
\Gamma_{r-}\varphi_{r}dV_{r}^{\left(  1\right)  }-%
{\displaystyle\int_{t}^{T}}
\Gamma_{r-}d\mathcal{I}_{r}-{\int_{t}^{T}}\varphi_{r}d\Gamma_{r}%
\end{array}
\label{eq to analyze}%
\end{equation}
Let us investigate the key elements, i.e. $-%
{\textstyle\int_{t}^{T}}
\Gamma_{r-}d\mathcal{I}_{r}-%
{\textstyle\int_{t}^{T}}
\varphi_{r}d\Gamma_{r}.$ Let's treat the first integral. According to Step 1,
formula (\ref{r21}), we have%
\[
-d\mathcal{I}_{r}\leq-\dfrac{\alpha_{\theta}}{4}d[M^{\varepsilon,\delta}%
]_{r}+2\kappa_{\theta}|Y_{r-}^{\varepsilon,\delta}|^{2}d\left(  \left[
M^{\varepsilon}\right]  +[M^{\delta}]\right)  _{r}.
\]
Therefore, using the elementary inequality $|Y_{r-}^{\varepsilon,\delta}%
|^{2}\leq2|Y_{r}^{\varepsilon,\delta}|^{2}+2|M_{r}^{\varepsilon,\delta}%
-M_{r-}^{\varepsilon,\delta}|^{2}$,%
\[%
\begin{array}
[c]{ccl}%
-%
{\displaystyle\int_{t}^{T}}
\Gamma_{r-}d\mathcal{I}_{r} & \leq & -\dfrac{\alpha_{\theta}}{4}%
{\displaystyle\int_{t}^{T}}
\Gamma_{r-}d\left[  M^{\varepsilon,\delta}\right]  _{r}+2\kappa_{\theta}%
{\displaystyle\int_{t}^{T}}
\Gamma_{r-}|Y_{r-}^{\varepsilon,\delta}|^{2}d\left(  \left[  M^{\varepsilon
}\right]  +\left[  M^{\delta}\right]  \right)  _{r}\medskip\\
& \leq & -\dfrac{\alpha_{\theta}}{4}%
{\displaystyle\int_{t}^{T}}
\Gamma_{r-}d\left[  M^{\varepsilon,\delta}\right]  _{r}+4\kappa_{\theta}%
{\displaystyle\int_{t}^{T}}
\Gamma_{r-}|Y_{r}^{\varepsilon,\delta}|^{2}d\left(  \left[  M^{\varepsilon
}\right]  +\left[  M^{\delta}\right]  \right)  _{r}\medskip\\
& + & 4\kappa_{\theta}%
{\displaystyle\sum\limits_{t<r\leq T}}
\Gamma_{r-}|M_{r}^{\varepsilon,\delta}-M_{r-}^{\varepsilon,\delta}|^{2}\left(
\left\vert \Delta M_{r}^{\varepsilon}\right\vert ^{2}+\left\vert \Delta
M_{r}^{\delta}\right\vert ^{2}\right)
\end{array}
\]
The last term contains
\[
\left\vert \Delta M_{r}-\Delta M_{r-}\right\vert ^{2}=\left\vert
\Delta(M^{\varepsilon,\delta})_{r}\right\vert ^{2}=\left\vert \Delta
(M_{r}^{\varepsilon}-M_{r}^{\delta})\right\vert ^{2}=\left\vert M_{r}%
^{\varepsilon}-M_{r-}^{\varepsilon}-(M_{r}^{\delta}-M_{r-}^{\delta
})\right\vert ^{2}.
\]
Since the jumps are uniformly bounded by Proposition \ref{rot-apriori}%
-$\left(  b\right)  $, formula (\ref{r18}), then, for $j=D_{K}+2$,%
\[%
\begin{array}
[c]{l}%
4\kappa_{\theta}%
{\displaystyle\sum\limits_{t<r\leq T}}
\Gamma_{r-}\left\vert \Delta\left(  M^{\varepsilon,\delta}\right)
_{r}\right\vert ^{2}\left(  \left\vert \Delta M_{r}^{\varepsilon}\right\vert
^{2}+\left\vert \Delta M_{r}^{\delta}\right\vert ^{2}\right)  \medskip\\
\quad\quad\leq4\kappa_{\theta}%
{\displaystyle\sum\limits_{t<r\leq T}}
\Gamma_{r-}\left\vert \Delta\left(  M^{\varepsilon,\delta}\right)
_{r}\right\vert ^{2}\left(  2j^{2}\right)  =8\kappa_{\theta}j^{2}%
{\displaystyle\int_{t}^{T}}
\Gamma_{r-}d\left[  \Delta M^{d}\right]  _{r}%
\end{array}
\]
From the smallness condition (\ref{C3}) $64\kappa_{\theta}\Lambda_{\ast}%
^{2}<\alpha_{\theta}$ we obtain $8\kappa_{\theta}j^{2}\leq\alpha_{\theta}/8,$
that is, this sum of jumps is partially absorbed by $-\frac{\alpha_{\theta}%
}{4}d\left[  M^{\varepsilon,\delta}\right]  _{r}$. It follows that%
\[
-\frac{\alpha_{\theta}}{4}%
{\displaystyle\int_{t}^{T}}
\Gamma_{r-}d[M^{\varepsilon,\delta}]_{r}+8\kappa_{\theta}j^{2}%
{\displaystyle\int_{t}^{T}}
\Gamma_{r-}d[\Delta M^{d}]_{r}\leq-\frac{\alpha_{\theta}}{8}%
{\displaystyle\int_{t}^{T}}
\Gamma_{r-}d[M^{\varepsilon,\delta}]_{r}~\,
\]
For the contribution in (\ref{eq to analyze}) of the integral$-%
{\textstyle\int_{t}^{T}}
\varphi_{r}d\Gamma_{r}$, we proceed as follows.

Since $\Gamma_{r}=e^{Vr}$ and $V_{r}:=2\nu_{\theta}\left(  \left[
M^{\varepsilon}\right]  _{r}+[M^{\delta}]_{r}\right)  +dV_{r}^{\left(
1\right)  }$, with $\nu_{\theta}=\dfrac{4\kappa_{\theta}}{\alpha_{\theta}}$,
we have the inequality of convex measure%
\[
d\Gamma_{r}\geq\Gamma_{r-}dV_{r}=2\nu_{\theta}\Gamma_{r-}d\left(  \left[
M^{\varepsilon}\right]  +[M^{\delta}]\right)  _{r}+\Gamma_{r-}dV_{r}^{\left(
1\right)  }%
\]
and our integral satisfies%
\[
-%
{\textstyle\int_{t}^{T}}
\varphi_{r}d\Gamma_{r}\leq-2\nu_{\theta}%
{\displaystyle\int_{t}^{T}}
\Gamma_{r-}\varphi_{r}d\left(  \left[  M^{\varepsilon}\right]  +[M^{\delta
}]\right)  _{r}-%
{\displaystyle\int_{t}^{T}}
\Gamma_{r-}\varphi_{r}dV_{r}^{\left(  1\right)  }.
\]
Also, from the coercivity of the test function $\varphi_{r}=\Phi_{r}\left(
Y_{r}^{\varepsilon},Y_{r}^{\delta}\right)  \geq\alpha_{\theta}|Y_{r}%
^{\varepsilon,\delta}|^{2}$ we find $|Y_{r}^{\varepsilon,\delta}|^{2}\leq
\frac{1}{\alpha_{_{\theta}}}\varphi_{r}$ and, due to the formula of
$\nu_{\theta}$,%
\begin{equation}
-2\nu_{\theta}\varphi_{r}+4\kappa_{\theta}|Y_{r}^{\varepsilon,\delta}|^{2}%
\leq-2\nu_{\theta}\varphi_{r}+4\kappa_{\theta}\dfrac{1}{\alpha_{\theta}%
}\varphi_{r}=-2\nu_{\theta}\varphi_{r}+\nu_{\theta}\varphi_{r}=-\nu_{\theta
}\varphi_{r}\label{comb terms}%
\end{equation}
Insert the estimates for these two integrals into (\ref{eq to analyze}) and it
follows, after we group the terms featuring $dV_{r}^{\left(  1\right)  }$ and,
respectively, $d\left(  \left[  M^{\varepsilon}\right]  +\left[  M^{\delta
}\right]  \right)  _{r}~$,
\[%
\begin{array}
[c]{ccl}%
\Gamma_{t}\varphi_{t} & \leq & \dfrac{\gamma}{2}%
{\displaystyle\int_{t}^{T}}
\Gamma_{r-}|\mathcal{R}_{r}^{\varepsilon,\delta}|^{2}dr+%
{\displaystyle\int_{t}^{T}}
\Gamma_{r-}d\mathcal{L}_{r}-%
{\displaystyle\int_{t}^{T}}
\Gamma_{r-}dm_{r}\medskip\\
& + & \dfrac{1}{2}%
{\displaystyle\int_{t}^{T}}
\Gamma_{r-}\varphi_{r}dV_{r}^{\left(  1\right)  }-2\nu_{\theta}%
{\displaystyle\int_{t}^{T}}
\Gamma_{r-}\varphi_{r}d\left(  \left[  M^{\varepsilon}\right]  +\left[
M^{\delta}\right]  \right)  _{r}-%
{\displaystyle\int_{t}^{T}}
\Gamma_{r-}\varphi_{r}dV_{r}^{\left(  1\right)  }\medskip\\
& + & 4\kappa_{\theta}%
{\displaystyle\int_{t}^{T}}
\Gamma_{r-}|Y_{r}^{\varepsilon,\delta}|^{2}d\left(  \left[  M^{\varepsilon
}\right]  +\left[  M^{\delta}\right]  \right)  _{r}-\dfrac{\alpha_{\theta}}{8}%
{\displaystyle\int_{t}^{T}}
\Gamma_{r-}d\left[  M^{\varepsilon,\delta}\right]  _{r}\medskip\\
& \leq & \dfrac{\gamma}{2}%
{\displaystyle\int_{t}^{T}}
\Gamma_{r-}|\mathcal{R}_{r}^{\varepsilon,\delta}|^{2}dr+%
{\displaystyle\int_{t}^{T}}
\Gamma_{r-}d\mathcal{L}_{r}-%
{\displaystyle\int_{t}^{T}}
\Gamma_{r-}dm_{r}-\dfrac{\alpha_{\theta}}{8}%
{\displaystyle\int_{t}^{T}}
\Gamma_{r-}d\left[  M^{\varepsilon,\delta}\right]  _{r}-\dfrac{1}{2}%
{\displaystyle\int_{t}^{T}}
\varphi_{r}d\Gamma_{r}~,\medskip
\end{array}
\]
since, using the inequality (\ref{comb terms}), and (\ref{ineg dV dGamma}) for
the last inequality sign,%
\[%
\begin{array}
[c]{l}%
-\dfrac{1}{2}%
{\displaystyle\int_{t}^{T}}
\Gamma_{r-}\varphi_{r}dV_{r}^{\left(  1\right)  }+4\kappa_{\theta}%
{\displaystyle\int_{t}^{T}}
\Gamma_{r-}|Y_{r}^{\varepsilon,\delta}|^{2}d\left(  \left[  M^{\varepsilon
}\right]  +\left[  M^{\delta}\right]  \right)  _{r}-2\nu_{\theta}%
{\displaystyle\int_{t}^{T}}
\Gamma_{r-}\varphi_{r}d\left(  \left[  M^{\varepsilon}\right]  +\left[
M^{\delta}\right]  \right)  _{r}\medskip\\
\quad\quad\leq-\dfrac{1}{2}%
{\displaystyle\int_{t}^{T}}
\Gamma_{r-}\varphi_{r}dV_{r}^{\left(  1\right)  }-\dfrac{1}{2}%
{\displaystyle\int_{t}^{T}}
2\nu_{\theta}\Gamma_{r-}\varphi_{r}d\left(  \left[  M^{\varepsilon}\right]
+\left[  M^{\delta}\right]  \right)  _{r}\medskip\\
\quad\quad=-\dfrac{1}{2}%
{\displaystyle\int_{t}^{T}}
\varphi_{r}\Gamma_{r-}\left(  dV_{r}^{\left(  1\right)  }+2\nu_{\theta
}d\left(  \left[  M^{\varepsilon}\right]  +\left[  M^{\delta}\right]  \right)
_{r}\right)  =-\dfrac{1}{2}%
{\displaystyle\int_{t}^{T}}
\varphi_{r}\left(  \Gamma_{r-}dV_{r}\right)  \leq-\dfrac{1}{2}%
{\displaystyle\int_{t}^{T}}
\varphi_{r}d\Gamma_{r}~.
\end{array}
\]
Finally, we deduced, for any $0\leq t\leq T$, $\mathbb{P}$-a.s., the
\textit{pathwise inequality}%
\begin{equation}
\Gamma_{t}\varphi_{t}+\dfrac{\alpha_{\theta}}{8}%
{\displaystyle\int_{t}^{T}}
\Gamma_{r-}d[M^{\varepsilon,\delta}]_{r}+\dfrac{1}{2}%
{\displaystyle\int_{t}^{T}}
\varphi_{r}d\Gamma_{r}\leq\dfrac{\gamma}{2}%
{\displaystyle\int_{t}^{T}}
\Gamma_{r-}|\mathcal{R}_{r}^{\varepsilon,\delta}|^{2}dr+%
{\displaystyle\int_{t}^{T}}
\Gamma_{r-}d\mathcal{L}_{r}-%
{\displaystyle\int_{t}^{T}}
\Gamma_{r-}dm_{r}~.\label{pathwise ineq}%
\end{equation}
\smallskip\noindent\textbf{Step 5: integrability of }$\Gamma_{T}$. Our
intention is to take $\mathbb{E}$ inside (\ref{pathwise ineq}), in order to
cancel the martingale term $%
{\textstyle\int_{t}^{T}}
\Gamma_{r-}dm_{r}$ and to derive $L^{2}$ estimates. However, $\Gamma_{T}$
contains an exponential of the martingales' quadratic variation. Therefore, to
reach our goal, we have to prove that $\mathbb{E}\left(  \Gamma_{T}%
^{p}\right)  <\infty$, for a $p>1$, uniformly with respect to $\varepsilon
,\delta\in(0,\varepsilon_{1}].$

According to Proposition \ref{rot-apriori}-$\left(  d\right)  $ we have
$\Lambda_{\ast}^{2}$ as being the BMO upper estimate for the martingale:%
\[
\mathbb{E}^{\mathcal{F}_{\tau}}\left(  \left[  M^{\varepsilon}\right]
_{T}-\left[  M^{\varepsilon}\right]  _{\tau}\right)  \leq\Lambda^{2}%
\leq\Lambda_{\ast}^{2}=\Lambda^{2}+2j^{2},\quad\mathbb{P}\text{-a.s.,}%
\quad\text{for any stopping time }\tau\text{.}%
\]
The smallness condition $\left(  C_{3}\right)  $ given in (\ref{C3}) states
that $64\kappa_{\theta}\Lambda_{\ast}^{2}<\alpha_{\theta}~$, which is
equivalent to $16\left(  \frac{4\kappa_{\theta}}{\alpha_{\theta}}\right)
\Lambda_{\ast}^{2}=16\nu_{\theta}\Lambda_{\ast}^{2}<1$ and, as consequence,
there exists $p>1$, closer to $1$ such that%
\[
16p\nu_{\theta}\Lambda_{\ast}^{2}<1\quad\Longrightarrow\quad8p\nu_{\theta
}<\frac{1}{2\Lambda_{\ast}^{2}}<\frac{1}{\Lambda_{\ast}^{2}}.
\]
We use twice the Cauchy--Schwarz inequality and obtain%
\begin{equation}
\mathbb{E}\left(  \Gamma_{T}^{p}\right)  \leq\left(  \mathbb{E}\exp\left(
8p\nu_{\theta}\left[  M^{\varepsilon}\right]  _{T}\right)  \right)
^{1/4}\left(  \mathbb{E}\exp\left(  8p\nu_{\theta}[M^{\delta}]_{T}\right)
\right)  ^{1/4}\left(  \mathbb{E}\exp\left(  2pV_{T}^{\left(  1\right)
}\right)  \right)  ^{1/2}.\label{E of Gamp}%
\end{equation}
Each of the three terms from the RHS of (\ref{E of Gamp}) are bounded. Indeed,
for the adapted increasing c\`{a}dl\`{a}g process $A_{t}:=\left[
M^{\varepsilon}\right]  _{t}$~, with $A_{0}=0$, one can apply Proposition
\ref{rot-apriori}-$\left(  d\right)  $, which gives $\mathbb{E}^{\mathcal{F}%
_{\tau}}\left(  A_{T}-A_{\tau}\right)  \leq\Lambda_{\ast}^{2}$. Now, according
to Lemma \ref{garsia} and Corollary \ref{rot-expmom}-$\left(  a\right)  $, for
every constant $\lambda<1/\Lambda_{\ast}^{2}$,%
\[
\mathbb{E}\exp\left(  \lambda\left[  M^{\varepsilon}\right]  _{T}\right)
\leq\frac{1}{1-\lambda\Lambda_{\ast}^{2}}<+\infty.
\]
But, we chosed $p>1$ such that $\lambda=8p\nu_{\theta}<1/\Lambda_{\ast}^{2}~$,
and this implies that there exists a positive constant $C_{1}>0$ such that%
\[
\mathbb{E}\exp\left(  8p\nu_{\theta}\left[  M^{\varepsilon}\right]
_{T}\right)  \leq C_{1}<+\infty\quad\text{and, identically,}\quad
\mathbb{E}\exp\left(  8p\nu_{\theta}[M^{\delta}]_{T}\right)  <C_{1}<+\infty
\]
For the third factor, let us remind the definition of $V_{T}^{\left(
1\right)  }$ given in (\ref{r19}):%
\[
V_{T}^{\left(  1\right)  }=\dfrac{2}{\alpha_{\theta}}{%
{\displaystyle\int_{0}^{T}}
}\left(  c_{\Phi}L_{r}+\dfrac{c_{\Phi}^{2}\ell_{r}^{2}}{2\gamma}+c_{\Phi
}\left\vert \left\vert \theta^{\prime}\right\vert \right\vert _{\infty
}+c_{\Phi}|F_{r}^{\delta}|\right)  dr,
\]
By standard estimates, we use Assumption $\left(  H_{2}^{\prime}\right)  $,
Proposition \ref{rot-apriori}-$\left(  b\right)  $ and the fact that ${%
{\textstyle\int_{0}^{T}}
}\left\vert \mathcal{R}_{r}\left(  M^{\delta}\right)  \right\vert ^{2}dr$ is
controlled by $\left[  M^{\delta}\right]  _{T}$, via Assumption $\left(
H_{3}^{\prime}\right)  $. Finally, according to Corollary \ref{rot-expmom}%
-$\left(  c\right)  $, $V_{T}^{\left(  1\right)  }$ admits exponential moments
of any order and, in particular, $\mathbb{E}\exp(2pV_{T}^{\left(  1\right)
})\leq C_{2}<+\infty$, the constant being independent of the penalization index.

As consequence, returning to (\ref{E of Gamp}), we obtain that $\mathbb{E}%
\left(  \,\Gamma_{T}^{p}\right)  <C<+\infty$. Define also the adjoint exponent
of $p$, $p_{0}:=p/\left(  p-1\right)  $, which will be used at Step 7.

This steps guaranties $\mathbb{E}\int_{t}^{T}\Gamma_{r-}dm_{r}=0$.

\smallskip\noindent\textbf{Step 6: applying expectation, and absorption of
}$\mathcal{R}^{\varepsilon,\delta}$. Apply now the expectation to the pathwise
inequality (\ref{pathwise ineq}). We just proved in the previous step that
$\mathbb{E}\int_{t}^{T}\Gamma_{r-}dm_{r}=0$. Therefore,%
\begin{equation}
\mathbb{E}\left(  \Gamma_{t}\varphi_{t}\right)  +\dfrac{\alpha_{\theta}}%
{8}\mathbb{E}%
{\displaystyle\int_{t}^{T}}
\Gamma_{r-}d\left[  M^{\varepsilon,\delta}\right]  _{r}+\dfrac{1}{2}\mathbb{E}%
{\displaystyle\int_{t}^{T}}
\varphi_{r}d\Gamma_{r}\leq\dfrac{\gamma}{2}\mathbb{E}%
{\displaystyle\int_{t}^{T}}
\Gamma_{r-}|\mathcal{R}_{r}^{\varepsilon,\delta}|^{2}dr+\mathbb{E}%
{\displaystyle\int_{t}^{T}}
\Gamma_{r-}d\mathcal{L}_{r}.\label{Epathwise ineq}%
\end{equation}
We invoke Assumption \ref{H3'}, with the bounded predictable nonnegative
weight $g_{r}^{\left(  n\right)  }=\left(  \Gamma_{r-}\wedge n\right)
\mathbf{1}_{(t,T]}\left(  r\right)  $. Letting $n\rightarrow\infty$, by the
monotone convergence theorem, we have, since $g_{r}^{\left(  n\right)
}\rightarrow g_{r}:=\Gamma_{r-}\mathbf{1}_{(t,T]}\left(  r\right)  =e^{V_{r-}%
}\mathbf{1}_{(t,T]}\left(  r\right)  $,%
\[
\mathbb{E}{%
{\displaystyle\int_{0}^{T}}
}g_{r}\left\vert \mathcal{R}_{r}\left(  M^{\varepsilon}\right)  -\mathcal{R}%
_{r}(M^{\delta})\right\vert ^{2}dr\leq C_{\mathcal{R}}^{2}\mathbb{E}{%
{\displaystyle\int_{t}^{T}}
}\Gamma_{r-}d[M^{\varepsilon}-M^{\delta}]_{r}~.
\]
It follows%
\[
\mathbb{E}{%
{\displaystyle\int_{t}^{T}}
}\Gamma_{r-}|\mathcal{R}_{r}^{\varepsilon,\delta}|^{2}dr\leq C_{\mathcal{R}%
}^{2}\mathbb{E}{%
{\displaystyle\int_{t}^{T}}
}\Gamma_{r-}d[M^{\varepsilon,\delta}]_{r}~.
\]
Remind, from (\ref{r19}), the value of $\gamma:=\tfrac{\alpha_{\theta}%
}{8C_{\mathcal{R}}^{2}}$, which implies $\frac{\gamma C_{\mathcal{R}}^{2}}%
{2}=\tfrac{\alpha_{\theta}}{16}.$ Relation (\ref{Epathwise ineq}) becomes,
after the absorption in the LHS of the term containing $\mathbb{E}%
{\textstyle\int_{t}^{T}}
\Gamma_{r-}|\mathcal{R}_{r}^{\varepsilon,\delta}|^{2}dr$,%
\begin{equation}
\mathbb{E}\left(  \Gamma_{t}\varphi_{t}\right)  +\dfrac{\alpha_{\theta}}%
{16}\mathbb{E}%
{\displaystyle\int_{t}^{T}}
\Gamma_{r-}d[M^{\varepsilon,\delta}]_{r}+\dfrac{1}{2}\mathbb{E}%
{\displaystyle\int_{t}^{T}}
\varphi_{r}d\Gamma_{r}\leq\mathbb{E}%
{\displaystyle\int_{t}^{T}}
\Gamma_{r-}d\mathcal{L}_{r}.\label{r27}%
\end{equation}

\smallskip\noindent\textbf{Step 7: the order of boundedness and convergence.}
Remind, from (\ref{r22}), the structure of
\[
d\mathcal{L}_{r}=2c_{\Phi}\left(  \varepsilon+\delta\right)  \left(
|U_{r}^{\varepsilon}|^{2}+|U_{r}^{\delta}|^{2}\right)  dr+c_{\Phi}\left(
\varepsilon\left\vert U_{r}^{\varepsilon}\right\vert +\delta|U_{r}^{\delta
}|\right)  \left(  \left\vert U_{r}^{\varepsilon}\right\vert +|U_{r}^{\delta
}|\right)  dr.
\]
By integration, since $\Gamma_{r-}\leq\Gamma_{T}=\exp\left(  V_{T}\right)  $,
for every $r\in(0,T]$, followed by the H\"{o}lder's inequality,%
\[%
\begin{array}
[c]{ccl}%
\mathbb{E}%
{\displaystyle\int_{t}^{T}}
\Gamma_{r-}d\mathcal{L}_{r} & \leq & \mathbb{E}%
{\displaystyle\int_{0}^{T}}
\Gamma_{r-}d\mathcal{L}_{r}\leq3c_{\Phi}\left(  \varepsilon+\delta\right)
\mathbb{E}\left(  \Gamma_{T}%
{\displaystyle\int_{0}^{T}}
\left(  \left\vert U_{r}^{\varepsilon}\right\vert ^{2}+\left\vert
U_{r}^{\delta}\right\vert ^{2}\right)  dr\right)  \medskip\\
& \leq & 3c_{\Phi}\left(  \varepsilon+\delta\right)  \left(  \mathbb{E}\left(
\Gamma_{T}^{p}\right)  \right)  ^{1/p}\left(  \mathbb{E}\left(
{\displaystyle\int_{0}^{T}}
\left(  \left\vert U_{r}^{\varepsilon}\right\vert ^{2}+\left\vert
U_{r}^{\delta}\right\vert ^{2}\right)  dr\right)  ^{p_{0}}\right)  ^{1/p_{0}%
}\medskip\\
& \leq & 3c_{\Phi}\left(  \varepsilon+\delta\right)  \left\vert \left\vert
\Gamma_{T}\right\vert \right\vert _{L^{p}}\cdot2\left(  \mathbb{E}\left(
{\displaystyle\int_{0}^{T}}
\left\vert U_{r}^{\varepsilon}\right\vert ^{2}dr\right)  ^{p_{0}}%
+\mathbb{E}\left(
{\displaystyle\int_{0}^{T}}
\left\vert U_{r}^{\delta}\right\vert ^{2}dr\right)  ^{p_{0}}\right)
^{1/p_{0}}\medskip\\
& \leq & 6c_{\Phi}\left(  \varepsilon+\delta\right)  \left\vert \left\vert
\Gamma_{T}\right\vert \right\vert _{L^{p}}C\left(  p_{0}\right)  \leq C\left(
\varepsilon+\delta\right)  ,\quad\quad\text{unif. w.r.t. }\varepsilon
,\delta\in(0,\varepsilon_{1}]\text{.}%
\end{array}
\]
The inequality from the last line is based on Corollary (\ref{rot-expmom}%
)-$\left(  b\right)  $, applied for the continuous increasing process $A_{t}:=%
{\textstyle\int_{0}^{t}}
\left\vert U_{r}^{\varepsilon}\right\vert ^{2}dr$, $U_{r}^{\varepsilon}$
having moments uniformly with respect to $\varepsilon$.$\smallskip$

Since:\vspace{-0.1in}

\begin{itemize}
\item $\Gamma_{t}\geq1$ and $\varphi_{t}\geq\alpha_{\theta}|Y_{t}%
^{\varepsilon,\delta}|^{2}=\alpha_{\theta}\left\vert Y_{t}^{\varepsilon}%
-Y_{t}^{\delta}\right\vert ^{2}$ (the coercivity condition); therefore,
$\alpha_{\theta}\left\vert Y_{t}^{\varepsilon}-Y_{t}^{\delta}\right\vert
^{2}\leq1\cdot\varphi_{t}\leq\Gamma_{t}\varphi_{t}~$;\vspace{-0.1in}

\item $\Gamma_{r-}\geq1$, so $\mathbb{E}%
{\displaystyle\int_{t}^{T}}
\Gamma_{r-}d\left[  M^{\varepsilon,\delta}\right]  _{r}\geq%
{\displaystyle\int_{t}^{T}}
d\left[  M^{\varepsilon,\delta}\right]  _{r}=\mathbb{E}\left(  \left[
M^{\varepsilon}-M^{\delta}\right]  _{T}-\left[  M^{\varepsilon}-M^{\delta
}\right]  _{t}\right)  ~;$\vspace{-0.1in}

\item $\mathbb{E}%
{\displaystyle\int_{0}^{T}}
\Gamma_{r-}d\left[  M^{\varepsilon,\delta}\right]  _{r}\geq\mathbb{E}\left[
M^{\varepsilon}-M^{\delta}\right]  _{T}~$, we have, for every $t\in\left[
0,T\right]  $,\vspace{-0.1in}%
\[
\alpha_{\theta}\mathbb{E}|Y_{t}^{\varepsilon}-Y_{t}^{\delta}|^{2}%
+\dfrac{\alpha_{\theta}}{16}\mathbb{E}[M^{\varepsilon}-M^{\delta}]_{T}%
+\frac{1}{2}\mathbb{E}%
{\displaystyle\int_{0}^{T}}
\varphi_{r}d\Gamma_{r}\leq C\left(  \varepsilon+\delta\right)  .
\]
Take $\sup_{t\in\left[  0,T\right]  }$, in order to obtain%
\begin{equation}
\sup_{t\in\left[  0,T\right]  }\mathbb{E}|Y_{t}^{\varepsilon}-Y_{t}^{\delta
}|^{2}+\mathbb{E}[M^{\varepsilon}-M^{\delta}]_{T}+\mathbb{E}%
{\displaystyle\int_{0}^{T}}
\varphi_{r}d\Gamma_{r}\leq\dfrac{\bar{C}}{\alpha_{\theta}}\left(
\varepsilon+\delta\right)  .\label{r29}%
\end{equation}
The inequality (\ref{r27}) also gives $\sup_{t\in\left[  0,T\right]  }\left(
\mathbb{E}\varphi_{t}\right)  \leq\bar{C}\left(  \varepsilon+\delta\right)  $.
\end{itemize}

\smallskip\noindent\textbf{Step 8: advancing from }$\sup_{t}\mathbb{E}%
$\textbf{ to }$\mathbb{E}\sup_{t}$.$\quad$Start again from the pathwise
inequality, where we isolate in the LHS only the first term. Denote the
martingale $N_{t}:=%
{\textstyle\int_{0}^{t}}
\Gamma_{r-}dm_{r}$ and extend the first two integrals on the entire time
interval. Pathwise, we have:%
\[
\Gamma_{t}\varphi_{t}\leq\dfrac{\gamma}{2}%
{\displaystyle\int_{0}^{T}}
\Gamma_{r-}|\mathcal{R}_{r}^{\varepsilon,\delta}|^{2}dr+%
{\displaystyle\int_{0}^{T}}
\Gamma_{r-}d\mathcal{L}_{r}+\left\vert N_{T}\right\vert +\left\vert
N_{t}\right\vert
\]
Define now the positive random variable $Z:=\sup_{t\in\left[  0,T\right]
}\left(  \Gamma_{t}\varphi_{t}\right)  $ and take the supremum over
$t\in\left[  0,T\right]  $ in the above inequality, followed by the
expectation. We obtain%
\[%
\begin{array}
[c]{ccl}%
\mathbb{E}Z & \leq & \dfrac{\gamma}{2}\mathbb{E}%
{\displaystyle\int_{0}^{T}}
\Gamma_{r-}|\mathcal{R}_{r}^{\varepsilon,\delta}|^{2}dr+\mathbb{E}%
{\displaystyle\int_{0}^{T}}
\Gamma_{r-}d\mathcal{L}_{r}+2\mathbb{E}\sup\limits_{t\in\left[  0,T\right]
}\left\vert N_{t}\right\vert \medskip\\
& \leq & C\left(  \varepsilon+\delta\right)  +2\mathbb{E}\sup\limits_{t\in
\left[  0,T\right]  }\left\vert N_{t}\right\vert \overset{BDG}{\leq}C\left(
\varepsilon+\delta\right)  +2C_{BDG}\mathbb{E}\left(  \left[  N\right]
_{T}^{1/2}\right)
\end{array}
\]
We evaluate now the quadratic variation $\left[  N\right]  _{T}.$ Starting
from the formula of $dm=\left\langle \nabla_{x}\Phi_{r},dM_{r}^{\varepsilon
}\right\rangle +\left\langle \nabla_{y}\Phi_{r},dM_{r}^{\delta}\right\rangle
$, one can obtain $d\left[  m\right]  _{r}\leq C\left(  c_{\Phi}\right)
\left\vert \Delta Y_{r-}\right\vert ^{2}d\left(  \left[  M^{\varepsilon
}\right]  _{r}+\left[  M^{\delta}\right]  _{r}\right)  $. Therefore,%
\[%
\begin{array}
[c]{ccl}%
\left[  N\right]  _{T} & = &
{\displaystyle\int_{0}^{T}}
\Gamma_{r-}^{2}d\left[  m\right]  _{r}\leq C\left(  c_{\Phi}\right)
{\displaystyle\int_{0}^{T}}
\Gamma_{r-}\left(  \Gamma_{r-}|Y_{r-}^{\varepsilon,\delta}|^{2}\right)
d\left(  \left[  M^{\varepsilon}\right]  _{r}+\left[  M^{\delta}\right]
_{r}\right)  \medskip\\
& \leq & C\left(  c_{\Phi}\right)
{\displaystyle\int_{0}^{T}}
\Gamma_{r-}\dfrac{1}{\alpha_{\theta}}\sup\nolimits_{t\in\left[  0,T\right]
}\left(  \Gamma_{t}\varphi_{t-}\right)  d\left(  \left[  M^{\varepsilon
}\right]  _{r}+\left[  M^{\delta}\right]  _{r}\right)  \medskip\\
& = & C\left(  c_{\Phi}\right)  \dfrac{1}{\alpha_{\theta}}%
{\displaystyle\int_{0}^{T}}
\Gamma_{r-}Zd\left(  \left[  M^{\varepsilon}\right]  _{r}+\left[  M^{\delta
}\right]  _{r}\right)  .
\end{array}
\]
Define not the control random variables:%
\[
R_{1}:=%
{\displaystyle\int_{0}^{T}}
\Gamma_{r-}d\left(  \left[  M^{\varepsilon}\right]  _{r}+[M^{\delta}%
]_{r}\right)  \quad\quad\text{and}\quad\quad R_{2}:=%
{\displaystyle\int_{0}^{T}}
\varphi_{r}d\Gamma_{r}~.
\]
The already obtained estimates (\ref{r27}), (\ref{r29}) give $\mathbb{E}%
R_{1}\leq C$ and $\mathbb{E}R_{2}\leq C\left(  \varepsilon+\delta\right)  $.
Consequently, $\left[  N\right]  _{T}^{1/2}\leq c\sqrt{Z}\sqrt{R_{1}+R_{2}}%
$.and, by Young inequality%
\[%
\begin{array}
[c]{ccl}%
2\mathbb{E}\sup\limits_{t\in\left[  0,T\right]  }\left\vert N_{t}\right\vert
& \leq & 2cC_{BDG}\mathbb{E}\left(  \left[  N\right]  _{T}^{1/2}\right)
\leq2cC_{BDG}\mathbb{E}\left(  \sqrt{Z}\sqrt{R_{1}+R_{2}}\right)  \medskip\\
& \leq & \dfrac{1}{2}\mathbb{E}Z+C^{\prime}\mathbb{E}\left(  R_{1}%
+R_{2}\right)  \leq\dfrac{1}{2}\mathbb{E}Z+C^{\prime}\left(  \varepsilon
+\delta\right)  .
\end{array}
\]
So, $\mathbb{E}Z\leq C\left(  \varepsilon+\delta\right)  +\dfrac{1}%
{2}\mathbb{E}Z+C^{\prime}\left(  \varepsilon+\delta\right)  $, which gives,
since the approximating sequences are in $\mathbb{D}_{2}^{2}$,%
\[
\mathbb{E}Z=\mathbb{E}\left(  \sup_{t\in\left[  0,T\right]  }\left(
\Gamma_{t}\varphi_{t}\right)  \right)  \leq C\left(  \varepsilon
+\delta\right)  .
\]
We are near to the closer of the proof of Theorem \ref{rot-thm-conv}. Since
$\Gamma_{t}\geq1$, for all $t\in\left[  0,T\right]  $ and $\varphi_{t}%
\geq\alpha_{\theta}\left\vert Y_{t}^{\varepsilon,\delta}\right\vert
^{2}=\alpha_{\theta}\left\vert Y_{t}^{\varepsilon}-Y_{t}^{\delta}\right\vert
^{2}$, we pass to the supremum and, after this, we apply the expectation:%
\[
\mathbb{E}\left(  \sup_{t\in\left[  0,T\right]  }|Y_{t}^{\varepsilon}%
-Y_{t}^{\delta}|^{2}\right)  \leq\frac{1}{\alpha_{\theta}}\mathbb{E}\left(
\sup_{t\in\left[  0,T\right]  }\varphi_{t}\right)  \leq\frac{1}{\alpha
_{\theta}}\mathbb{E}\left(  \sup_{t\in\left[  0,T\right]  }\left(  \Gamma
_{t}\varphi_{t}\right)  \right)  \leq C\left(  \varepsilon+\delta\right)  .
\]
Considering also the estimate (\ref{r29}), we conclude that there exists a
positive constant $\mathcal{C}$, such that%
\[
\mathbb{E}\left(  \sup_{t\in\left[  0,T\right]  }|Y_{t}^{\varepsilon}%
-Y_{t}^{\delta}|^{2}\right)  +\mathbb{E}[M^{\varepsilon}-M^{\delta}]_{T}%
\leq\mathcal{C}\left(  \varepsilon+\delta\right)  .
\]
Consequently $\left(  Y^{\varepsilon}\right)  _{\varepsilon}$ is a Cauchy
sequence in the Banach ${\mathcal{{\mathbb{D}}}}_{2}^{2}$, $\left(
M^{\varepsilon}\right)  _{\varepsilon}$ is Cauchy in the Hilbert
$\mathcal{M}_{2}^{2}$, and $\left(  U^{\varepsilon}\right)  _{\varepsilon}$ is
bounded in $\Lambda_{2}^{2}$. The proof of Theorem \ref{rot-thm-conv} is now
complete.\hfill
\end{proof}

\subsection{Existence and Uniqueness of the solution\label{rot-uniq}}

\begin{theorem}
[Uniqueness]\label{rot-thm-uniq}Under the assumptions of Theorem
\ref{rot-thm-conv}, problem (\ref{r1}) has at most one solution $\left(
Y,M,U\right)  \in{\mathcal{{\mathbb{D}}}}_{2}^{2}\times\mathcal{M}_{2}%
^{2}\times\Lambda_{2}^{2},$ with $M_{0}=0$.
\end{theorem}

\begin{proof}
This proof follows classical steps. Let $\left(  Y^{1},M^{1},U^{1}\right)  $
and $\left(  Y^{2},M^{2},U^{2}\right)  $ be two solutions and keep the
notation of the proof of Theorem \ref{rot-thm-conv}, with the indexes
$\varepsilon,\delta$ replaced by $1,2$. Every step of that proof applies
verbatim, with a single simplification. More precisely, the two reflection
terms now vanish or have the right sign. Indeed, for $i=1,2$, $Y_{r}^{i}\in K$
and $U_{r}^{i}\in\partial I_{K}\left(  Y_{r}^{i}\right)  =N_{K}\left(
Y_{r}^{i}\right)  $. Now, if $U_{r}^{1}\neq0$, then $Y_{r}^{1}\in bd\left(
K\right)  $ and $U_{r}^{1}=\left\vert U_{r}^{1}\right\vert \mathbf{n}\left(
Y_{r}^{1}\right)  $, so that Assumption $\left(  C_{2}\right)  $, applied at
the pair $\left(  Y_{r}^{1},Y_{r}^{2}\right)  \in bd\left(  K\right)  \times
K$, gives $\left\langle \nabla_{x}\Phi_{r},\Theta_{r}U_{r}^{1}\right\rangle
\geq0$. A symmetrical relation can be obtain for $U^{2}$, also. No projection
and no Lipschitz correction is needed, and, easily, $\mathcal{L}_{1,2}\equiv
0$. Moreover, Proposition \ref{rot-apriori} applies to the triplet $\left(
Y^{i},M^{i},U^{i}\right)  $, for $i=1,2$, by its last assertion, so that the
weight $\Gamma$ given by (\ref{r19}) and which is constructed now on $\left[
M^{1}\right]  +\left[  M^{2}\right]  $, satisfies $\mathbb{E}\Gamma_{T}%
^{p}<\infty$, under the Assumption $\left(  C_{3}\right)  $. The key
inequality (\ref{r27}) becomes here:%
\[
\mathbb{E}\left(  \Gamma_{t}\varphi_{t}\right)  +\dfrac{\alpha_{\theta}}%
{16}\mathbb{E}{%
{\displaystyle\int_{t}^{T}}
}\Gamma_{r-}d\left[  M^{1,2}\right]  _{r}\leq0,
\]
whence $\mathbb{E}\varphi_{t}=0$, for all $t\in\left[  0,T\right]  $, and no
Gronwall argument being needed, at all. According to (\ref{r16}), it follows
$Y^{1}=Y^{2}$, $d\mathbb{P}\otimes dr$-a.e, and, after this, we obtain
$\mathbb{E}\left[  M^{1,2}\right]  _{T}=0$. Consequently, $M^{1}=M^{2}$,
$d\mathbb{P}\otimes dr$-a.e. and, finally, if we return to the original
equation, (\ref{r1}), it remains%
\[
{%
{\displaystyle\int_{t}^{T}}
}\Theta_{r}\left(  U_{r}^{1}-U_{r}^{2}\right)  dr=0,\quad\quad\forall
\,t\in\left[  0,T\right]  .
\]
Since $\Theta_{r}$ is an invertible rotation matrix, it follows $U^{1}=U^{2}$,
$d\mathbb{P}\otimes dr$-a.e. The uniqueness of the triplet $\left(
Y,M,U\right)  \in{\mathcal{{\mathbb{D}}}}_{2}^{2}\times\mathcal{M}_{2}%
^{2}\times\Lambda_{2}^{2}~$, as being the solution is proved.\hfill
\end{proof}

\begin{theorem}
[Existence]\label{rot-thm-exist}Under the assumptions of Theorem
\ref{rot-thm-conv}, problem (\ref{r1}) has a unique solution $\left(
Y,M,U\right)  \in{\mathcal{{\mathbb{D}}}}_{2}^{2}\times\mathcal{M}_{2}%
^{2}\times\Lambda_{2}^{2}$, and it satisfies,%
\[
\mathbb{E}\sup\limits_{t\in\left[  0,T\right]  }\left\vert Y_{t}^{\varepsilon
}-Y_{t}\right\vert ^{2}+\mathbb{E}\left[  M^{\varepsilon}-M\right]  _{T}\leq
C\varepsilon,
\]
where $C$ is independent of $\varepsilon\in(0,\varepsilon_{1}]$ and $\left(
Y_{t}^{\varepsilon},M_{t}^{\varepsilon}\right)  _{t}$ is the sequence of
penalized solution, provided by the Moreau-Yosida technique.
\end{theorem}

\begin{proof}
According to Theorem \ref{rot-thm-conv}, there exist $Y\in
{\mathcal{{\mathbb{D}}}}_{2}^{2}$ and $M\in\mathcal{M}_{2}^{2}$ with
$Y^{\varepsilon}\rightarrow Y$ in ${\mathcal{{\mathbb{D}}}}_{2}^{2}$ and
$M^{\varepsilon}\rightarrow M$ in $\mathcal{M}_{2}^{2}$, the convergences
being strong in the corresponding spaces. If we let $\delta\searrow0$ in
(\ref{r25}) we also obtain the convergence rate. From (\ref{r5})-$\left(
b\right)  $ we conclude that the family $\left(  U^{\varepsilon}\right)
_{\varepsilon}$ is bounded in the Hilbert space $\Lambda_{2}^{2}$, so along a
subsequence, denoted with $\varepsilon_{n}~$, $U^{\varepsilon_{n}%
}\rightharpoonup U$, weakly in $\Lambda_{2}^{2}$.

We consider the first equation from (\ref{r4}) and we pass to the limit, as
$\varepsilon_{n}\rightarrow0$. This can be done because the map
\[
\left(  Y,M,U\right)  \longmapsto Y_{t}+\int_{t}^{T}\Theta_{r}U_{r}dr-\int
_{t}^{T}F\left(  r,Y_{r},\mathcal{R}_{r}(M)\right)  dr+\left(  M_{T}%
-M_{t}\right)
\]
is affine and continuous for the weak topology from $\Lambda_{2}^{2}$ and the
strong ones from ${\mathcal{{\mathbb{D}}}}_{2}^{2}$ and $\mathcal{M}_{2}^{2}$
(please consult Assumption \ref{H3}, which makes $\mathcal{R}$ Lipschitz as an
application from $\mathcal{M}_{2}^{2}$ into $\Lambda_{2\times k}^{2}$). As
consequence, the limit triplet $\left(  Y,M,U\right)  $ obeys the equation we study.

The only item remaining is to identify $U$ as being an element from the
subdifferential of $Y$. By (\ref{r5})-$\left(  c\right)  $, we have%
\[
\mathbb{E}%
{\displaystyle\int_{0}^{T}}
d_{K}^{2}\left(  Y_{r}^{\varepsilon_{n}}\right)  dr\leq C\varepsilon_{n}%
^{2}~,
\]
and the distance operator $d_{K}$ being $1$-Lipschitz, it implies%
\[
\mathbb{E}\int_{0}^{T}d_{K}^{2}\left(  Y_{r}\right)  dr=\lim
\limits_{\varepsilon_{n}\searrow0}\mathbb{E}\int_{0}^{T}d_{K}^{2}\left(
Y_{r}^{\varepsilon_{n}}\right)  dr=0,
\]
so that $Y_{r}\in K$, $d\mathbb{P}\otimes dr$-a.e. Moreover, for every $v\in
K$,
\[
\left\langle v-Y_{r}^{\varepsilon_{n}},U_{r}^{\varepsilon_{n}}\right\rangle
=\left\langle v-\bar{Y}_{r}^{\varepsilon_{n}},U_{r}^{\varepsilon_{n}%
}\right\rangle -\varepsilon_{n}\left\vert U_{r}^{\varepsilon_{n}}\right\vert
^{2}\leq0,
\]
because $U_{r}^{\varepsilon_{n}}\in N_{K}\left(  \bar{Y}_{r}^{\varepsilon_{n}%
}\right)  $ and $\bar{Y}_{r}^{\varepsilon_{n}}=\pi_{K}\left(  Y_{r}%
^{\varepsilon_{n}}\right)  $. Passing to the limit, the strong convergence of
$Y^{\varepsilon_{n}}$, together with the weak convergence of $U^{\varepsilon
_{n}}$ will produce
\[
\mathbb{E}\int_{s}^{t}\left\langle v-Y_{r},U_{r}\right\rangle dr\leq
0,\quad\text{for all }0\leq s\leq t\leq T\text{,}\quad\text{and all }v\in K,
\]
that is $U_{r}\in\partial I_{K}\left(  Y_{r}\right)  $, $d\mathbb{P}\otimes
dr$-a.e. The uniqueness given by Theorem \ref{rot-thm-uniq} assures that the
entire family, and not merely a subsequence, converges to the obtained
limit.\hfill
\end{proof}

\section{Rotation-angle sensitivity and associated control problems}

In the spirit of the previous sections, let $K\subset\mathbb{R}^{2}$ be a
bounded, uniformly convex domain with $C^{3}$ boundary, described by a convex
defining function $\psi:\mathbb{R}^{2}\rightarrow\mathbb{R}$ with
$\operatorname{int}K=\{\psi<0\}$, $\partial K=\{\psi=0\}$, $|\nabla\psi|=1$ on
$\partial K$, and constants
\[
a_{\psi}>0,\quad\quad b_{\psi}\geq1,\quad\quad D_{K}:=\operatorname{diam}%
K,\quad\quad b_{\ast}:=b_{\psi}(1+b_{\psi}D_{K}),
\]
such that, on a neighborhood of $K$, $a_{\psi}I\preceq D^{2}\psi\preceq
b_{\psi}I$ and $\Vert D^{3}\psi\Vert\leq b_{\psi}$. The same constant
$b_{\psi}$ bounds both the Hessian and the third derivative. The outward unit
normal at $x\in\partial K$ is $\mathbf{n}(x)=\nabla\psi(x)$. Fix $y_{0}%
\in\operatorname{int}K$ and $\rho>0$ with $B(y_{0},\rho)\subset K$, and set
$R_{K}:=\max_{z\in K}|z-y_{0}|$ ($\rho\leq R_{K}$) and $c_{K}:=|y_{0}|+R_{K}$,
which is a constant of $K$, independent of $\theta$.

The reflected equation is, for $\theta\in C^{1}([0,T];(0,\pi/2))$, with
$a_{\theta}:=\min_{r}\cos\theta_{r}>0$,%
\begin{equation}
\mathcal{OP}\left(  \Theta^{\theta}\right)  :\left\{
\begin{array}
[c]{l}%
Y_{t}+%
{\displaystyle\int_{t}^{T}}
\Theta_{r}^{\theta}U_{r}dr=\eta+%
{\displaystyle\int_{t}^{T}}
F\left(  r,Y_{r},\mathcal{R}_{r}(M)\right)  dr-(M_{T}-M_{t}),\quad
\mathbb{P}\text{-a.s.,}\medskip\\
U_{r}\in\partial I_{K}(Y_{r}),\text{ \ }d\mathbb{P\otimes}dr\text{-a.e. on
}\Omega\times\left[  0,T\right]  ,
\end{array}
\right.  \label{original pb}%
\end{equation}
where $\Theta_{r}^{\theta}=\cos\theta_{r}I+\sin\theta_{r}J$ and $a_{\theta
}|u|^{2}\leq\langle\Theta_{r}u,u\rangle\leq|u|^{2}$ for every $u$. Recall the
compatibility conditions:\vspace{-0.06in}

\begin{itemize}
\item[$(C_{1})$:] $\alpha_{\theta}:=1-\max_{r}\tan\theta_{r}>0$;\vspace
{-0.06in}

\item[$(C_{2})$:] $\langle\nabla_{x}\Phi_{r}(x,y),\Theta_{r}\mathbf{n}%
(x)\rangle\geq0$, for all $r,x\in\partial K,\ y\in K$;\vspace{-0.06in}

\item[$(C_{0})$:] $\tan\theta_{\max}<\rho/R_{K}$;\vspace{-0.06in}

\item[$(C_{3})$:] $64\kappa_{\theta}\Lambda_{\ast}^{2}<\alpha_{\theta}$, where
$\Lambda_{\ast}^{2}:=\Lambda^{2}+2j^{2}$, $j:=D_{K}+2$, and%
\begin{equation}
\left\{
\begin{array}
[c]{l}%
\kappa_{\theta}:=2b_{\psi}^{2}\tan\theta_{\max}\left(  1+\dfrac{8}%
{\alpha_{\theta}}\tan\theta_{\max}\right)  ,\medskip\\
\Lambda^{2}:=R_{K}^{2}+2R_{K}\sqrt{T\bar{\Psi}}+\frac{\bar{\Psi}}{2a_{\theta}%
}\quad\quad\text{and}\quad\quad\bar{\Psi}:=(c_{7}+\sqrt{c_{7}^{2}+2c_{6}}%
)^{2},
\end{array}
\right.  \label{some constants}%
\end{equation}
with $c_{6},c_{7}$ some explicit constants depending only on $f_{\infty},\Vert
L\Vert_{\infty},\Vert\ell\Vert_{\infty},C_{\mathcal{R}},R_{K},c_{K},T,$ and
not on $\theta$.
\end{itemize}

\begin{remark}
[Why the reflection admits no meaningful amplitude control]Section
\ref{rot-uniq} just proved, in Theorem \ref{rot-thm-exist}, that there exists
a unique solution $\left(  Y^{\theta},M^{\theta},U^{\theta}\right)
\in{\mathcal{{\mathbb{D}}}}_{2}^{2}\times\mathcal{M}_{2}^{2}\times\Lambda
_{2}^{2}$ for the obstacle problem $\mathcal{OP}\left(  \Theta^{\theta
}\right)  $. Since $\partial I_{K}=N_{K}\left(  y\right)  $, then for any
predictable bounded $\rho$, $0<\rho_{\min}\leq\rho_{r}\leq\rho_{\max}$,%
\[
U_{r}\in\partial I_{K}\left(  Y_{r}\right)  =N_{K}\left(  y\right)
\quad\Longleftrightarrow\quad\rho_{r}U_{r}\in\partial I_{K}\left(
Y_{r}\right)  ,\quad d\mathbb{P}\otimes dr\text{-a.e.}%
\]
Consequently, if we consider $\mathcal{OP}\left(  \rho\Theta^{\theta}\right)
$, one can write $\rho_{r}\Theta_{r}^{\theta}U_{r}dr=\Theta_{r}^{\theta
}\left(  \rho_{r}U_{r}\right)  dr=\Theta_{r}^{\theta}\tilde{U}_{r}dr$, where
we denoted $\tilde{U}_{r}^{\theta,\rho}:=\rho_{r}U_{r}^{\theta}\in\partial
I_{K}\left(  Y_{r}\right)  $. Therefore, the existence and uniqueness results
do not change if we consider the perturbing rotation matrix replaced by
$\rho\Theta^{\theta}$, with the amplitude function $\rho$. Hence, we do not
have a consistent change of the problems' dynamic, the scaling being absorbed
by the exterior normal cone. More precisely, $(Y^{\theta,\rho},M^{\theta,\rho
})=(Y^{\theta},M^{\theta})$, for every $\rho$. Only the angle $\theta_{r}$,
which changes the direction $\Theta_{r}\mathbf{n}(x)$, not merely its length,
is a meaningful control. Only $\theta$ is retained below.
\end{remark}

The considered control is $\theta=(\theta_{t})_{t\in\lbrack0,T]}$,
$\mathcal{F}_{t}$-adapted and predictable, valued in $(0,\theta_{\max}]$, for
a fixed constant $\theta_{\max}\in(0,\pi/4)$ chosen and fixed.

\begin{definition}
[Admissible controls]\label{def:Uad}We say that $\theta_{\max}\in(0,\pi/4)$ is
compatible with $K$ and the data if it satisfies $(C_{0})$ to $(C_{3}),$ with
$\max_{r}\theta_{r}$ replaced by $\theta_{\max}$
(Section~\ref{sec:geometric-design} computes explicitly the largest
$\theta_{\max}$). The non-empty admissible class of controls is $\mathcal{U}%
_{\mathrm{ad}}:=\{\theta:[0,T]\times\Omega\rightarrow(0,\theta_{\max}],$
$\ \theta$ is predictable$\}.$
\end{definition}

\begin{proposition}
[Well-posedness, and uniform a priori bounds]\label{prop:wellposed} Let
$\theta_{\max}$ be compatible with $K$ and the data, and let all the previous
assumptions hold. Then:

\begin{enumerate}
\item[$(a)$] For every $\theta\in\mathcal{U}_{\mathrm{ad}}$, $\mathcal{OP}%
\left(  \Theta^{\theta}\right)  $ admits a unique solution $\left(  Y^{\theta
},M^{\theta},U^{\theta}\right)  \in\mathbb{D}_{2}^{2}\times\mathcal{M}_{2}%
^{2}\times\Lambda_{2}^{2}$.

\item[$(b)$] $\sup_{t\in\lbrack0,T]}|Y_{t}^{\theta}|\leq c_{K}+1$,
$\mathbb{P}$-a.s., for every $\theta\in\mathcal{U}_{\mathrm{ad}}$.

\item[$(c)$] $\mathbb{E}%
{\displaystyle\int_{0}^{T}}
|U_{r}^{\theta}|^{2}\,dr\leq\bar{\Psi}/a_{\star}^{2}~$, for every $\theta
\in\mathcal{U}_{\mathrm{ad}}$, where $a_{\star}:=\cos\theta_{\max}$ and
$\bar{\Psi}$ is given by (\ref{some constants}) and it is independent of
$\theta$, except through the single constant $a_{\star}$.
\end{enumerate}
\end{proposition}

\subsection{Geometric design: the largest compatible constant
angle\label{sec:geometric-design}}

\begin{proposition}
[Purely geometric threshold]\label{prop:geom-threshold} Define
$g:\mathbb{R\rightarrow R}$, $g(\theta):=a_{\psi}\cos(2\theta)-16b_{\ast}%
\sin\theta,$ on $[0,\pi/4]$. Then $g$ is continuous, strictly decreasing,
$g(0)=a_{\psi}>0$, $g(\pi/4)=-16b_{\ast}/\sqrt{2}<0$, and there exists a
unique angle $\theta_{c}\in(0,\pi/4)$ with $g(\theta_{c})=0$ and $g\geq0$, on
$[0,\theta_{c}]$. More precisely, $\theta_{C_{0}}:=\arctan(\rho/R_{K}%
)\in(0,\pi/4]$ is well defined. Setting
\[
\theta_{\max}^{\mathrm{geo}}(K):=\min\left(  \theta_{c},\theta_{C_{0}}\right)
,
\]
every $\theta_{\max}\in(0,\theta_{\max}^{\mathrm{geo}}(K)]$ satisfies the
Assumptions $(C_{1}),(C_{2}),(C_{0})$.
\end{proposition}

\begin{proof}
By easy calculus, $g^{\prime}(\theta)=-2a_{\psi}\sin(2\theta)-16b_{\ast}%
\cos\theta<0$ on $(0,\pi/4)$. The intermediate value theorem gives $\theta
_{c},$ as stated. For $\theta_{C_{0}}$, we start from the formula of
$\delta_{0}$, given by $\left(  C_{0}\right)  $ and observe that, for
$h(\theta):=\rho\cos\theta-R_{K}\sin\theta$ one has $h^{\prime}(\theta
)=-\rho\sin\theta-R_{K}\cos\theta<0$, $h(0)=\rho>0$, $h(\pi/2)=-R_{K}\leq0$,
giving the unique root $\arctan(\rho/R_{K})\in(0,\pi/4]$ (using $\rho\leq
R_{K}$). Monotonicity gives $g\geq0$ (resp.\ $h>0$, i.e.\ $(C_{0})$) on the
stated ranges, which completes the proof.\hfill
\end{proof}

\begin{proposition}
[Data refined threshold]\label{prop:data-threshold} On the interval
$(0,\theta_{\max}^{\mathrm{geo}}(K))$, the function
\[
h_{\star}(\theta):=64\kappa_{\theta}(\theta)\,\Lambda_{\ast}^{2}%
(\theta)-\alpha_{\theta}(\theta)
\]
is continuous and non-decreasing, with $h_{\star}(0^{+})=-1<0$. Consequently, either

\begin{itemize}
\item[$\left(  a\right)  $] $h_{\star}<0$ throughout $(0,\theta_{\max
}^{\mathrm{geo}}(K))$ and $(C_{3})$ imposes no further restriction, so
$\theta_{\max}^{\star}:=\theta_{\max}^{\mathrm{geo}}(K)$;

\item[$\left(  b\right)  $] or there is a unique $\theta_{\mathrm{data}}%
\in(0,\theta_{\max}^{\mathrm{geo}}(K))$ with $h_{\star}(\theta_{\mathrm{data}%
})=0$, and $\theta_{\max}^{\star}:=\theta_{\mathrm{data}}$.
\end{itemize}

\noindent\ In either case, $\theta_{\max}^{\star}$ is the largest constant
$\theta_{\max}~,$ for which every $\theta\in C^{1}([0,T];(0,\theta_{\max}])$
satisfies $(C_{0})$-$(C_{3})$, for the given $K,F,\mathcal{R},T$ of
$\mathcal{OP}\left(  \Theta^{\theta}\right)  $.
\end{proposition}

\begin{proof}
The functions $\tan\theta$ and $1/\alpha_{\theta}=1/(1-\tan\theta)$ are
strictly increasing on $(0,\pi/4)$, so $\kappa_{\theta}$ given by
(\ref{some constants}) is strictly increasing, while $\alpha_{\theta}%
=1-\tan\theta$ is strictly decreasing. As noted, the constants $c_{6},c_{7}$,
and as consequence $\bar{\Psi}$, do not depend on $\theta$ at all. The
function $\theta$ enters $\Lambda^{2}$ solely through $\bar{\Psi}/(2a_{\theta
})$, via $a_{\theta}=\cos\theta$ decreasing (so $1/a_{\theta}$ increasing), as
$\theta$ increases. Hence $\Lambda^{2}$, and also $\Lambda_{\ast}^{2}%
=\Lambda^{2}+2j^{2}$ and $64\kappa_{\theta}\Lambda_{\ast}^{2}$, are
non-decreasing in $\theta$, while $-\alpha_{\theta}$ is strictly increasing:
$h_{\star}$ strictly is non-decreasing. The rest follows by the intermediate
value theorem.\hfill
\end{proof}

\subsection{A finite tracking cost}

Fix $y^{\mathrm{target}}\in C([0,T];K)$ and $\mu>0$. Consider the cost
functional we want to optimize:%
\begin{equation}
J(\theta):=\mathbb{E}\left[  \int_{0}^{T}\chi(t,Y_{t}^{\theta},\theta
_{t})dt\right]  ,\quad\text{with}\quad\chi(t,y,\theta):=\frac{1}{2}\left\vert
y-y^{\mathrm{target}}(t)\right\vert ^{2}+\frac{\mu}{2}\theta^{2}%
.\label{eq:cost}%
\end{equation}
We do not consider a terminal cost $\Phi(Y_{T}^{\theta})$ since $Y_{T}%
^{\theta}=\eta$, for every admissible control $\theta\in\mathcal{U}%
_{\mathrm{ad}}$, and any such term would be a constant, contributing with
nothing to the optimization. The angle penalty term $\tfrac{\mu}{2}\theta^{2}$
costs a large tangential rotation of the push away from the outward normal. A
small value for $\mu$ can allow an agile, strongly tangential correction,
while large values of $\mu$ can force a near-normal reflection.

\begin{proposition}
\label{prop:finite-value} Under Proposition~\ref{prop:wellposed},%
\begin{equation}
0\leq\inf_{\theta\in\mathcal{U}_{\mathrm{ad}}}J(\theta)\leq J_{\max}%
(\theta_{\max}):=T\left[  \frac{1}{2}\left(  c_{K}+1+\Vert y^{\mathrm{target}%
}\Vert_{C([0,T])}\right)  ^{2}+\frac{\mu}{2}\,\theta_{\max}^{2}\right]
<\infty, \label{eq:value-bounds}%
\end{equation}
a fully explicit, $\theta$-independent bound.
\end{proposition}

\begin{proof}
Obviously, $\chi\geq0$ gives the lower bound. According to
Proposition~\ref{prop:wellposed}-(b), $\left\vert Y_{t}^{\theta}\right\vert
\leq c_{K}+1$, $\mathbb{P}$-a.s.\ for every $t$ and every $\theta
\in\mathcal{U}_{\mathrm{ad}}$, so pointwise%
\[
|Y_{t}^{\theta}-y^{\mathrm{target}}(t)|^{2}\leq(c_{K}+1+\Vert
y^{\mathrm{target}}\Vert_{C([0,T])})^{2},\quad\text{and}\quad\theta_{t}%
^{2}\leq\theta_{\max}^{2}.
\]
Integrating $\chi$ over $[0,T]$ and taking expectation give the conclusion
directly, because the $\mathbb{P}$-a.s.\ boundedness is already pointwise
sharp, no Cauchy--Schwarz inequality or Jensen intermediary step are needed
anymore.\hfill\medskip
\end{proof}

\begin{remark}
The existing of the finite value$J_{\max}(\theta_{\max})$ certifies the
problem is non-degenerate, but cannot, alone, distinguish between admissible
controls because it is a robustness statement, not a synthesis. This is the
reason why, Section~\ref{sec:stability} gives genuine existence of a
minimizer, on the class $\mathcal{A}_{L}$ of angle paths with a uniformly
bounded rate of change.
\end{remark}

\subsubsection{Lipschitz stability of the control-to-state map, and existence
of an optimal control\label{sec:stability}}

Fix $L>0$ and let us define%
\[
\mathcal{A}_{L}:=\{\theta\in C^{1}([0,T];(0,\theta_{\max}^{\star}%
]):\Vert\theta^{\prime}\Vert_{C([0,T])}\leq L\},
\]
which is compact in $(C([0,T]),\Vert\cdot\Vert_{\infty})$, by
Arzel\`{a}--Ascoli. Here $\theta^{\prime}$ denotes $d\theta/dt$; the second
angle field compared against $\theta$ is denoted by $\sigma$.

\begin{theorem}
[Lipschitz stability]\label{thm:stability} Let $\theta,\sigma\in
\mathcal{A}_{L}$. Under the standing assumptions, there exists a positive
constant $C=C(\theta_{\max}^{\star},L,K,F,\mathcal{R},T)$, which is
independent of the control functions $\theta,\sigma\in\mathcal{A}_{L}$, such
that%
\begin{equation}
\mathbb{E}\sup_{t\in\lbrack0,T]}|Y_{t}^{\theta}-Y_{t}^{\sigma}|^{2}%
+\mathbb{E}[M^{\theta}-M^{\sigma}]_{T}\leq C\Vert\theta-\sigma\Vert
_{C([0,T])}^{2}.\label{eq:stability}%
\end{equation}

\end{theorem}

\begin{proof}
The proof follows the same lines as Theorem \ref{rot-thm-conv}, with some
specific changes.

Let us denote%
\[
\left\{
\begin{array}
[c]{l}%
Y:=Y^{\theta},\quad\tilde{Y}:=Y^{\sigma},\quad M:=M^{\theta},\quad\tilde
{M}:=M^{\sigma},\medskip\\
U:=U^{\theta},\quad\tilde{U}:=U^{\sigma},\quad Y_{r}^{\theta,\sigma}%
:=Y_{r}-\tilde{Y}_{r},\quad M_{r}^{\theta,\sigma}:=M_{r}-\tilde{M}_{r}.
\end{array}
\right.
\]
Highlight that the triplets $\left(  Y,M,U\right)  $, respectively $(\tilde
{Y},\tilde{M},\tilde{U})$ are exact the strong solutions of $\mathcal{OP}%
\left(  \Theta^{\theta}\right)  $ and, respectively, $\mathcal{OP}\left(
\Theta^{\sigma}\right)  $ (no Moreau--Yosida penalizations). Hence
$Y_{T}=\tilde{Y}_{T}=\eta$ and $\varphi_{T}:=\Phi_{T}(Y_{T},\tilde{Y}_{T})=0$.
Let $\Phi_{r}(x,y)$ be the two-point kernel, built using the first angle field
$\theta$ throughout, and which is asymmetric here, since the two states
genuinely reflect along different directions, caused by the controls $\theta$
and $\gamma$. Denote also $\varphi_{t}:=\Phi_{t}(Y_{t},\tilde{Y}_{t})$.

\smallskip\noindent\textit{Step 1.} Lemma \ref{rot-hess} bounds $\Phi_{r}$'s
Hessian, for arbitrary $x,y$ and directions $h,h^{\prime}$. It is a property
of $\Phi^{\theta}$ alone, independent of which $\mathcal{OP}$, $\tilde{Y}$
solves. Applied it to $(dM,d\tilde{M})$ exactly as in Step 1 of Theorem
\ref{rot-thm-conv}, and absorbing jumps via the same smallness margin as
$(C_{3})$, we obtain%
\[
-d\mathcal{I}_{r}\leq-\frac{\alpha_{\theta}}{4}d[M^{\theta,\sigma}%
]_{r}+2\kappa_{\theta}|Y_{r-}^{\theta,\sigma}|^{2}d\left(  [M]+[\tilde
{M}]\right)  _{r}.
\]

\smallskip\noindent\textit{Step 2.} Starting with the differential form of the
$\mathcal{OP}\left(  \Theta^{\theta}\right)  $, the reflection bracket is
\[
\langle\nabla_{x}\Phi_{r},\Theta_{r}^{\theta}U_{r}\rangle+\langle\nabla
_{y}\Phi_{r},\Theta_{r}^{\sigma}\tilde{U}_{r}\rangle,\quad\text{on}\quad
U_{r}\neq0\},Y_{r}\in\partial K,U_{r}=|U_{r}|\mathbf{n}(Y_{r}).
\]
Symmetrically, it can be deduced for $\tilde{U}_{r}$. Decompose, as we did
before,%
\[
\left\langle \nabla_{y}\Phi_{r},\Theta_{r}^{\sigma}\tilde{U}_{r}\right\rangle
=\underbrace{\left\langle \nabla_{y}\Phi_{r},\Theta_{r}^{\theta}\tilde{U}%
_{r}\right\rangle }_{(\mathrm{I})}+\underbrace{\left\langle \nabla_{y}\Phi
_{r},(\Theta_{r}^{\sigma}-\Theta_{r}^{\theta})\tilde{U}_{r}\right\rangle
}_{(\mathrm{II})}.
\]
Since $Y_{r},\tilde{Y}_{r}\in K$ exactly, and no as penalization defect, one
can apply Assumption $(C_{2})$ successively, by the symmetry of $\Phi$, to
both pairs $(Y_{r},\tilde{Y}_{r})$ and $(\tilde{Y}_{r},Y_{r})$. As
consequence, we obtain $\langle\nabla_{x}\Phi_{r},\Theta_{r}^{\theta}%
U_{r}\rangle\geq0$ and $(\mathrm{I})\geq0$. The mechanism of Theorem
\ref{rot-thm-uniq}'s proof does not need any correction since both states
already lie in $K$.

For the second term, $(\mathrm{II})$, $\Vert d\Theta_{r}/d\theta
\Vert_{\mathrm{op}}=\Vert J\Theta_{r}^{\theta}\Vert_{\mathrm{op}}=1$ so, we
have%
\[
\Vert\Theta_{r}^{\sigma}-\Theta_{r}^{\theta}\Vert_{\mathrm{op}}\leq|\sigma
_{r}-\theta_{r}|.
\]
Since we also have, from (\ref{r11}), $|\nabla_{y}\Phi_{r}|\leq c_{\Phi}%
|Y_{r}^{\theta,\sigma}|$, it yields, by Young's inequality,%
\[
\left\vert (\mathrm{II})\right\vert \leq c_{\Phi}|Y_{r}^{\theta,\sigma
}||\sigma_{r}-\theta_{r}||\tilde{U}_{r}|\leq\frac{\alpha_{\theta}}{16}%
|Y_{r}^{\theta,\sigma}|^{2}+\frac{4c_{\Phi}^{2}}{\alpha_{\theta}}|\sigma
_{r}-\theta_{r}|^{2}|\tilde{U}_{r}|^{2}.
\]
The first term is smaller than $\tfrac{1}{16}\varphi_{r}~$, by the coercivity
property given by (\ref{r16}): $|Y_{r}^{\theta,\sigma}|^{2}\leq\varphi
_{r}/\alpha_{\theta}$~.

\smallskip\noindent\textit{Step 3. }Using only the Lipschitz continuity of
$F,\mathcal{R}$ and the bounds $\left\vert \nabla_{x}\Phi_{r}\right\vert
,\left\vert \nabla_{y}\Phi_{r}\right\vert \leq c_{\Phi}|Y_{r}^{\theta,\sigma
}|$, $\left\vert \nabla_{x}\Phi_{r}+\nabla_{y}\Phi_{r}\right\vert \leq
c_{\Phi}|Y_{r}^{\theta,\sigma}|^{2}$, $\left\vert \partial_{r}\Phi
_{r}\right\vert \leq c_{\Phi}\Vert\theta^{\prime}\Vert_{\infty}|Y_{r}%
^{\theta,\sigma}|^{2}$ (the properties of $\Phi^{\theta},F,\mathcal{R}$ are
independent of $\varphi$) we have%
\[
\left\langle \nabla_{x}\Phi_{r},F_{r}\right\rangle +\left\langle \nabla
_{y}\Phi_{r},\tilde{F}_{r}\right\rangle -\partial_{r}\Phi_{r}\leq\frac{1}%
{2}\varphi_{r}\dot{V}_{r}^{(1)}+\frac{\gamma}{2}|\mathcal{R}_{r}%
^{\theta,\sigma}|^{2},\quad\text{with}\quad\gamma:=\frac{\alpha_{\theta}%
}{8C_{\mathcal{R}}^{2}},
\]
$\dot{V}_{r}^{(1)}:=\dfrac{2c_{\Phi}}{\alpha_{\theta}}\left(  L_{r}%
+\dfrac{c_{\Phi}\ell_{r}^{2}}{2\gamma}+\Vert\theta^{\prime}\Vert_{\infty
}+|\tilde{F}_{r}|\right)  $ and $\mathcal{R}_{r}^{\theta,\sigma}%
:=\mathcal{R}_{r}(M)-\mathcal{R}_{r}(\tilde{M})$.

\smallskip\noindent\textit{Step 4.} As verified directly, the pairing between
Step 3's bound $\tfrac{1}{2}\varphi_{r}\dot{V}_{r}^{(1)}$ and the absolutely
continuous part of $-\int\varphi_{r}d\Gamma_{r}$ (via $d\Gamma_{r}\geq
\Gamma_{r-}dV_{r}$, with $\Gamma_{t}=e^{V_{t}}$) cancels exactly, whatever the
rate $\dot{V}_{r}^{(1)}$ is. We incorporate Step 2's new term $\tfrac{1}%
{16}\varphi_{r}$ into the weight, by constructing%
\[
\dot{V}_{r}^{(1),\,\mathrm{new}}:=\dot{V}_{r}^{(1)}+\frac{1}{8}%
\]
(so $\tfrac{1}{2}\cdot\tfrac{1}{8}=\tfrac{1}{16}$, matching exactly). This
adds a bounded deterministic quantity $(\leq T/8$ over $[0,T])$ to $V_{T}$, so
it changes $\mathbb{E}\left[  \Gamma_{T}^{p}\right]  $ only with the finite
factor $e^{pT/8}$ (Step 5 remains unaffected, since the exponent
$8p\nu_{\theta}$ used there for the $[M]+[\tilde{M}]$ part of $V_{T}$ does not
involve $\dot{V}^{(1)}$). It does touch neither Step 1 (uses only $\nu
_{\theta}=4\kappa_{\theta}/\alpha_{\theta}$), nor Step 8 (uses only
$\Gamma_{T}$'s boundedness and $\Phi$'s coercivity).

With the enlarged weight, Step 2's $(\mathrm{I})$-term is absorbed exactly as
above, leaving the purely data-driven term:%
\[
d\mathcal{L}_{r}^{\mathrm{new}}:=\frac{4c_{\Phi}^{2}}{\alpha_{\theta}}%
|\sigma_{r}-\theta_{r}|^{2}|\tilde{U}_{r}|^{2}\,dr.
\]
It plays the exact role of the penalization defect: independent of $\Delta
Y_{r}$, it depends only on $\tilde{U}_{r}$ and on $\left\vert \sigma
_{r}-\theta_{r}\right\vert \leq\Vert\sigma-\theta\Vert_{\infty}$, so it is
absorbed at Step 7 by the identical H\"{o}lder argument used in Theorem
\ref{rot-thm-conv} for $d\mathcal{L}_{r}:$%
\[
\mathbb{E}\int_{0}^{T}\Gamma_{r-}d\mathcal{L}_{r}^{\mathrm{new}}\leq
\frac{4c_{\Phi}^{2}}{\alpha_{\theta}}\Vert\sigma-\theta\Vert_{\infty}^{2}%
\Vert\Gamma_{T}\Vert_{L^{p}}%
{\displaystyle\int_{0}^{T}}
|\tilde{U}_{r}|^{2}dr_{L^{p_{0}}}~,
\]
$1/p+1/p_{0}=1$, the last factor being finite, uniformly with respect to
$\sigma\in\mathcal{U}_{\mathrm{ad}}$ by Proposition \ref{prop:wellposed}-(c).
Carrying this through Steps 6-8 (Step 6 - absorption of $\mathcal{R}%
_{r}^{\theta,\sigma}$~, via $\gamma C_{\mathcal{R}}^{2}/2=\alpha_{\theta}/16$,
and Step 8 - BDG's inequality applied to the weighted local martingale)
yields
\[
\mathbb{E}\sup_{t\in\left[  0,T\right]  }|Y_{t}^{\theta,\sigma}|^{2}%
+\mathbb{E}[M^{\theta,\sigma}]_{T}\leq C\cdot\frac{4c_{\Phi}^{2}}%
{\alpha_{\theta}}\cdot\frac{\bar{\Psi}}{a_{\star}^{2}}\cdot\left\vert
\left\vert \sigma-\theta\right\vert \right\vert _{\infty}^{2}.
\]
The proof is now complete.\hfill
\end{proof}

\begin{corollary}
[Existence of an optimal control in $\mathcal{A}_{L}$]\label{cor:existence}%
Under the hypothesis of Theorem~\ref{thm:stability}, the cost functional $J$,
given by \ref{eq:cost}, is Lipschitz on $(\mathcal{A}_{L},\Vert\cdot
\Vert_{\infty})$. More precisely, there exists a positive constant $C^{\prime
}=C^{\prime}(\theta_{\max}^{\star},L,K,F,\mathcal{R},T,\mu,y^{\mathrm{target}%
})$ such that%
\[
\left\vert J(\theta)-J(\sigma)\right\vert \leq C^{\prime}\left\vert \left\vert
\theta-\sigma\right\vert \right\vert _{\infty},\qquad\theta,\sigma
\in\mathcal{A}_{L},
\]
Consequently $J$ attains its minimum on $\mathcal{A}_{L}$, i.e. an optimal
control $\theta^{\star}\in\mathcal{A}_{L}$ exists.
\end{corollary}

\begin{proof}
According to Proposition \ref{prop:wellposed}-(b) and $|a|^{2}-|b|^{2}=\langle
a-b,a+b\rangle$,%
\[
\left\vert \chi(t,Y_{t}^{\theta},\theta_{t})-\chi(t,Y_{t}^{\sigma},\sigma
_{t})\right\vert \leq\left(  c_{K}+1+\left\vert \left\vert y^{\mathrm{target}%
}\right\vert \right\vert _{\infty}\right)  \,\left\vert Y_{t}^{\theta}%
-Y_{t}^{\sigma}\right\vert +\mu\theta_{\max}^{\star}\,|\theta_{t}-\sigma_{t}|.
\]
Integrating, taking expectation and applying Cauchy--Schwarz inequality in
$t$, we obtain, due to (\ref{eq:stability}),%
\[
\mathbb{E}\int_{0}^{T}\left\vert Y_{t}^{\theta}-Y_{t}^{\sigma}\right\vert
dt\leq T\left(  \sup\nolimits_{t\in\left[  0,T\right]  }\mathbb{E}%
|Y_{t}^{\theta}-Y_{t}^{\sigma}|^{2}\right)  ^{1/2}\leq T\sqrt{C}\left\vert
\left\vert \theta-\sigma\right\vert \right\vert _{\infty}~.
\]
Therefore, $\left\vert J(\theta)-J(\sigma)\right\vert \leq\left[
T(c_{K}+1+\Vert y^{\mathrm{target}}\Vert_{\infty})\sqrt{C}+\mu\theta_{\max
}^{\star}T\right]  \Vert\theta-\varphi\Vert_{\infty}=:C^{\prime}\Vert
\theta-\varphi\Vert_{\infty}$.

Finally, since $\mathcal{A}_{L}$ is compact in $(C([0,T]),\Vert\cdot
\Vert_{\infty})$ and $J$ is Lipschitz on $\mathcal{A}_{L}$, it infers that $J$
attains its minimum and the proof is now complete.\hfill
\end{proof}

\begin{problem}
[Characterization of an optimal control]Establish a Pontryagin-type necessary
condition or a verification-type sufficient condition characterizing a
minimizer $\theta^{\star}\in\mathcal{U}_{\mathrm{ad}}$ on the full class of
adapted, predictable angle processes (beyond the deterministic, bounded-rate
of change class $\mathcal{A}_{L}$ of Corollary~\ref{cor:existence}). To our
best knowledge, the closest established result in the literature is the one
provided by Huang, Wang and Wu \cite{Huang/Wang/Wu:18}. Their sufficient
stochastic maximum principle deals with recursive control problems with an
obstacle as constraint. It treats a scalar barrier reflection ($Y_{t}\geq
L_{t}$), rather than c\`{a}dl\`{a}g oblique reflection in a two-dimensional
convex domain, with a controlled rotation angle. A rigorous treatment of this
characterization would require a separate control theoretic analysis and is
beyond the scope of the present paper.
\end{problem}

\section{Annex}

\subsection{Tools for the c\`{a}dl\`{a}g calculus\label{cadlag tools}}

We recall here only the arguments and tools we shall use in our c\`{a}dl\`{a}g
study. For a process $X$, denote $X_{t-}:=\lim_{s\uparrow t}X_{s}.$ According
to Protter \cite[Chapter II, Th. 21, page 64]{Protter:05}, if $\left\{
X_{t}:t\geq0\right\}  $ is an $\mathbb{R}^{d}$-valued adapted c\`{a}dl\`{a}g
semimartingale, the stochastic integral with respect to $X$, $I_{X}%
:\mathbb{L}_{m\times d}^{0}\rightarrow\mathbb{D}_{d}^{0}$ is a linear
continuous mapping.

Let $X\in\mathbb{D}_{d}^{0}$ and $Y\in\mathbb{L}_{m\times d}^{0}.$ Each of the
following properties of $X$ is inherited by $I_{X}\left(  Y\right)  $: is a
semimartingale (Protter \cite[Chapter II, Th.19, page 62]{Protter:05}), is a
bounded variation stochastic process (Protter \cite[Chapter II, Th.17, page
61]{Protter:05}), is a locally square integrable local martingale (Protter
\cite[Chapter II, Th.20, page 63]{Protter:05}) and it is a local martingale
(Protter \cite[Chapter III, Th.29, page128]{Protter:05}).

Let us introduce the \textit{Burkholder--Davis--Gundy inequality }suited to
our working setup: for any $p\in\lbrack1,\infty)$ there exist two constants
$c_{p},C_{p}>0,$ depending only on $p,$ such that, for all local martingales
$X,$ with $X_{0}=0,$ and any stopping time $\tau$, the following inequality
holds%
\begin{equation}
c_{p}\mathbb{E}\left[  X\right]  _{\tau}^{p/2}\leq\mathbb{E}\sup_{0\leq
t\leq\tau}\left\vert X_{t}\right\vert ^{p}\leq C_{p}\mathbb{E}\left[
X\right]  _{\tau}^{p/2}. \label{BDG}%
\end{equation}
If $X$ is a continuous local martingale, then the inequality (\ref{BDG}) holds
for all $0<p<\infty$. In the case of the stochastic integral defined for
$X\in\mathbb{D}_{d}^{0}$ (being a local martingale) and $Y\in\mathbb{L}%
_{m\times d}^{0}~$, we have, for any $1\leq p<\infty$,%
\begin{equation}
\mathbb{E}\sup_{0\leq t\leq\tau}\left\vert \int_{0+}^{t}Y_{s}dX_{s}\right\vert
^{p}\leq C_{p}\mathbb{E}\left(  \int_{0+}^{\tau}\left\vert Y_{s}\right\vert
^{2}d\left[  X\right]  _{s}\right)  ^{p/2}. \label{BDG1}%
\end{equation}

It\^{o}'s formula for semimartingales becomes:

\begin{lemma}
If $Y$ is a $d$-dimensional c\`{a}dl\`{a}g semimartingale and $u\in
C^{1,2}(\left[  0,T\right]  \times\mathbb{R}^{d};\mathbb{R})$, then $u\left(
\cdot,Y_{\cdot}\right)  $ is a semimartingale and the following formula
holds:\vspace{-0.1in}%
\[%
\begin{array}
[c]{l}%
u\left(  t,Y_{t}\right) \\
\quad=u\left(  s,Y_{s}\right)  +%
{\displaystyle\int_{(s,t]}}
\dfrac{\partial u\left(  r,Y_{r}\right)  }{\partial t}dr+%
{\displaystyle\sum_{i=1}^{d}}
{\displaystyle\int_{(s,t]}}
\dfrac{\partial u\left(  r,Y_{r-}\right)  }{\partial x_{i}}dY_{r}^{i}%
+\dfrac{1}{2}%
{\displaystyle\sum_{1\leq i,j\leq d}}
{\displaystyle\int_{(s,t]}}
\dfrac{\partial^{2}u\left(  r,Y_{r-}\right)  }{\partial x_{i}\partial x_{j}%
}d[Y^{i},Y^{j}]_{r}\smallskip\\
\quad+%
{\displaystyle\sum_{s<r\leq t}}
\left\{  u(r,Y_{r})-u(r,Y_{r-})-%
{\displaystyle\sum_{i=1}^{d}}
\dfrac{\partial u\left(  r,Y_{r-}\right)  }{\partial x_{i}}\Delta Y_{r}%
^{i}-\dfrac{1}{2}%
{\displaystyle\sum_{1\leq i,j\leq d}}
\dfrac{\partial^{2}u\left(  r,Y_{r-}\right)  }{\partial x_{i}\partial x_{j}%
}\Delta Y_{r}^{i}\Delta Y_{r}^{j}\right\}  .
\end{array}
\]
for all $0\leq s\leq t\leq T,$ $\mathbb{P}$-$a.s.$ In particular, the energy
equality reads%
\begin{equation}
\left\vert Y_{t}\right\vert ^{2}=\left\vert Y_{s}\right\vert ^{2}+2%
{\displaystyle\int_{(s,t]}}
\left\langle Y_{r-},dY_{r}\right\rangle +[Y,Y]_{t}-[Y,Y]_{s}~. \label{ee}%
\end{equation}

\end{lemma}

\noindent In the formula above, $Y_{r}=Y_{r-}$ for $dr$-a.e. $r$, so the
integrals with respect to $dr$ may equivalently be taken on $\left[
s,t\right]  $.

Arguing as in Pardoux and R\u{a}\c{s}canu \cite[Lemma 2.37]%
{Pardoux/Rascanu:14}, we obtain that, if $Y$ is a $d$-dimensional
c\`{a}dl\`{a}g semimartingale and $\psi\in C^{1}(\mathbb{R}^{d};\mathbb{R})$
is convex, then the following c\`{a}dl\`{a}g stochastic subdifferential
inequality holds, for every $0\leq t<s\leq T$:%
\begin{equation}%
{\displaystyle\int_{t+}^{s}}
\left\langle \nabla_{y}\psi\left(  Y_{r-}\right)  ,dY_{r}\right\rangle
\leq\psi\left(  Y_{s}\right)  -\psi\left(  Y_{t}\right)  .
\label{cadlag subdiff ineq}%
\end{equation}
A useful \textit{backward Gronwall's inequality} is stated below.

\begin{lemma}
\label{Gronwall}\textit{Let }$\Theta,K,V:\left[  0,T\right]  \rightarrow
\mathbb{R}$ \textit{be bounded c\`{a}dl\`{a}g functions such that }$K_{\cdot
}\in BV\left(  \left[  0,T\right]  ;\mathbb{R}\right)  $ and $V$ is
\emph{continuous} and nondecreasing. \textit{If, for all }$0\leq s\leq T,$%
\begin{equation}
\Theta_{s}\leq\Theta_{T}+%
{\displaystyle\int_{s+}^{T}}
\left[  dK_{r}+\Theta_{r}dV_{r}\right]  \label{g1}%
\end{equation}
\textit{then}%
\begin{equation}
\left\{
\begin{array}
[c]{ll}%
\left(  i\right)  & e^{V_{s}}\Theta_{s}\leq\Theta_{T}e^{V_{T}}+{%
{\displaystyle\int_{s+}^{T}}
}e^{V_{r}}dK_{r}\quad\text{and}\medskip\\
\left(  ii\right)  & \Theta_{s}\leq e^{V_{T}-V_{s}}\left[  \Theta
_{T}+\left\Vert K\right\Vert _{BV\left(  \left[  0,T\right]  ;\mathbb{R}%
\right)  }\right]  .
\end{array}
\right.  \label{g3}%
\end{equation}
\textit{If, moreover, }$a\geq0$\textit{ is a constant and, for all }$0\leq
s\leq T,$ $\Theta_{s}\leq a+%
{\textstyle\int_{s+}^{T}}
\Theta_{r}dV_{r}~$, \textit{then}%
\begin{equation}
\Theta_{s}\leq a\,e^{V_{T}-V_{s}}~. \label{g4}%
\end{equation}

\end{lemma}

\subsection{Regularization of convex functions}

We now present some classical instruments used when we need to approximate a
maximal monotone operator given by the subdifferential of a proper convex
lower semicontinuous function. For more details, the interested reader can
consult Pardoux and R\u{a}\c{s}canu \cite[Section 6.3.7]{Pardoux/Rascanu:14}.

\begin{lemma}
\label{conv}Let $\varphi:\mathbb{R}^{d}\rightarrow(-\infty,+\infty]$ be a
proper lower semicontinuous convex function and let $\varphi_{\varepsilon}$ be
its Moreau regularization,
\[
\varphi_{\varepsilon}(x):=\inf\left\{  \frac{1}{2\varepsilon}|z-x|^{2}%
+\varphi(z):z\in\mathbb{R}^{d}\right\}  ,
\]
where $\varepsilon>0.$ Then for all $x,y\in\mathbb{R}^{d}$ and every
$\varepsilon,\delta>0$, the following are true:

\begin{itemize}
\item[$\left(  a\right)  $] $\varphi_{\varepsilon}$ is a convex function of
class $C^{1}$ on $\mathbb{R}^{d}$, and $\nabla\varphi_{\varepsilon}$ is
Lipschitz on $\mathbb{R}^{d}$ with Lipschitz constant $\varepsilon^{-1};$

\item[$\left(  b\right)  $] $\nabla\varphi_{\varepsilon}\left(  x\right)
=\dfrac{1}{\varepsilon}(x-J_{\varepsilon}x)$, where $J_{\varepsilon
}x=(I+\varepsilon\partial\varphi)^{-1}(x);$

\item[$\left(  c\right)  $] $\varphi_{\varepsilon}(x)=\dfrac{1}{2\varepsilon
}|x-J_{\varepsilon}x|^{2}+\varphi(J_{\varepsilon}x);$

\item[$\left(  d\right)  $] $\nabla\varphi_{\varepsilon}(x)\in\partial
\varphi(J_{\varepsilon}x);$

\item[$\left(  e\right)  $] $\left\langle \nabla\varphi_{\varepsilon
}(x)-\nabla\varphi_{\delta}(y),x-y\right\rangle \geq-\left(  \varepsilon
+\delta\right)  |\nabla\varphi_{\varepsilon}(x)||\nabla\varphi_{\delta}(y)|;$

\item[$\left(  f\right)  $] if $\left(  u_{0},\hat{u}_{0}\right)  \in
\partial\varphi$ then

\begin{itemize}
\item[$\left(  f_{1}\right)  $] $\left\vert \nabla\varphi_{\varepsilon}\left(
u_{0}\right)  \right\vert \leq\left\vert \hat{u}_{0}\right\vert ,$

\item[$\left(  f_{2}\right)  $] $\left\vert J_{\varepsilon}\left(  y\right)
\right\vert \leq\left\vert y-u_{0}\right\vert +\varepsilon\left\vert \hat
{u}_{0}\right\vert +\left\vert u_{0}\right\vert $ for all $y\in\mathbb{R}%
^{d},$

\item[$\left(  f_{3}\right)  $] $\varphi_{\varepsilon}\left(  y\right)
\geq\varphi\left(  J_{\varepsilon}y\right)  \geq\varphi\left(  u_{0}\right)
-\left\vert \hat{u}_{0}\right\vert \left\vert y-u_{0}\right\vert
-\varepsilon\left\vert \hat{u}_{0}\right\vert ^{2}$ for all $y\in
\mathbb{R}^{d}$,

\item[$\left(  f_{4}\right)  $] $\left\vert \varphi\left(  J_{\varepsilon
}y\right)  -\varphi\left(  u_{0}\right)  \right\vert \leq\left\langle
\nabla\varphi_{\varepsilon}\left(  y\right)  ,J_{\varepsilon}y-u_{0}%
\right\rangle +2\left\vert \hat{u}_{0}\right\vert \left\vert y-u_{0}%
\right\vert +2\varepsilon\left\vert \hat{u}_{0}\right\vert ^{2}$ for all
$y\in\mathbb{R}^{d}$.
\end{itemize}
\end{itemize}
\end{lemma}

\noindent Funding: Not applicable.\bigskip

\begin{acknowledgement}
The authors would like to express their sincere gratitude to the anonymous
referees for their comments and suggestions, which have resulted in
considerable improvement of the results and presentation of this article.
\end{acknowledgement}


\bigskip

\begin{thebibliography}{99}                                                                                               %


\bibitem {Attal/Belton:07}Attal, S.; Belton, A.C.R., \textsl{The
chaotic-representation property for a class of normal martingales}, Probab.
Theory Related Fields, Volume 139, pp. 543--562, 2007.\vspace{-0.1in}

\bibitem {Attal/Emery:94}Attal, S.; \'{E}mery, M., \textsl{\'{E}quations de
structure pour des martingales vectorielles}, S\'{e}minaire de
Probabilit\'{e}s XXVIII, Lecture Notes in Math. 1583, pp. 256--278, Springer,
1994.\vspace{-0.1in}

\bibitem {Barles:99}Barles, G., \textsl{Nonlinear Neumann boundary conditions
for quasilinear degenerate elliptic equations and applications}, J.
Differential Equations \textbf{154} (1999), 191--224.\vspace{-0.1in}

\bibitem {Barles/DaLio:06}Barles, G.; Da Lio, F., \textsl{Local $C^{0,\alpha}$
estimates for viscosity solutions of Neumann-type boundary value problems}, J.
Differential Equations \textbf{225} (2006), 202--241.\vspace{-0.1in}

\bibitem {Bensoussan/Li/Yam:2018}Bensoussan, A.; Li, Y.; Yam, S.,
\textsl{Backward stochastic dynamics with a subdifferential operator and
non-local parabolic variational inequalities}, Stochastic Processes and their
Applications, Volume 128, Issue 2, pp. 644--688, 2018.\vspace{-0.1in}

\bibitem {Brezis:73}Br\'{e}zis, H., \textsl{Op\'{e}rateurs Maximaux Monotones
et Semi-Groupes de Contractions dans les Espaces de Hilbert}, North-Holland,
Amsterdam, 1973.\vspace{-0.1in}

\bibitem {Burdzy/Chen/Marshall/Ramanan:17}Burdzy, K.; Chen, Z.-Q.; Marshall,
D.; Ramanan, K., \textsl{Obliquely reflected Brownian motion in nonsmooth
planar domains}, Ann. Probab. \textbf{45} (2017), no.~5.\vspace{-0.1in}

\bibitem {Chassagneux/Nadtochiy/Richou:22}Chassagneux, J.-F.; Nadtochiy, S.;
Richou, A., \textsl{Reflected BSDEs in non-convex domains}, Probab. Theory
Relat. Fields \textbf{183} (2022), 1237--1284.\vspace{-0.1in}

\bibitem {Chassagneux/Richou:20}Chassagneux, J.-F.; Richou, A.,
\textsl{Obliquely reflected backward stochastic differential equations}, Ann.
Inst. Henri Poincar\'{e} Probab. Stat. \textbf{56} (2020), 2868--2896.\vspace
{-0.1in}

\bibitem {Dellacherie/Meyer:80}Dellacherie, C.; Meyer, P.-A.,
\textsl{Probabilit\'{e}s et potentiel, Chapitres V \`{a} VIII: Th\'{e}orie des
martingales}, Hermann, Paris, 1980.\vspace{-0.1in}

\bibitem {Dupuis/Ishii:93}Dupuis, P.; Ishii, H., \textsl{SDEs with oblique
reflection on nonsmooth domains}, Ann. Probab. \textbf{21} (1993), 554--580;
correction, Ann. Probab. \textbf{36} (2008), 1992--1997.\vspace{-0.1in}

\bibitem {Emery:89}\'{E}mery, M., \textsl{On the Az\'{e}ma martingales},
S\'{e}minaire de Probabilit\'{e}s XXIII, Lecture Notes in Math. 1372, pp.
66--87, Springer, 1989.\vspace{-0.1in}

\bibitem {Gassous/Rascanu/Rotenstein:12}Gassous, A.; R\u{a}\c{s}canu, A.;
Rotenstein, E., \textsl{Stochastic variational inequalities with oblique
subgradients}, Stochastic Process. Appl., Volume 122, Issue 7, pp. 2668--2700,
2012.\vspace{-0.1in}

\bibitem {Gassous/Rascanu/Rotenstein:15}Gassous, A.; R\u{a}\c{s}canu, A.;
Rotenstein, E., \textsl{Multivalued BSDEs with oblique subgradients}, Stoch.
Process. Appl., Volume 125, Issue 8 (August), pp. 3170--3195, 2015.\vspace
{-0.1in}

\bibitem {Hu/Tang:10}Hu, Y.; Tang, S., \textsl{Multi-dimensional BSDE with
oblique reflection and optimal switching}, Probab. Theory and Related Fields,
Volume 147, Issue 1-2, pp. 89-121, 2010.\vspace{-0.1in}

\bibitem {Huang/Wang/Wu:18}Huang, J.; Wang, H.; Wu, Z., \textsl{A sufficient
stochastic maximum principle for a kind of recursive optimal control problem
with obstacle constraint}, Systems \& Control Letters 114 (2018),
27--30.\vspace{-0.1in}

\bibitem {Kazamaki:94}Kazamaki, N., \textsl{Continuous Exponential Martingales
and BMO}, Lecture Notes in Mathematics \textbf{1579}, Springer, Berlin,
1994.\vspace{-0.1in}

\bibitem {Liang/Lyons/Qian:11}Liang, G; Lyons, T.; Qian, Z., \textsl{Backward
stochastic dynamics on a filtered probability space}, The Annals of
Probability, Vol. 39, No. 4, pp. 1422--1448, 2011.\vspace{-0.1in}

\bibitem {Lions/Sznitman:84}Lions, P.-L.; Sznitman, A.-S., \textsl{Stochastic
differential equations with reflecting boundary conditions}, Comm. Pure Appl.
Math. \textbf{37} (1984), 511--537.\vspace{-0.1in}

\bibitem {Maticiuc/Rotenstein:18}Maticiuc, L.; Rotenstein, E.,
\textsl{Anticipated backward stochastic variational inequalities with
generalized reflection}, Stoch. Dyn., Vol. 18, No. 2, article ID: 1850008,
pages: 1-21, 2018.\vspace{-0.1in}

\bibitem {Negrut/Rascanu/Rotenstein:26}Negru\c{t}, A.; R\u{a}\c{s}canu, A.;
Rotenstein, E., \textsl{C\`{a}dl\`{a}g Solutions to Backward Stochastic
Dynamics featuring Oblique Subgradients and driven by Martingale Noise},
preprint, 2026.\vspace{-0.1in}

\bibitem {Pardoux/Rascanu:98}Pardoux, E.; R\u{a}\c{s}canu, A.,
\textsl{Backward stochastic differential equations with subdifferential
operator and related variational inequalities}, Stochastic Processes and their
Applications, 76(2), pp. 191-215, 1998.\vspace{-0.1in}

\bibitem {Pardoux/Rascanu:99}Pardoux, E.; R\u{a}\c{s}canu, A.,\textit{\ }%
\textsl{Backward stochastic variational inequalities}, Stochastics Stochastics
Rep., 67(3-4), pp. 159-167, 1999.\vspace{-0.1in}

\bibitem {Pardoux/Rascanu:14}Pardoux, E.; R\u{a}\c{s}canu, A.,
\textsl{Stochastic differential equations, Backward SDEs, Partial differential
equations}, Stochastic Modelling and Applied Probability, Vol. 69, XVII,
Springer, 2014.\vspace{-0.1in}

\bibitem {Protter:05}Protter, P., \textsl{Stochastic Integration and
Differential Equations}, Stochastic Modelling and Applied Probability, Vol.
21, Springer-Verlag Berlin Heidelberg, 2005.\vspace{-0.1in}

\bibitem {Ramasubramanian:02}Ramasubramanian, S., \textsl{Reflected backward
stochastic differential equations in an orthant}, Proc. Indian Acad. Sci.
Math. Sci. \textbf{112} (2002), 347--360.\vspace{-0.1in}

\bibitem {Ren:08}Ren, Y.-F., \textsl{On the Burkholder Davis Gundy
inequalities for continuous martingales}, Statistics and Probability Letters
78 pp. 3034-3039, 2008.\vspace{-0.1in}
\end{thebibliography}
\end{document}